\documentclass[a4paper, 10pt, oneside]{amsart}
\usepackage{times,latexsym,amssymb}
\usepackage{amsmath}
\usepackage{amsthm}

\usepackage{typearea}
\areaset{13,4cm}{22,5cm}

\allowdisplaybreaks
\newtheorem{theorem}{Theorem}
\newtheorem*{theorem-plain}{Theorem}
\newtheorem{lemma}[theorem]{Lemma}
\newtheorem{definition}[theorem]{Definition}
\newtheorem{proposition}[theorem]{Proposition}

\newtheorem{remark}[theorem]{Remark}

\numberwithin{theorem}{section}
\numberwithin{equation}{section}

\newcommand{\mint}{- \mskip-19,5mu \int}
\newcommand{\tmint}{- \mskip-16,5mu \int}

\def\N{\mathbb{N}}
\def\R{\mathbb{R}}

\def\K{\mathcal K}
\def\dz{\,dz}
\def\dx{\,dx}

\def\dt{\,dt}
\def\ds{\,ds}
\def\dtau{\,d\tau}

\def\half{\textstyle\frac{1}{2}\displaystyle}

\def\<{\langle}
\def\>{\rangle}

\def\nn{\nonumber}

\DeclareMathOperator{\Div}{div}
\DeclareMathOperator{\diam}{diam}

\renewcommand{\epsilon}{\varepsilon}
\renewcommand{\rho}{\varrho}

\begin{document}
\title[Global gradient bound]{Global gradient bounds for the parabolic\\ $p$-Laplacian system}
\date{\today}
\author[V. B\"ogelein]{Verena B\"{o}gelein${}^\dagger$}
\address{Verena B\"ogelein\\Department Mathematik, Universit\"at
Erlangen--N\"urnberg\\ Cauerstra\ss e 11, 91058 Erlangen, Germany}
\email{boegelein@math.fau.de}
\thanks{${}^\dagger$Phone: +49 9131 85-67098, Fax: +49 9131 85-67100, E-mail address: boegelein@math.fau.de}


\begin{abstract} 
A by now classical result due to DiBenedetto states that the spatial gradient of solutions to the parabolic $p$-Laplacian system is locally H\"older continuous in the interior.
However, the boundary regularity is not yet well understood.
In this paper we prove a boundary $L^\infty$-estimate for the spatial gradient $Du$ of solutions to the parabolic $p$-Laplacian system
\begin{equation*}
	\partial_t u - \Div \big(|Du|^{p-2}Du\big) = 0
	\quad\mbox{in $\Omega\times(0,T)$}
\end{equation*}
for $p\ge 2$, together with a quantitative estimate. In particular, this implies the global Lipschitz regularity of solutions.
The result continues to hold for the so called asymptotically regular parabolic systems.
\end{abstract}

\maketitle

\section{Introduction}

The question concerning the regularity of solutions to partial differential equations of $p$-Laplacian type was a longstanding open problem and still there are unsolved questions in this field. The first breakthrough was achieved by Ural'tseva \cite{Uraltseva:1968} who proved that solutions of the elliptic $p$-Laplacian equation 
\begin{equation}\label{el-plap}
	\Div \big(|Du|^{p-2}Du\big) = 0
\end{equation}
are of class $C^{1,\alpha}$ in the interior of the domain. The analogous result for elliptic $p$-Laplacian systems -- which cannot be treated by the techniques used by Ural'tseva for equations -- was achieved  ten years later in the famous paper of Uhlenbeck \cite{Uhlenbeck:1977}. 
In turn, the elliptic techniques did not apply to treat the evolutionary counterpart, the parabolic $p$-Laplacian system
\begin{equation}\label{para-plap}
	\partial_t u - \Div \big(|Du|^{p-2}Du\big) = 0.
\end{equation}
It turned out that the inhomogeneity of the system, i.e. the fact that the scaling with respect to space and time is not homogeneous and therefore $const\cdot u$ is in general not anymore a solution, is a basic obstruction to deduce homogeneous estimates which are unavoidable in any regularity proof. The brilliant idea to use a certain intrinsic geometry which reflects the inhomogeneity of the parabolic system was  invented by 
DiBenedetto \& Friedman \cite{DiBenedetto-Friedman:1984, DiBenedetto-Friedman:1985} who proved H\"older continuity of the spatial gradient of the solution in the interior of the domain.
For the $C^{1,\alpha}$-estimate we also refer to Wiegner \cite{Wiegner:1986}. 
In this setting everywhere regularity cannot be expected.
The crucial idea of DiBenedetto \& Friedman to deal with the parabolic case was to introduce a system of parabolic cylinders different from the standard ones and whose space-time scaling depends on the local behavior of the solution itself. In a certain sense this re-balances the non-homogeneous scaling of the parabolic $p$-Laplacian system. The strategy is to find so called \textit{intrinsic parabolic cylinders} of the form 
\begin{equation*}
	Q_{\rho,\lambda}(z_o)
	:=
	B_\rho(x_o)\times\big(t_o-\lambda^{2-p}\rho^2, t_o+\lambda^{2-p}\rho^2\big),
	\quad z_o=(x_o,t_o)
\end{equation*} 
in such a way that the scaling parameter $\lambda>0$ and the average of $|Du|^p$ over the cylinder are coupled by a condition of the type
\begin{equation}\label{intrinsic}
	\mint_{Q_{\rho,\lambda}(z_o)} |Du|^p \dz
	\approx
	\lambda^p.
\end{equation} 
The delicate aspect within this coupling clearly relies in the fact that the value of the integral average must be comparable to the scaling factor $\lambda$ which itself is involved in the construction of its support. On such intrinsic cylinders the parabolic $p$-Laplacian system \eqref{para-plap} behaves in a certain sense like $\partial_t u - \lambda^{p-2}\Delta u$. Therefore, using cylinders of the type $Q_{\rho,\lambda}(z_o)$ allows to re-balance the occurring multiplicative factor $\lambda^{p-2}$ by re-scaling $u$ in time with a factor $\lambda^{2-p}$. For an application of the technic of intrinsic scaling in the context of higher integrability we refer to Kinnunen \& Lewis \cite{Kinnunen-Lewis:2000, Kinnunen-Lewis:2002}.

With respect to the boundary regularity the situation is quite different. In the elliptic as well as in the parabolic case the boundary regularity is well understood only for the equations \eqref{el-plap} and \eqref{para-plap}. In the case of equations the Lipschitz regularity of solutions to \eqref{para-plap} up to the boundary has been established by DiBenedetto \& Manfredi \& Vespri \cite{DiBenedetto-Manfredi-Vespri:1992} and the boundary $C^{1,\alpha}$ regularity is due to Lieberman \cite{Lieberman:1988, Lieberman:1990}. In the parabolic setting by $C^{1,\alpha}$-regularity we understand H\"older-continuity 
of the spatial gradient $Du$ with respect to space and time. Unfortunately the known proofs of these results use tools like maximum principles that are available only for equations and therefore these techniques cannot apply in the case of systems.
For the associated systems the  $C^{1,\alpha}$ regularity is only known for homogeneous boundary data, i.e. $u\equiv 0$ on the lateral boundary \cite{Chen-DiBenedetto:1989}. 

On the contrary, we are  interested in general boundary data, which cannot be treated like homogeneous boundary data, since reflection arguments are not anymore available.
In this case, i.e. general data at the boundary, it has been shown by DiBenedetto \& Chen \cite{Chen-DiBenedetto:1989} that  the solution -- not the gradient -- is globally H\"older continuous with respect to the parabolic metric for any H\"older exponent $\alpha\in(0,1)$. 
This means that the H\"older exponent with respect to the spatial direction is $\alpha$ and the one with respect to the time direction is $\alpha/2$. 
The case $\alpha=1$ cannot be achieved by this method of proof.
With this respect, the case $\alpha=1$, i.e. Lipschitz continuity with respect to the parabolic metric and also higher regularity remained an open problem.

Very recently it has been proved in the elliptic case by Foss \cite{Foss:2008} that solutions of \eqref{el-plap} are Lipschitz continuous up to the boundary. 
As before, these techniques cannot be transferred to the parabolic setting because of the non-homogeneous scaling behavior of the problem. The main goal in this paper is to prove global boundedness of the spatial gradient of solutions to the parabolic $p$-Laplacian system and in turn to obtain the global Lipschitz continuity.
The basic difference with respect to the known results is, that this result provides a first boundary regularity result for the gradient of the solution, i.e. boundedness of the spatial gradient. 

In order not to overburden the exposition we restrict ourselves to the more interesting lateral boundary regularity. 
The result -- in a simplified version -- reads as follows:

\begin{theorem}
Let $p\ge 2$ and $\Omega$ be a bounded smooth domain in $\R^n$ and $T>0$ and suppose that 
$$
	u \in C^0\big([0,T];L^2(\Omega,\R^{N})\big)\cap L^p\big(0,T;W^{1,p}(\Omega,\R^{N})\big)
$$ 
is a weak solution to the parabolic Cauchy-Dirichlet problem 
\begin{equation}\label{system-intro}
\left\{
\begin{array}{cc}
    \partial_t u - \Div \big(|Du|^{p-2}Du\big) = 0
    &\qquad\mbox{in $\Omega\times(0,T)$} \\[5pt]
    u=g
    &\qquad\mbox{on $\partial\Omega\times(0,T)\cup \overline\Omega\times\{0\}$}
\end{array}
\right.
\end{equation}
with smooth boundary data $g$. 
Then, $Du$ is bounded and $u$ is Lipschitz-continuous with respect to the parabolic metric up to the lateral boundary, i.e. there holds
$$
	Du\in L^\infty\big(\Omega\times(\epsilon,T),\R^{Nn}\big)
	\quad\mbox{and}\quad
	u\in C^{0;1,1/2}\big(\Omega\times(\epsilon,T),\R^{N}\big)
	\quad\forall\, \epsilon\in(0,T).
$$
\end{theorem}

Of course, the assumptions on the boundary data can be weakened, see Theorem \ref{thm:main-lip} below. But we emphasize that the result is new even for smooth boundary data.
Moreover, we also obtain a point-wise estimate for the spatial gradient, see again Theorem \ref{thm:main-lip} below.

The result of the preceding Theorem continues to hold for a much larger class of degenerate parabolic systems, the so called asymptotically regular systems. By this we mean parabolic systems of the type 
\begin{equation*}
	\partial_t u - \Div a(Du) ,
\end{equation*}
where $a\colon\R^{Nn}\to\R^{Nn}$ is a $C^1$ vector field which behaves asymptotically like the $p$-Laplacian in the sense that
\begin{equation*}
    \lim_{|\xi|\to\infty} \frac{D a(\xi) - D b(\xi)}{|\xi|^{p-2}}
    =
    0,
    \qquad\mbox{with $b(\xi):=\K|\xi|^{p-2}\xi$}
\end{equation*}
holds for some $\K>0$. 
The crucial point here is that apart from the fact that $a$ has to be of class $C^1$ we neither impose a growth assumption for ``small values of $\xi$'', nor do we assume that $a$ has a quasi-diagonal structure which is usually necessary to obtain everywhere regularity results in the case of systems. 
Nevertheless, since $a$ tends to the regular vector field $b$ when the gradient of the solution becomes large it is reasonable to obtain gradient estimates for this kind of problems. 

In the elliptic framework such a result was first obtained by Chipot \& Evans \cite{Chipot-Evans:1986}. They proved Lipschitz regularity of minimizers to integral functionals $F(v)=\int_\Omega f(Dv)\dx$ with an integrand satisfying $D^2f(\xi)\to A$ when $|\xi|\to\infty$ for some elliptic bilinear form $A$ on $\R^{Nn}$. More general integrands were treated later by 
Giaquinta \& Modica \cite{Giaquinta-Modica:1986} and Raymond \cite{Raymond:1991} and the case of higher order functionals has been considered by Schemm \cite{Schemm:2009}.
These results provide a huge class of elliptic systems, respectively integral functionals with Lipschitz solutions, respectively minimizers, which is much larger than the well known class of quasi-diagonal structure.
Moreover, a Calder\'on \& Zygmund theory and partial Lipschitz regularity for a very general class of asymptotically regular elliptic systems and integral functionals has been developed by Scheven \& Schmidt \cite{Scheven-Schmidt:2009, Scheven-Schmidt:2010}.
Finally, global Morrey and Lipschitz regularity results have been obtained by 
Foss \cite{Foss:2008} and Foss \& Passarelli di Napoli \& Verde \cite{Foss-Passarelli-Verde:2008}.
To our knowledge asymptotically regular parabolic problems have not yet been studied.

The aim of this paper now is twofold. The first and most important one, is to prove the global Lipschitz regularity for the parabolic $p$-Laplacian system. Our second aim is to start the investigation of asymptotically regular parabolic systems and thereby deduce the global and local Lipschitz regularity result.

\section{Statement of the result}

We let $n\ge 2$, $N\ge 1$ and fix a growth exponent $p\ge 2$. For a differentiable vector field $a\colon\R^{Nn}\to\R^{Nn}$ we consider the 
parabolic Cauchy-Dirichlet problem 
\begin{equation}\label{system}
\left\{
\begin{array}{cc}
    \partial_t u - \Div a(Du) = 0
    &\qquad\mbox{in $\Omega_T$} \\[5pt]
    u=g
    &\qquad\mbox{on $\partial_{\mathcal P}\Omega_T$}
\end{array}
\right.
\end{equation}
in a cylindrical domain $\Omega_T:=\Omega\times(0,T)$, where $\Omega$ is a bounded domain in $\R^n$ and $T>0$. The parabolic boundary of $\Omega_T$ is given by
\begin{equation*}
	\partial_{\mathcal P}\Omega_T
	:=
	\partial\Omega\times(0,T)\cup \overline\Omega\times\{0\}.
\end{equation*}
At this point we emphasize that the solution $u\colon \Omega_T\to\R^N$ is allowed to be vector valued and refer to Definition \ref{def:weak-solution} below for the precise notion of a weak solution.
The only assumption we put on the vector field $a$ is that it is asymptotically of first order related to the $p$-Laplacian vector field in the sense that 
\begin{equation}\label{def_asymp}
    \lim_{|\xi|\to\infty} \frac{D a(\xi) - D b(\xi)}{|\xi|^{p-2}}
    =
    0,
    \qquad\mbox{where $b(\xi):=\K|\xi|^{p-2}\xi$}
\end{equation}
holds for some $\K>0$.
Note that the model case of the $p$-Laplacian, i.e. $a(\xi)=|\xi|^{p-2}\xi$ is included in \eqref{def_asymp}. 
Concerning the regularity of the boundary
values, i.e. of  $\partial\Omega$
and $g$, we shall assume that $g\colon \overline{\Omega}_T \to \R^N$ is a continuous function and
\begin{equation}\label{data}
	\partial\Omega \text{ is } C^{1;\beta},\quad
	Dg\in C^{0;\beta,0}\big(\overline\Omega_T,\R^{Nn}\big),\quad
	\partial_t g\in
	L^{p',(1-\beta)p'}\big(\Omega_T,\R^{N}\big)
\end{equation}
for some $\beta\in(0,1)$. As usual, by $p':=\frac{p}{p-1}$ we denote the H\"older conjugate of $p$ and $\overline\Omega_T:=\overline\Omega\times[0,T]$.
The definition of parabolic H\"older- and Morrey-spaces of the type $C^{0;\beta,0}$ and $L^{p',(1-\beta)p'}$ is given in Definitions \ref{def:Hoelder} and \ref{def:Morrey} below.
We now provide the notion of a weak solution
to the parabolic Cauchy-Dirichlet problem \eqref{system}.
\begin{definition}\label{def:weak-solution}
A map 
$$
	u \in C^0\big([0,T];L^2(\Omega,\R^{N})\big)\cap L^p\big(0,T;W^{1,p}(\Omega,\R^{N})\big)
$$ 
is called a (weak) solution to the parabolic Cauchy-Dirichlet problem \eqref{system}
if and only if
\begin{align}\label{weak}
	\int_{\Omega_T}
	u\cdot\varphi_t - \langle a(Du), D\varphi\rangle \dz
	=0
\end{align}
holds for every test-function $\varphi\in C_0^\infty(\Omega_T,\R^N)$, and the following boundary conditions are satisfied:
\begin{equation*}
	u(\cdot,t) - g(\cdot,t) \in W^{1,p}_0(\Omega,\R^N)
	\qquad\text{for a.e. } t\in(0,T)
\end{equation*}
and
$$
  \lim_{h \downarrow0 }\frac{1}{h} \int_{0}^{h} \int_\Omega |u(x,t) - g(x,0)|^2 \dx\dt =0
  \,.
$$
\end{definition}

As already explained before, our aim is to prove global and local Lipschitz regularity of the spatial gradient $Du$ of weak solutions to \eqref{system}. By Lipschitz regularity we of course mean Lipschitz with respect to the parabolic metric
$$
    d_{\mathcal P}\big((x,t),(y,s)\big):=\max \left\{|x-y|,\sqrt{|t-s|}\right\},
$$
for $x,y\in\R^n$ and $t,s\in\R$.
This is equivalent with the parabolic H\"older space $C^{0;1,1/2}$, i.e. the space of functions which are Lipschitz with respect to the spatial direction and H\"older-continuous with H\"older-exponent $1/2$ with respect to time; see Definition \ref{def:Hoelder} below.
Instead of \eqref{def_asymp} it will be convenient to use the following equivalent definition of asymptotic regularity:

\begin{remark}\label{rem:asyp-equiv}\upshape
The vector field $a$ is asymptotically of first order related to the $p$-Laplacian in the sense of \eqref{def_asymp} if and only if
\begin{equation}\label{def_asymp-}
    |Da(\xi) - Db(\xi)|
    \le
    \omega(|\xi|)(1+|\xi|)^{p-2},
    \qquad\forall\, \xi\in \R^{Nn}
\end{equation}
holds for some bounded function $\omega\colon[0,\infty)\to[0,\infty)$ with 
\begin{equation}\label{omega}
    \lim_{s\to \infty} \omega(s)
    = 0.
\end{equation}
\hfill$\Box$
\end{remark}

We are now in the position to state the main result.
\begin{theorem}\label{thm:main-lip}
Let
$$
	u \in C^0\big([0,T];L^2(\Omega,\R^{N})\big)\cap L^p\big(0,T;W^{1,p}(\Omega,\R^{N})\big)
$$ 
be a weak solution to the parabolic Cauchy-Dirichlet problem \eqref{system} in $\Omega_T$ under the assumptions \eqref{def_asymp} and \eqref{data}. Then, for any $\epsilon\in(0,T)$ we have
$$
	Du\in L^\infty\big(\Omega\times(\epsilon,T),\R^{Nn}\big)
	\quad\mbox{and}\quad
	u\in C^{0;1,1/2}\big(\Omega\times(\epsilon,T),\R^{N}\big)
$$
and there exists $\rho_o\in(0,1]$ such that the quantitative estimate
\begin{equation*}
    |Du(z_o)|
    \le
    c_1
    \bigg(\mint_{Q_{\rho}(z_o)\cap\Omega_T}|Du|^p\dz\bigg)^\frac12 + 
    c_2
\end{equation*}
holds for a.e. $z_o\in\Omega\times(\epsilon,T)$ and any parabolic cylinder $Q_{\rho}(z_o)\subset\R^n\times(0,T)$ with $\rho\in(0,\rho_o]$. Thereby, the constant
$c_1$ depends only on $n,N,p,\K$ and 
$c_2$ and $\rho_o$ depend on $n,N,p,\K,$ $\partial\Omega,\beta,|a(0)|, \omega(\cdot), \|Dg\|_{C^{0;\beta,0}(\overline\Omega_T)}, \|g_t\|_{L^{p',(1-\beta)p'}(\Omega_T)}$, where $\omega(\cdot)$ is from Remark \ref{rem:asyp-equiv}.
\end{theorem}

\begin{remark}\upshape
Here, we remark on some possible generalizations of Theorem \ref{thm:main-lip}.
Indeed, the same result holds when the vector field $a$ is asymptotically of first order related to a vector field $b$ of quasi-diagonal Uhlenbeck structure, i.e.
\begin{equation}\label{Uhlenbeck}
	b(\xi)
	=
	f(|\xi|)\xi
	\quad\mbox{with } f(|\xi|) \approx |\xi|^{p-2}.
\end{equation}
The symbol $\approx$ of course has to be made precise in a suitable way.
Moreover, one could allow that $a$ additionally depends on $(x,t)$. In order to keep the exposition as clear as possible we decided to consider only the easiest case where $a$ satisfies \eqref{def_asymp}. Furthermore, we did not include investigations about the regularity at the initial time $t=0$ and restricted our considerations to the lateral boundary.
\hfill$\Box$
\end{remark}

The strategy in the proof to deal with non-homogeneous boundary data $g$ is to interpret $g$ as a perturbation of the solution $u$. More precisely, instead of $u$ we consider the function $v:=u-g$ which has boundary values equal to zero. Then, $v$ satisfies the parabolic system 
\begin{equation*}
	\partial_t v - \Div a(Dv+Dg) = -\partial_t g
\end{equation*}
and since $Dg$ is assumed to be bounded we have that $Dv$ is bounded if and only if $Du$ is bounded. 
Here, we should also mention that a change of variable allows us to reduce the proof to the model situation where $\Omega_T$ is a half-cylinder.
Now, the first important observation is that it is enough to consider regions where $|Dv|$ is in a certain averaged sense larger then $\|Dg\|_\infty$, since otherwise we can bound $|Dv|$ in terms of $\|Dg\|_\infty$.
The next crucial step is to exploit the structure of the vector field $a$ which is done by two different comparison arguments. If -- roughly speaking -- the mean value of $|Dv|^p$ 
on some cylinder is small compared to the oscillations of $Dv$ (this case is called the degenerate regime) we compare $v$ to the solution $w$ of the parabolic system
\begin{equation*}
	\partial_t w - \Div b(Dw) = 0
	\qquad \mbox{on $Q\cap\Omega_T$ for some cylinder $Q$}
\end{equation*}
which has lateral and initial boundary values equal to $v$. The advantage now is that $w$ is zero on the boundary portion $Q\cap(\partial\Omega\times(0,T))$ and therefore satisfies certain a priori estimates which are a consequence of the $C^{1,\alpha}$ regularity theory of DiBenedetto \& Friedman. Nevertheless, the application of the a priori estimate is not straight forward since it involves intrinsic cylinders as explained in \eqref{intrinsic}.
The comparison argument together with the a priori estimate then ensure that the mean value of $|Dv|$ on some smaller cylinder remains small. Thereby one has to ensure that the smaller cylinder still is an intrinsic one, with a possibly different scaling factor.

On the other hand, if -- roughly speaking -- the mean value of $|Dv|^p$ 
on some intrinsic cylinder is large compared to the oscillations of $Dv$ (this case is called the non-degenerate regime)
we compare $v$ to the solution of a linear parabolic system. This is achieved via the so called $\mathcal A$-calloric approximation lemma from \cite{Duzaar-Mingione-Steffen:2011} which is a parabolic counterpart of De Giorgi's Harmonic Approximation Lemma \cite{DeGiorgi:1968}.
Together with good a priori estimates for solutions to linear systems and an iteration argument we can prove a bound for the gradient in the center of the cylinder in this case.

Finally, in order to obtain the desired gradient bound in any case we have to combine the degenerate and the non-degenerate regime. This is achieved via a delicate choice of the involved radii 
and the observation that the conditions for both regimes perfectly match together.

\section{Notation and auxiliary tools}

\subsection{Notations}\label{sec:notation}
Throughout the paper we will write $x=(x_1,\dots,x_n)$ for a point in $\R^n$ and $z=(x,t)=(x_1,\dots,x_n,t)$ for a point in $\R^{n+1}$. By $B_\rho(x_o):=\{x\in\R^n:|x-x_o|<\rho\}$, respectively $B_\rho^+(x_o):=B_\rho(x_o)\cap \{x\in\R^n:x_n>0\}$ we denote the open ball, respectively upper part of the open ball in $\R^n$ with center $x_o\in\R^n$ and radius $\rho>0$. When considering $B_\rho^+(x_o)$ we do not necessarily assume $(x_o)_n=0$. Indeed, if $B_\rho(x_o)\subset\{x\in\R^n:x_n>0\}$ it can also happen that $B_\rho^+(x_o)\equiv B_\rho(x_o)$. Moreover, we write 
$$
	\Lambda_{\rho,\lambda}(t_o)
	:=
	\big(t_o-\lambda^{2-p}\rho^2, t_o+\lambda^{2-p}\rho^2\big)
$$
for the open interval around $t_o\in\R$ of length $2\lambda^{2-p}\rho^2$ with $\rho,\lambda>0$. As basic sets for our estimates we usually take cylinders. These are denoted by 
$$
	Q_{\rho,\lambda}(z_o)
	:=
	B_\rho(x_o)\times \Lambda_{\rho,\lambda}(t_o)
$$
and the upper part of the cylinder by
$$
	Q_{\rho,\lambda}^+(z_o)
	:=
	B_\rho^+(x_o)\times \Lambda_{\rho,\lambda}(t_o),
$$
where $z_o=(x_o,t_o)\in\R^{n+1}$. As before, when considering $Q_{\rho,\lambda}^+(z_o)$ we do not necessarily assume $(x_o)_n=0$. For the hyperplane $x_n=0$ in $\R^{n+1}$ we write
$$
	\Gamma
	:=
	\big\{(x_1,\dots,x_{n-1},0,t)\in\R^{n+1}\big\}
$$
and
$$
	\Gamma_{\rho,\lambda}(z_o)
	:=
	Q_{\rho,\lambda}(z_o)\cap\Gamma
$$
for the flat part of the lateral boundary of $Q_{\rho,\lambda}^+(z_o)$. Note that it can happen that $\Gamma_{\rho,\lambda}(z_o)=\emptyset$. If $\lambda=1$ we use the shorter notations 
$\Lambda_{\rho}(t_o):=\Lambda_{\rho,1}(t_o)$, $Q_{\rho}(z_o):=Q_{\rho,1}(z_o)$, $\Gamma_{\rho}(z_o):=\Gamma_{\rho,1}(z_o)$
and if furthermore $z_o=0$ we write
$B_{\rho}:=B_{\rho}(0)$, $\Lambda_{\rho}:=\Lambda_{\rho}(0)$, $Q_{\rho}:=Q_{\rho}(0)$, $\Gamma_{\rho}:=\Gamma_{\rho}(0)$.
For an integrable map $v\colon A\to \R^k$, $k\in\N$, $|A|>0$ we denote by 
$$
	(v)_{A}
	:=
	\mint_A v\dz
	=
	\frac{1}{|A|}\int v\dz
$$ 
its mean value on $A$. If $A=Q_{\rho,\lambda}(z_o)$ we write 
$(v)_{z_o;\rho,\lambda}$ for the mean value of $v$ on the cylinder $Q_{\rho,\lambda}(z_o)$ and 
$(v)_{z_o;\rho,\lambda}^+$ for the mean value on the upper part $Q_{\rho,\lambda}^+(z_o)$ of the cylinder. As before, when $\lambda=1$ we use the short-hand notations
$(v)_{z_o;\rho}:=(v)_{z_o;\rho,1}$ and
$(v)_{z_o;\rho}^+:=(v)_{z_o;\rho,1}^+$.
Next, we define the relevant function spaces
\begin{definition}\label{def:Hoelder}
With $\alpha,\beta \in (0,1]$ and $Q \subset \R^{n+1}$ being a bounded open set, a  map $v\colon Q\to \R^k$, $k\ge 1$ belongs
to the parabolic H\"older space $C^{0;\alpha,\beta}(Q,\R^k)$ if and only if
\begin{equation*}
	\|v\|_{C^{0;\alpha,\beta}(Q,\R^k)}
	:=
	\sup_{z_o\in Q} |v(z_o)| +
	\sup_{z_o,z_1\in Q, z_o\not= z_1} 
	\frac{|v(z_o)-v(z_1)|}{|x_o-x_1|^\alpha + |t_o-t_1|^\beta}
	<
	\infty.
\end{equation*}
\end{definition}

\begin{definition}\label{def:Morrey}
With $q \geq 1$, $\theta \in [0,n+2]$ and $Q \subset \R^{n+1}$ being a bounded open set, a measurable map $v\colon Q\to \R^k$, $k\ge 1$ belongs
to the parabolic Morrey space $L^{q,\theta}(Q,\R^k)$ if and only if
\begin{equation*}
	\|v\|_{L^{q,\theta}(Q,\R^k)}^q
	:=
	\sup_{z_o\in Q,\, 0<\rho<\diam(Q)}
	\varrho^{\theta-(n+2)} \int_{Q\cap Q_\rho(z_o)} |v|^q \dz
	<
	\infty.
\end{equation*}
\end{definition}

\subsection{The $V$-function}\label{sec:V}
Since we are dealing with $p$-growth problems it is convenient to use
the function $V_\mu\colon\R^k\to\R^k$, with $1<p<\infty$, $\mu\in [0,1]$ and $k\in\N$, devined by
\begin{equation}\label{def-V}
    V_\mu(A)
    :=
    \big(\mu^2 + |A|^2\big)^{\frac{p-2}{4}}A
    \qquad\text{for }\ A\in\R^k .
\end{equation}
The basic properties of the $V$-function and some related estimates are stated in the following lemmata.

\begin{lemma}\label{lem:monotone}
For $1<p<\infty$ and $k\in\N$ there exists a constant $c=c(k,p)\ge 1$ such that for any $A,B\in\R^k$ there holds

\vspace{0.1cm}
{\rm (i)} 
$
	\big\langle |B|^{p-2}B - |A|^{p-2}A, B-A\big\rangle
    \ge
    \tfrac{1}{c}\, |V_{|A|}(B-A)|^2
$,

{\rm (ii)} 
$
    \big||B|^{p-2}B - |A|^{p-2}A\big|
    \le
    c\, \big(|A|^2 + |B-A|^2\big)^{\frac{p-2}{2}}|B-A|.
$
\end{lemma}

The following algebraic fact can be deduced from \cite[Lemma 2.1]{Acerbi-Fusco:1989}.
\begin{lemma}\label{lem:Fusco}
For every $\sigma\in(-1/2,0)$, $\mu\ge 0$ and $k\in\N$ we have
\begin{equation*}
    \int_0^1 \big(\mu^2 + |A + sB|^2\big)^\sigma\ds
    \le
    \frac{16}{2\sigma  +1}\,
    \big(\mu^2 + |A|^2 + |B|^2\big)^\sigma
    \qquad\forall A,B\in\R^{k}.
\end{equation*}
\end{lemma}

The first part of the following lemma states that the mean value $(Du)_Q$ is a quasi-minimum of the mapping $\xi\mapsto \tmint_Q |V_{|\xi|}(Du-\xi)|^2\dz$ amongst all $\xi\in\R^{Nn}$, cf. \cite[Lemma 2.4]{Boegelein-Duzaar-Mingione:2010}. The second part of the lemma can be seen as a boundary version of this fact and can be proved by similar arguments.

\begin{lemma}\label{lem:V-A}
Let $p\ge 2$, $Q\subset\R^{n+1}$ such that $|Q|>0$ and $u\in L^p(Q,\R^N)$ with $Du\in L^p(Q,\R^{Nn})$. 

\vspace{0.1cm}
{\rm (i)} Then, for any $\xi\in\R^{Nn}$ we have
\begin{align*}
    \mint_{Q} \big|V_{|(Du)_Q|}\big(Du-(Du)_Q\big)\big|^2 \dz
    \le
    2^{2p}
    \mint_{Q} \big|V_{|\xi|}(Du-\xi)\big|^2 \dz .
\end{align*}

{\rm (ii)} Moreover, for any $\eta\in\R^{N}$ we have
\begin{align*}
    \mint_{Q} \big|V_{|(D_nu)_Q\otimes e_n|}\big( Du-(D_nu)_Q\otimes e_n\big)\big|^2 \dz
    \le
    2^{2p}
    \mint_{Q} \big|V_{|\eta\otimes e_n|}(Du-\eta\otimes e_n)\big|^2 \dz .
\end{align*}
\end{lemma}


\subsection{Relevant affine functions}\label{sec:affine}
Later on, we will transform the original Cauchy-Dirichlet problem \eqref{system} to a model problem on a half-cylinder $Q_R^+$ for some $R>0$. 
In order to approximate the solution to the model problem on a cylinder $Q_{\rho,\lambda}^+(z_o)\subset Q_R^+$ we shall use certain affine functions depending only on the spatial variable $x$. This reflects the fact that we want to prove an $L^\infty$-bound for the spatial derivative $Du$. 
When defining these affine functions one usually considers only the cases where $Q_{\rho,\lambda}(z_o)$ is an interior cylinder, i.e. when $(x_o)_n\ge \rho$ and hence $Q_{\rho,\lambda}(z_o)\subset Q_R^+$ or a boundary cylinder, i.e. when $(x_o)_n=0$, respectively $z_o\in\Gamma$.
In the latter case one has to work with affine functions which are zero on $\Gamma_{\rho,\lambda}(z_o)$, since such functions can be used as testing functions in the weak formulation of the problem. 
On the contrary, for our purposes it will be more convenient to work on general cylinders $Q_{\rho,\lambda}(z_o)\subset Q_R$ with $z_o\in Q_R^+\cup\Gamma_R$ which are not necessarily interior or boundary cylinders. Therefore, we will extend the definition of the relevant affine functions to such cylinders.
Recalling the notational convention 
$z_o\equiv (x_o,t_o)\equiv ((x_o)_1,\dots,(x_o)_n,t_o)$ we define for $z_o \in \R^{n+1}$ with $(x_o)_n\ge 0$ the affine function
\begin{equation}\label{def-ell-mean}
	\ell_{z_o;\rho,\lambda}(x)
	:=
	\begin{cases}
		(u)_{z_o;\rho,\lambda}^+ + (Du)_{z_o;\rho,\lambda}^+ (x-x_o)
		&\mbox{if $(x_o)_n\ge \frac{\rho}{2}$,}\\[5pt]
		(D_nu)^+_{z_o;\rho,\lambda}\, x_n
		&\mbox{if $0\le (x_o)_n< \frac{\rho}{2}$.}
	\end{cases}
\end{equation}
Moreover, by $\widehat \ell_{z_o;\rho,\lambda}$ we denote the unique affine map minimizing 
\begin{equation}\label{def-lmin}
	\ell\mapsto \mint_{Q_{\rho,\lambda}^+(z_o)} |u-\ell|^2 \dz
\end{equation}
amongst all affine maps $\ell(z)=\ell(x)$ which are independent of $t$ and additionally satisfy $\ell =0$ on $\Gamma$ in the case that
$(x_o)_n< \frac{\rho}{2}$.
In the following we do not distinguish in notation affine maps on $\R^{n+1}$ from their restriction to $\R^n=\R^n\times\{0\}$.
If $Q_{\rho,\lambda}(z_o)$ is an interior or boundary cylinder, i.e. if $(x_o)_n\ge \rho$ or $(x_o)_n=0$, a straightforward computation 
shows that 
\begin{equation}\label{def-ell-min}
	\widehat \ell_{z_o;\rho,\lambda}(x)
	\equiv
	\begin{cases}
		(u)_{z_o;\rho,\lambda} + 
		\displaystyle\tfrac{n+2}{\rho^2}
		\mint_{Q_{\rho,\lambda}(z_o)} u\otimes (\tilde x{-}x_o) \,d\tilde z\, (x{-}x_o)
		&\mbox{if $(x_o)_n\ge \rho$}\\[12pt]
		\phantom{mmmmm}\displaystyle
		\tfrac{n+2}{\rho^2}
		\mint_{Q_{\rho,\lambda}^+(z_o)} u(\tilde z)\, \tilde x_n \,d\tilde z\ x_n
		&\mbox{if $(x_o)_n=0$.}
	\end{cases}
\end{equation}
Moreover, if $Q_{\rho,\lambda}(z_o)$ is an interior cylinder, i.e. if $(x_o)_n\ge \rho$ one can show
that for any $\eta\in\R^N$ and $\xi\in\R^{Nn}$ there holds
\begin{equation}\label{kronz-est}
    |D\widehat\ell_{z_o;\rho,\lambda} - \xi|^2
    \le
    \frac{n(n+2)}{\rho^2}
    \mint_{Q_{\rho,\lambda}(z_o)} |u-\eta-\xi(x-x_o)|^2 \dz
\end{equation}
while in the case that $Q_{\rho,\lambda}(z_o)$ is a boundary cylinder, i.e. $(x_o)_n=0$ we deduce from \cite[Lemma 2.2]{Boegelein-Duzaar-Mingione:2010-boundary} that for any $\zeta\in\R^{N}$ there holds
\begin{equation}\label{kronz-est-boundary}
    |D\widehat\ell_{z_o;\rho,\lambda} - \zeta\otimes e_n|^2
    \le
    \frac{n+2}{\rho^2}
    \mint_{Q_{\rho,\lambda}^+(z_o)} |u-\zeta\, x_n|^2 \dz.
\end{equation}
Finally, for interior and boundary cylinders one can show that the minimizer $\widehat\ell_{z_o;\rho,\lambda}$ of \eqref{def-lmin} is a quasi minimum of the mapping $\ell\mapsto \tmint_{Q_{\rho,\lambda}^+(z_o)} |u-\ell|^s\dz$ for any $s\ge 2$.
This fact is stated in the following lemma. The proof for interior cylinders can be found in 
\cite[Lemma 2.8]{Boegelein-Duzaar-Mingione:2010}, while the one for boundary cylinders follows by similar arguments.
\begin{lemma}\label{lem:quasimin}
Let $s\ge 2$ and $Q_{\rho,\lambda}(z_o)\subset \R^{n+1}$ be a parabolic cylinder. 

\vspace{0.1cm}
{\rm (i)} 
If $(x_o)_n\ge \rho$, then for any affine function $\ell\colon\R^n\to \R^N$ independent of $t$ there holds
\begin{align*}
    \mint_{Q_{\rho,\lambda}(z_o)} |u-\widehat\ell_{z_o;\rho,\lambda}|^s \dz
    \le
    c(n,s) \mint_{Q_{\rho,\lambda}(z_o)} |u-\ell|^s \dz.
\end{align*}

{\rm (ii)} 
If $(x_o)_n=0$, then for any $\zeta\in \R^N$ we have that
\begin{align*}
    \mint_{Q_{\rho,\lambda}^+(z_o)} |u-\widehat\ell_{z_o;\rho,\lambda}|^s \dz
    \le
    c(n,s) \mint_{Q_{\rho,\lambda}^+(z_o)} |u-\zeta\,x_n|^s \dz.
\end{align*}
\end{lemma}

\subsection{A priori estimates}
From the theory of linear parabolic systems it is known that weak solutions are smooth in the interior and up to the boundary.
Later on, in the so called non-degenerate regime where the solution behaves approximately like a solution to a linear parabolic system, we shall exploit good \textit{excess-decay-estimates} for linear parabolic systems with constant coefficients. Since we are dealing with different situations, i.e. interior cylinders contained in the half-space $\{z\in\R^{n+1}:x_n>0\}$ and cylinders intersecting the hyperplane $\Gamma$, we will need suitable excess-decay-estimates for any of them.

Our aim in this section is to provide a priori estimates not depending on the position of the larger cylinder $Q_\rho(z_o)$; see inequality \eqref{apriori} below.
The precise form of the estimates plays a crucial role later, since it allows for a proper comparison argument after linearization. For a parabolic cylinder $Q_\rho(z_o)$ with $(x_o)_n\ge 0$ we consider the following linear parabolic system with constant coefficients
\begin{equation}\label{lin-system}
	\left\{
	\begin{array}{cc}
	\partial_t h - \Div( ADh)
	=
	0
	&\qquad\mbox{in $Q_\rho^+(z_o)$} \\[7pt]
	h=0
	&\qquad\mbox{on $\Gamma_{\rho}(t_o)$ if $\Gamma_{\rho}(t_o)\not=\emptyset$}.
	\end{array}
	\right.
\end{equation}
Thereby the coefficients $A$ are supposed to satisfy the following ellipticity and boundedness conditions:
\begin{equation}\label{lin-cond-A}
	\langle A\xi,\xi\rangle\geq \nu\,|\xi|^2\, ,
	\qquad
	\langle A\xi,\xi_o\rangle\le L\,|\xi||\xi_o|\, ,
\end{equation}
whenever $\xi,\xi_o\in\R^{Nn}$ and with some parameters $0<\nu\le L<\infty$. For a good first order approximation with respect to $x$ of the solution one typically uses certain affine functions depending only on $x$ and not on the time variable $t$. As in Section \ref{sec:affine} we distinguish in the definition of the affine function whether we are more in the interior, or more in the boundary situation. More precisely, we set
\begin{equation}\label{def-lh}
	\ell_{z_o;\rho}^{(h)}(x,t)
	:=
	\begin{cases}
		(h)_{z_o;\rho}^+ + (Dh)_{z_o;\rho}^+ (x-x_o)
		&\mbox{if $(x_o)_n\ge \rho/2$,}\\[5pt]
		(D_nh)^+_{z_o;\rho}\, x_n
		&\mbox{if $0\le (x_o)_n< \rho/2$.}
	\end{cases}
\end{equation}
Note that in the boundary situation, i.e. when the second line of the definition is in force, we have $\ell_{z_o;\rho}^{(h)}\equiv 0$ on $\Gamma$.
With this notation at hand we can now state a unified \textit{up-to-the-boundary excess-decay-estimate for linear parabolic systems}.

\begin{proposition}\label{prop:lin}
Let $s\ge 2$ and $Q_\rho(z_o)\subset \R^{n+1}$ with $(x_o)_n\ge 0$ and suppose that
$$
	h\in C^0\big(\Lambda_{\rho}(t_o); L^2(B_\varrho^+ (x_o) ,\R^N)\big)\cap
	L^s\big(\Lambda_{\rho}(t_o); W^{1,s}(B_\varrho^+ (x_o) ,\R^N)\big)
$$
is a weak solution of the linear parabolic system \eqref{lin-system} in $Q_\varrho^+(z_o)$ under the assumption \eqref{lin-cond-A}. Then, $h$ is smooth in $Q_\rho^+(z_o)$ and also smooth up to the lateral boundary $\Gamma_\rho(z_o)$ if $\Gamma_\rho(z_o)\not=\emptyset$. Moreover, there exists a constant $c=c(n,N,\nu,L,s)$ such that for any $\theta\in(0,1/16]$ there holds
\begin{equation}\label{apriori}
	\mint_{Q_{\theta\rho}^+(z_o)}
	\bigg|
    \frac{h-\ell_{z_o;\theta\rho}^{(h)}}{\theta\rho}
    \bigg|^s
	\dz
	\le
	c\,\theta^s
	\mint_{Q_{\rho}^+(z_o)}
	\Big|\frac{h}{\rho}\Big|^s + |Dh|^s \dz 
\end{equation}
provided either $(x_o)_n\ge 2\theta\rho$ or $(x_o)_n=0$.
\end{proposition}

For the proof of Proposition \ref{prop:lin} we will combine the following classical results for linear parabolic systems. The first Lemma contains a Caccioppoli and a Poincar\'e's inequality for solutions which can be deduced for instance from 
\cite[Chapter 5]{Campanato:1966} and \cite[Lemma 3.1]{Boegelein-Duzaar-Mingione:2010} applied with $w(x,t)=h(x,t)-D\ell_{z_o;r}^{(h)}(x-x_o)$ and $\xi(x,t)= A(Dh(x,t)-D\ell_{z_o;r}^{(h)})$.

\begin{lemma}\label{lem:lin-cac}
Let $Q_\rho(z_o)\subset \R^{n+1}$ with $(x_o)_n\ge 0$ and suppose that
$$
	h\in C^0\big(\Lambda_{\rho}(t_o); L^2(B_\varrho^+ (x_o) ,\R^N)\big)\cap
	L^2\big(\Lambda_{\rho}(t_o); W^{1,2}(B_\varrho^+ (x_o) ,\R^N)\big)
$$
is a weak solution of the linear parabolic system \eqref{lin-system} in $Q_\varrho^+(z_o)$ under the assumption \eqref{lin-cond-A}. Then, for any $r\in(0,\rho]$ with $r\le (x_o)_n$ there holds
\begin{equation*}
	\mint_{Q_{r/2}(z_o)}
	\big|Dh-D\ell_{z_o;r}^{(h)}\big|^2 \dz
	\le
	c(\nu,L)\mint_{Q_{r}(z_o)}
	\bigg|\frac{h-\ell_{z_o;r}^{(h)}}{r}\bigg|^2 \dz 
\end{equation*}
and
\begin{equation*}
	\mint_{Q_{r}(z_o)}
	\big|h-\ell_{z_o;r}^{(h)}\big|^2 \dz
	\le
	c(n,N,L)\, r
	\mint_{Q_{r}(z_o)}
	\big|Dh-D\ell_{z_o;r}^{(h)}\big|^2 \dz .
\end{equation*}
\end{lemma}

%

The next lemma already states the up-to-the-boundary smoothness of solutions to the linear parabolic system \eqref{lin-system} and provides a priori estimates on interior cylinders $Q_\rho(z_o)$ with $(x_o)_n\ge\rho$ and on half-cylinders $Q_\rho^+(z_o)$ with $(x_o)_n=0$.
The proof for the boundary situation can be deduced from the proof of 
\cite[Corollary 3.4]{Boegelein-Duzaar-Mingione:2010-boundary}, while the interior regularity is a consequence of \cite[Chapter 5]{Campanato:1966}. 

\begin{lemma}\label{lem:lin-excess}
Let $s\ge 2$ and $Q_\rho(z_o)\subset \R^{n+1}$ with $(x_o)_n\ge 0$ and suppose that
$$
	h\in C^0\big(\Lambda_{\rho}(t_o); L^2(B_\varrho^+ (x_o) ,\R^N)\big)\cap
	L^s\big(\Lambda_{\rho}(t_o); W^{1,s}(B_\varrho^+ (x_o) ,\R^N)\big)
$$
is a weak solution of the linear parabolic system \eqref{lin-system} in $Q_\varrho^+(z_o)$ under the assumption \eqref{lin-cond-A}. Then, $h$ is smooth in $Q_\rho^+(z_o)$ and also smooth up to the lateral boundary $\Gamma_\rho(z_o)$ if $\Gamma_\rho(z_o)\not=\emptyset$. Moreover, there exists a constant $c = c(n,\nu,L,s)$ such that for any $\theta\in(0,1/2)$ there holds
\begin{equation*}
 	\mint_{Q_{\theta \rho}^+(z_o)}
	\bigg|
    \frac{h-\ell_{z_o;\theta \rho}^{(h)}}{\theta \rho}
    \bigg|^s
	\dz
	\le
	c\,\theta^s
	\bigg( \mint_{Q_{\rho}^+(z_o)}
	\bigg|
    \frac{h-\ell_{z_o;\rho}^{(h)}}{\rho}
    \bigg|^2 \dz \bigg)^{\frac s2},
\end{equation*}
provided either $(x_o)_n\ge \rho$ or $(x_o)_n=0$. 
\end{lemma}

Now we come to the proof of Proposition \ref{prop:lin}. In the case where either $(x_o)_n\ge \rho$ or $(x_o)_n=0$ holds Proposition \ref{prop:lin} directly follows from Lemma \ref{lem:lin-excess}. Therefore, the crucial point in the proof will be to replace the assumption $(x_o)_n\ge \rho$ by the weaker assumption $(x_o)_n\ge 2\theta\rho$.

\begin{proof}[\rm\bf Proof of Proposition \ref{prop:lin}]
The smoothness of $h$ directly follows from Lemma \ref{lem:lin-excess}. Moreover, 
in the case that either $(x_o)_n\ge \rho$ or $(x_o)_n=0$ holds, \eqref{apriori} follows from Lemma \ref{lem:lin-excess} and the definition of 
$\ell_{z_o;\rho}^{(h)}$.
Therefore, it remains to consider the case where $2\theta\rho\le (x_o)_n< \rho$.

We fist consider the case $2\theta\rho\le (x_o)_n\le \rho/8$.
Applying Lemma \ref{lem:lin-excess} with $(\theta\rho,(x_o)_n)$ instead of $(\theta\rho,\rho)$ we obtain
\begin{align*}
	\mint_{Q_{\theta\rho}(z_o)}
	\bigg|\frac{h-\ell_{z_o;\theta\rho}^{(h)}}{\theta\rho}\bigg|^s
	\dz 
	\le
	c\,\bigg(\frac{\theta\rho}{(x_o)_n}\bigg)^s
	\bigg(\mint_{Q_{(x_o)_n}(z_o)}
	\bigg|\frac{h-\ell_{z_o;(x_o)_n}^{(h)}}{(x_o)_n}\bigg|^2
	\dz \bigg)^\frac{s}{2}.
\end{align*}
In the following we denote by $z_o':=(x_o',t_o):=((x_o)_1\dots,(x_o)_{n-1},0,t_o)$ the orthogonal projection of $z_o$ on $\Gamma$.
With the help of Lemma \ref{lem:lin-cac} and the minimizing property of $D\ell_{z_o;(x_o)_n}^{(h)}=(Dh)_{z_o;(x_o)_n}$ we further estimate the integral on the right-hand side as follows:
\begin{align*}
	\mint_{Q_{(x_o)_n}(z_o)}
	\bigg|\frac{h-\ell_{z_o;(x_o)_n}^{(h)}}{(x_o)_n}\bigg|^2 \dz
	&\le
	c\mint_{Q_{(x_o)_n}(z_o)}
	|Dh-D\ell_{z_o;(x_o)_n}^{(h)}|^2 \dz \\
	&\le
	c\mint_{Q_{(x_o)_n}(z_o)}
	|Dh-D\ell_{z_o';4(x_o)_n}^{(h)}|^2 \dz \\
	&\le
	c\mint_{Q_{2(x_o)_n}^+(z_o')}
	|Dh-D\ell_{z_o';4(x_o)_n}^{(h)}|^2 \dz \\
	&\le
	c\mint_{Q_{4(x_o)_n}^+(z_o')}
	\bigg|\frac{h-\ell_{z_o';4(x_o)_n}^{(h)}}{4(x_o)_n}\bigg|^2 \dz ,
\end{align*}
where $c=c(n,N,\nu,L)$.
Now we are in the position to apply once again Lemma \ref{lem:lin-excess} to the right-hand side, but now in the boundary version, i.e. with $(z_o,\theta\rho,\rho)$ replaced by $(z_o',4(x_o)_n,\rho/2)$ and with $s=2$. This leads to the estimate
\begin{align*}
	\mint_{Q_{(x_o)_n}(z_o)}
	\bigg|\frac{h-\ell_{z_o;(x_o)_n}^{(h)}}{(x_o)_n}\bigg|^2 \dz
	&\le
	c\,\bigg(\frac{4(x_o)_n}{\rho/2}\bigg)^2
	\mint_{Q_{\rho/2}^+(z_o')}
	\bigg|
    \frac{h-\ell_{z_o';\rho/2}^{(h)}}{\rho/2}
    \bigg|^2 \dz .\\
\end{align*}
Inserting this above and applying H\"older's inequality we find
\begin{align*}
	\mint_{Q_{\theta\rho}(z_o)}
	\bigg|
    \frac{h-\ell_{z_o;\theta\rho}^{(h)}}{\theta\rho}
    \bigg|^s
	\dz 
	&\le
	c\,\theta^s 
	\mint_{Q_{\rho/2}^+(z_o')}
	\bigg|
    \frac{h-\ell_{z_o';\rho/2}^{(h)}}{\rho/2}
    \bigg|^s
	\dz \\
	&\le
	c(n,N,\nu,L,s)\,\theta^s 
	\mint_{Q_{\rho}^+(z_o)}
	\Big|\frac{h}{\rho}\Big|^s + |Dh|^s
	\dz 
\end{align*}
and proves the asserted a priori estimate in the case $2\theta\rho\le (x_o)_n< \rho/8$.
If $\rho/8\le (x_o)_n< \rho$ we apply the a priori estimate from above with $(\theta\rho,\rho/8)$ instead of $(\theta\rho,\rho)$ and subsequently enlarge the domain of integration from $Q_{\rho/8}^+(z_o)$ to $Q_{\rho}^+(z_o)$. 
This finishes the proof of the lemma.
\end{proof}

It is not always possible to compare the solution of the original parabolic system \eqref{system} to a solution of a linear parabolic system. In the so called degenerate regime the solution will only be comparable to the solution of a parabolic $p$-Laplacian system. Consequently, we shall also need a priori estimates for solutions to the $p$-Laplacian system. These a priori estimates can be deduced from the $C^{1;\alpha}$-theory of DiBenedetto \& Friedman
\cite{DiBenedetto:book, DiBenedetto-Friedman:1984, DiBenedetto-Friedman:1985} by the arguments of \cite{Boegelein-Duzaar-Mingione:2010, Boegelein-Duzaar:2011}. In the boundary situation we shall use a reflection argument to obtain suitable estimates. This is possible since we assume that the solution is zero on the flat part of the boundary.

\begin{lemma}\label{lem:DiBe}
Let $p\ge 2$, $c_\ast\ge 1$ and let $Q_{\rho,\lambda}(z_o)\subset \R^{n+1}$ be a cylinder with
$(x_o)_n\ge 0$, $\rho\in(0,1]$ and $\lambda>0$ and suppose that
\begin{equation*}
    v\in C^0\big(\Lambda_{\rho,\lambda}(t_o);L^2(B_\rho^+(x_o),\R^N)\big) \cap
    L^p\big(\Lambda_{\rho,\lambda}(t_o);W^{1,p}(B_\rho^+(x_o),\R^N)\big)
\end{equation*}
is a weak solution of
\begin{equation}\label{p-laplace}
	\left\{
	\begin{array}{cc}
    \partial_t v - \K\Div \big(|Dv|^{p-2}Dv\big) = 0
    &\mbox{in $Q_{\rho,\lambda}^+(z_o)$,} \\[7pt]
    v= 0
    & \mbox{on $\Gamma_{\rho,\lambda}(z_o)$ if\, $\Gamma_{\rho,\lambda}(z_o)\not=\emptyset$,}
    \end{array}
    \right.
\end{equation}
satisfying
\begin{equation}\label{DiBe-intrinsic}
    \mint_{Q_{\rho,\lambda}^+(z_o)} |Dv|^p\dz
    \le
    c_*\, \lambda^p.
\end{equation}
Moreover, let $r\in(0,\rho/2]$ and suppose that either $(x_o)_n\ge r$ or $(x_o)_n=0$.
Then, there exist constants $\alpha_o=\alpha_o(n,p,\K) \in(0,1)$, $\mu_o=\mu_o(n,N,p,\K,c_*)\ge 1$ and $\rho_s\in[0,\rho/2]$,
such that the following holds: In the case $\rho_s>0$ there exists $\mu$
such that $Q_{r,\lambda\mu}(z_o) \subset Q_{\rho/2,\lambda}(z_o)$ and
\begin{equation}\label{DiBe-mu}
    \mu_o\bigg(\frac{2\max\{r,\rho_s\}}{\rho}\bigg)^{\alpha_o}
    \le
    \mu
    \le
    2\mu_o \bigg(\frac{2\max\{r,\rho_s\}}{\rho}\bigg)^{\alpha_o}
\end{equation}
and
\begin{equation}\label{DiBe-sup}
    \sup_{Q_{r,\lambda\mu}^+(z_o)} |Dv|
    \le
    \lambda\mu
\end{equation}
holds. Moreover, 
for any $s\in [2,p]$ we have
\begin{equation}\label{DiBe-ex}
    \mint_{Q_{r,\lambda\mu}^+(z_o)} |Dv-D\ell^{(v)}_{z_o;r,\lambda\mu}|^s\dz
    \le
    c\, \lambda^s\mu^s\,\min\Big\{1,\frac{r}{\rho_s}\Big\}^{2\alpha_o} 
\end{equation}
with a constant $c=c(n,N,p,\K)$.
Finally, if $0<r<\rho_s$ then there holds
\begin{equation}\label{DiBe-mean}
    |D\ell^{(v)}_{z_o;r,\lambda\mu}|
    \ge
    \frac{\lambda\mu}{16}\, ,
\end{equation}
where $\ell^{(v)}_{z_o;r,\lambda\mu}$ is defined according to \eqref{def-ell-mean} with $v$ instead of $u$.
In the case $\rho_s=0$ we have \eqref{DiBe-mu} with $\rho_s=0$, \eqref{DiBe-sup}, and \eqref{DiBe-ex} with $r/\rho_s:=\infty$. We
note that $\rho_s$ cannot be explicitly computed and might depend on $z_o$ and the solution $v$ itself.
\end{lemma}

\begin{proof}
If $Q_{\rho,\lambda}(z_o)$ is an interior cylinder, i.e. $(x_o)_n\ge \rho$ the assertion has been proved in \cite[Lemma 3.3]{Boegelein-Duzaar:2011}, respectively in a slightly less general from in
\cite[Lemma 7.1]{Boegelein-Duzaar-Mingione:2010}.
Therefore, it remains to consider the case where $0\le (x_o)_n< \rho$ and hence $Q_{\rho,\lambda}(z_o)\cap\Gamma\not=\emptyset$.
We set $Q:=Q_{\rho,\lambda}(z_o)\cup Q_{\rho,\lambda}(z_o^-)$, where $z_o^-:=(x_o',-(x_o)_n,t_o)$ denotes the reflection of $z_o$ at the hyperplane $\Gamma$.
For $z=:(x',x_n,t)\in Q$ we define
\begin{equation*}
	\tilde v(x',x_n,t)
	:=
	\begin{cases}
	v(x',x_n,t)
	& \mbox{if $x_n\ge 0$}\\[5pt]
	-v(x',-x_n,t)
	& \mbox{if $x_n< 0$.}
	\end{cases}
\end{equation*}
Then, one can show that $\tilde v$ is a weak solution of \eqref{p-laplace}$_1$ on $Q$.
From \eqref{DiBe-intrinsic} and the definition of $\tilde v$ we further have
\begin{equation*}
    \mint_{Q_{\rho,\lambda}(z_o)} |D\tilde v|^p\dz
    \le
    2\mint_{Q_{\rho,\lambda}^+(z_o)} |Dv|^p\dz
    \le
    2c_*\, \lambda^p.
\end{equation*}
Therefore, we can apply \cite[Lemma 3.3]{Boegelein-Duzaar:2011} to $\tilde v$ on $Q_{\rho,\lambda}(z_o)$ to infer the existence of $\alpha_o=\alpha_o(n,p,\K) \in(0,1)$, $\mu_o=\mu_o(n,N,p,\K,c_*)\ge 1$ and $\rho_s\in[0,\rho/2]$,
such that the following holds: In the case $\rho_s>0$, for any $0<r\le\rho/2$ there exists $\mu$
such that $Q_{r,\lambda\mu}(z_o) \subset Q_{\rho/2,\lambda}(z_o)$ and
\eqref{DiBe-mu} hold
and moreover
\begin{equation}\label{DiBe-sup-tilde}
    \sup_{Q_{r,\lambda\mu}(z_o)} |D\tilde v|
    \le
    \lambda\mu\, .
\end{equation}
In addition, for any $s\in [2,p]$ we have
\begin{equation}\label{DiBe-ex-tilde}
    \mint_{Q_{r,\lambda\mu}(z_o)} |D\tilde v-(D\tilde v)_{z_o;r,\lambda\mu}|^s\dz
    \le
    c\, \lambda^s\mu^s\,\min\Big\{1,\frac{r}{\rho_s}\Big\}^{2\alpha_o} ,
\end{equation}
with a constant $c=c(n,N,p,\K)$. If $0<r<\rho_s$ we additionally have
\begin{equation}\label{DiBe-mean-tilde}
    |(D\tilde v)_{z_o;r,\lambda\mu}|
    \ge
    \frac{\lambda\mu}{8}\ .
\end{equation}
In the case $\rho_s=0$ we have \eqref{DiBe-mu} with $\rho_s=0$, \eqref{DiBe-sup-tilde}, and \eqref{DiBe-ex-tilde} with $r/\rho_s:=\infty$. 
In the case $(x_o)_n\ge r$ the assertions \eqref{DiBe-sup} -- \eqref{DiBe-mean} directly follow from \eqref{DiBe-sup-tilde} -- \eqref{DiBe-mean-tilde} since then $\tilde v\equiv v$ on $Q_{r,\lambda\mu}^+(z_o)=Q_{r,\lambda\mu}(z_o)$.
Therefore, it remains to consider the case $(x_o)_n=0$.
Since $|Dv|=|D\tilde v|$ on $Q_{r,\lambda\mu}^+(z_o)$ the assertion \eqref{DiBe-sup} immediately follows from \eqref{DiBe-sup-tilde}. Moreover, from the definitions of $\ell^{(v)}_{z_o;r,\lambda\mu}$ and $\tilde v$ we obtain that 
\begin{equation*}
    D\ell^{(v)}_{z_o;r,\lambda\mu}
    =
    (D_n v)_{z_o;r,\lambda\mu}^+\otimes e_n
    =
    (D_n \tilde v)_{z_o;r,\lambda\mu}^+\otimes e_n
\end{equation*}
and
\begin{equation*}
    (D_n \tilde v)_{z_o;r,\lambda\mu}\otimes e_n
    =
    (D\tilde v)_{z_o;r,\lambda\mu}
\end{equation*}
since $(D_i \tilde v)_{z_o;r,\lambda\mu}=0$ for $i\in\{1,\dots,n-1\}$.
Hence from \eqref{DiBe-ex-tilde} we conclude that
\begin{align*}
    &\mint_{Q_{r,\lambda\mu}^+(z_o)} |Dv-D\ell^{(v)}_{z_o;r,\lambda\mu}|^s\dz
    =
    \mint_{Q_{r,\lambda\mu}^+(z_o)}
    |D\tilde v-(D_n\tilde v)_{z_o;r,\lambda\mu}^+\otimes e_n|^s\dz \\
    &\ \ \ \le
    c\mint_{Q_{r,\lambda\mu}^+(z_o)}
    |D\tilde v-(D_n\tilde v)_{z_o;r,\lambda\mu}\otimes e_n|^s\dz + 
    c\,|(D_n\tilde v)^+_{z_o;r,\lambda\mu}-(D_n\tilde v)_{z_o;r,\lambda\mu}|^s\\
    &\ \ \ \le
    c \mint_{Q_{r,\lambda\mu}(z_o)} |D\tilde v-(D\tilde v)_{z_o;r,\lambda\mu}|^s\dz 
    \le
    c\, \lambda^s\mu^s\,\min\Big\{1,\frac{r}{\rho_s}\Big\}^{2\alpha_o} ,
\end{align*}
which proves \eqref{DiBe-ex}.
Finally, if $0<r<\rho_s$ we have from \eqref{DiBe-mean-tilde} that
\begin{equation*}
	|D\ell^{(v)}_{z_o;r,\lambda\mu}|
	=
	|(D_n \tilde v)^+_{z_o;r,\lambda\mu}|
	=
	\tfrac12 |(D_n \tilde v)_{z_o;r,\lambda\mu}|
	=
	\tfrac12|(D\tilde v)_{z_o;r,\lambda\mu}|
	\ge 
	\frac{\lambda\mu}{16}
\end{equation*}
which proves \eqref{DiBe-mean}. This concludes the proof of the proposition.
\end{proof}

\subsection{$\mathcal A$-caloric approximation lemma}

In order to compare the solution of the original problem to the solution of a linear parabolic system 
we shall use the $\mathcal A$-caloric approximation lemma from \cite[Lemma 3.2]{Duzaar-Mingione-Steffen:2011}.

\begin{lemma}\label{lem:a-cal}
Given $\epsilon > 0$, $0<\nu\le L$ and $p\ge 2$ there exists a positive
function $\delta = \delta(n,$ $p,\nu,L,\epsilon) \in (0,1]$ with the
following property: Whenever $A$ is a  bilinear form on $\R^{Nn}$ satisfying \eqref{lin-cond-A} and whenever
\begin{equation*}
	u \in L^p\big(\Lambda_{\rho}(t_o);W^{1,p}(B_\rho(x_o),\R^N)\big)
\end{equation*} 
is a map satisfying
\begin{equation*}
	\mint_{Q_\rho(z_o)}
	\Big|\frac{u}{\rho}\Big|^2 + |Du|^2 \dz + 
	\gamma^{p-2} \mint_{Q_\rho(z_o)} \Big|\frac{u}{\rho}\Big|^p + |Du|^p \dz
	\le 1 ,
\end{equation*}
for some $0<\gamma\le 1$ and which is approximately $A$-caloric in the sense that
\begin{equation*}
	\bigg| \mint_{Q_\rho(z_o)}
	u \cdot\varphi_t - \langle A Du , D\varphi\rangle \dz
	\bigg|
	\le
	\delta \sup_{Q_\rho(z_o)} |D\varphi|,
	\quad\forall\, 
	\varphi \in C_0^\infty(Q_\rho(z_o),\R^N),
\end{equation*}
then there exists an $A$-caloric map
$h \in L^p(\Lambda_{\rho/2}(t_o); W^{1,p}(B_{\rho/2}(x_o),\R^N))$ (i.e. $h$ satisfies \eqref{lin-system}$_1$ on $Q_{\rho/2}(z_o)$)
satisfying
\begin{equation*}
	\mint_{Q_{\rho/2}(z_o)}
	\Big|\frac{h}{\rho/2}\Big|^2 + |Dh|^2 \dz + 
	\gamma^{p-2} \mint_{Q_{\rho/2}(z_o)}
	\Big|\frac{h}{\rho/2}\Big|^p + |Dh|^p \dz
	\le
	2^{n+3+2p}
\end{equation*}
and
\begin{equation*}
	\mint_{Q_{\rho/2}(z_o)}
	\Big|\frac{u-h}{\rho/2}\Big|^2 +
	\gamma^{p-2}\Big|\frac{u-h}{\rho/2}\Big|^p \dz
	\le
	\epsilon .
\end{equation*}
\end{lemma}

The following boundary version of the $\mathcal A$-caloric approximation lemma can be deduced from \cite[Lemma 4.1]{Boegelein-Duzaar-Mingione:2010-boundary} by considering suitable nested cylinders. 
This modification allows a slightly more flexible choice of the cylinders, in the sense that the assumption $(x_o)_n=0$ is replaced by the weaker assumption $(x_o)_n<\rho/2$.
We refer to \cite[Corollary 2.11]{Boegelein-habil} for the precise proof.

\begin{lemma}\label{lem:a-cal-boundary}
Given $\epsilon > 0$, $0<\nu\le L$ and $p\ge 2$ there exists a positive
function $\delta = \delta(n,p,$ $\nu,L,\epsilon) \in (0,1]$ with the
following property: Whenever $A$ is a  bilinear form on $\R^{Nn}$ satisfying \eqref{lin-cond-A} and whenever
\begin{equation*}
	u \in L^p\big(\Lambda_{\rho}(t_o);W^{1,p}(B_\rho^+(x_o),\R^N)\big)
\end{equation*}
with $(x_o)_n<\rho/2$ and
$u\equiv 0$ on the lateral boundary $\Gamma_\rho(z_o)$ and
\begin{equation*}
	\mint_{Q_\rho^+(z_o)}
	|Du|^2 + \gamma^{p-2}|Du|^p \dz
	\le 2^{-(n+2)} ,
\end{equation*}
for some $0<\gamma\le 1$, is approximately $A$-caloric in the sense that
\begin{equation*}
	\bigg| \mint_{Q_\rho^+(z_o)}
	u \cdot\varphi_t - \langle A Du , D\varphi\rangle \dz
	\bigg|
	\le
	\delta \sup_{Q_\rho^+(z_o)} |D\varphi|,
	\quad\forall\,
	\varphi \in C_0^\infty(Q_\rho^+(z_o),\R^N),
\end{equation*}
then there exists an $A$-caloric map
$h \in L^p(\Lambda_{\rho/4}(t_o); W^{1,p}(B_{\rho/4}^+(x_o),\R^N))$ 
(i.e. $h$ satisfies \eqref{lin-system} on $Q_{\rho/4}(z_o)$)
with $h\equiv0$ on $\Gamma_{\rho/4}(z_o)$ if $\Gamma_{\rho/4}(z_o)\not=\emptyset$ satisfying
\begin{equation*}
	\mint_{Q_{\rho/4}^+(z_o)}
	\Big|\frac{h}{\rho/4}\Big|^2 + |Dh|^2 \dz + 
	\gamma^{p-2}\mint_{Q_{\rho/4}^+(z_o)}
	\Big|\frac{h}{\rho/4}\Big|^p + |Dh|^p \dz
	\le
	c(n,p)
\end{equation*}
and
\begin{equation*}
	\mint_{Q_{\rho/4}^+(z_o)}
	\Big|\frac{u-h}{\rho/4}\Big|^2 +
	\gamma^{p-2}\Big|\frac{u-h}{\rho/4}\Big|^p \dz
	\le
	\epsilon .
\end{equation*}
\end{lemma}

\subsection{Asymptotic estimates}
In this section we deduce some useful estimates from the fact that the vector field $a$ is asymptotically regular in the sense of \eqref{def_asymp}. We start with an upper bound for $Da$.

\begin{lemma}\label{lem:growth-a}
Suppose that the vector field $a$ is asymptotically regular in the sense of \eqref{def_asymp}, then there exists a constant $c=c(p)$ such that
\begin{equation}\label{growth-Da}
    |Da(\xi)|
    \le
    c\,(\|\omega\|_\infty + \K)\,(1+|\xi|)^{p-2}
    \quad\forall\, \xi\in\R^{Nn}
\end{equation}
and
\begin{equation}\label{diff-a}
    |a(\xi) - a(\xi_o)| 
    \le
    c\, (\|\omega\|_\infty + \K) 
    \big(1+|\xi_o|^2 + |\xi - \xi_o|^2\big)^{\frac{p-2}{2}}
    |\xi - \xi_o|
    \quad\forall\, \xi,\xi_o\in\R^{Nn}.
\end{equation}
\end{lemma}

\begin{proof}
The first assertion \eqref{growth-Da} immediately follows from \eqref{def_asymp-} and the definition of $b$ in \eqref{def_asymp}, since
\begin{align*}
    |Da(\xi)|
    &\le
    |Da(\xi) - Db(\xi)| +
    |Db(\xi)| \\
    &\le
    \omega(|\xi|)(1+|\xi|)^{p-2} + \K(p-1) |\xi|^{p-2} \\
    &\le
    c(p)\,(\|\omega\|_\infty + \K)(1+|\xi|)^{p-2} .
\end{align*}
The second assertion \eqref{diff-a} now is a consequence of \eqref{growth-Da} and the following computation
\begin{align*}
    |a(\xi) - a(\xi_o)| 
    &=
    \bigg|\int_0^1 Da\big(\xi_o + s(\xi - \xi_o)\big) \ds\, (\xi - \xi_o)\bigg|
    \\
    &\le
    c(p)\,(\|\omega\|_\infty + \K) 
    \int_0^1 (1+|\xi_o + s(\xi - \xi_o)|)^{p-2} \ds\,
    |\xi - \xi_o| \\
    &\le
    c(p)\, (\|\omega\|_\infty + \K) 
    \big(1+|\xi_o|^2 + |\xi - \xi_o|^2\big)^{\frac{p-2}{2}}
    |\xi - \xi_o|.
\end{align*}
This completes the proof of the lemma.
\end{proof}

Next, we provide an auxiliary estimate for the vector field $b$ that will be useful several times later on.
\begin{lemma}\label{lem:diff-1}
For the vector field $b$ from \eqref{def_asymp} there holds 
\begin{align*}
    \big|b(\xi\,\Psi) - b(\xi_o\,\Psi) &-
    \big[b\big((\xi+G)\Psi\big) - b\big((\xi_o+G)\Psi\big)\big] \big|\\
    &\le
    c\,|G| \big(|\xi_o|^2 + |\xi-\xi_o|^2 + |G|^2\big)^{\frac{p-3}{2}} 
    |\xi-\xi_o| 
\end{align*}
for any $\xi,\xi_o,G\in\R^{Nn}$ and $\Psi\in\R^{n\cdot n}$ invertible and with a constant $c=c(p,\K,|\Psi|,$ $|\Psi^{-1}|)$.
We note that both sides of the inequality converge to zero when $G\to 0$.
\end{lemma}

\begin{proof}
In the case $p=2$ the expression under consideration is equal to zero, since then $b(\xi)\equiv\K\xi$. Therefore, it remains to consider the case $p>2$. Here, we first compute
\begin{align*}
    & \big|b(\xi\,\Psi) - b(\xi_o\,\Psi) -
    \big[b\big((\xi+G)\Psi\big) - b\big((\xi_o+G)\Psi\big)\big] \big|\\
    &\phantom{m}=
    \bigg|\int_0^1\int_0^1
    D^2 b\big((\xi_o + s(\xi-\xi_o) + \sigma G) \Psi\big)\big(G\Psi, (\xi-\xi_o)\Psi\big) 
    \ds\,d\sigma\, 
    \bigg|\\
    &\phantom{m}\le
    \int_0^1\int_0^1
    \big|D^2 b\big((\xi_o + s(\xi-\xi_o) + \sigma G)\Psi\big)\big| \ds\,d\sigma\, 
    |G\,\Psi||(\xi-\xi_o)\Psi| \\
    &\phantom{m}\le
    c(p)\,\K|\Psi|^2 \int_0^1\int_0^1
    |(\xi_o + s(\xi-\xi_o) + \sigma G)\Psi|^{p-3} \ds\,d\sigma\, |G||\xi-\xi_o| .
\end{align*}
To be precise, in the case $p\in(2,3)$ 
the third identity of the preceding computation needs to be justified, since the argument of $D^2b(\cdot)$ could be zero (see \cite[Remark 2.14]{Boegelein-habil} for the details).
Now, we continue estimating the right-hand side as follows, using Lemma \ref{lem:Fusco} twice if $p\in(2,3)$:
\begin{align*}
    & \big|b(\xi\,\Psi) - b(\xi_o\,\Psi) -
    \big[b\big((\xi+G)\Psi\big) - b\big((\xi_o+G)\Psi\big)\big] \big|\\
    &\phantom{m}\le
    c(p)\,\K|\Psi|^2 \int_0^1
    \big(|(\xi_o + s(\xi-\xi_o))\Psi|^2 + |G\Psi|^2\big)^{\frac{p-3}{2}} \ds\, 
    |G||\xi-\xi_o| \\
    &\phantom{m}\le
    c(p)\,\K|\Psi|^2 
    \big(|\xi_o\Psi|^2 + |(\xi-\xi_o)\Psi|^2 + |G\Psi|^2\big)^{\frac{p-3}{2}}  
    |G||\xi-\xi_o| \\
    &\phantom{m}\le
    c(p,\K,|\Psi|,|\Psi^{-1}|)\,|G|
    \big(|\xi_o|^2 + |\xi-\xi_o|^2 + |G|^2\big)^{\frac{p-3}{2}}  
    |\xi-\xi_o| .
\end{align*}
Note that in the case $p\in(2,3)$ we used in the last line the inequality $|A\Psi|^{-1}\le|\Psi^{-1}||A\Psi\Psi^{-1}|^{-1}=|\Psi^{-1}||A|^{-1}$ for $A\in\R^{Nn}$ .
This yields the result of the lemma. 
\end{proof}

The next lemma provides a useful estimate for the vector field $a$ that will be needed several times in the sequel.
\begin{lemma}\label{lem:diff-2}
Suppose that the vector field $a$ is asymptotically regular in the sense of \eqref{def_asymp}. Assume further that $\Psi\in C^{0;\beta}(B,\R^{n\cdot n})$ and $G\in C^{0;\beta}(B,\R^{Nn})$ with $\beta\in(0,1)$, $B\equiv B_\rho^+(x_o)$ and $\rho\in(0,1]$. Then, for any $x\in B$ and $\xi,\xi_o,\zeta\in\R^{Nn}$ there holds
\begin{align*}
	& \big|\big\langle a \big((\xi+G(x))\Psi(x)\big),
    \zeta\,\Psi(x)\big\rangle - 
    \big\langle a\big((\xi_o+G(x_o))\Psi(x_o)\big),
    \zeta\,\Psi(x_o)\big\rangle\big| \\
    &\phantom{mmmmmmm}\le
    c\,\Big[ \big(1+|\xi_o|^2 + |\xi-\xi_o|^2\big)^{\frac{p-2}{2}}|\xi-\xi_o| +
    \rho^\beta(1+|\xi_o|)^{p-1}\Big] |\zeta|
\end{align*}
for a constant $c$ depending on $p, \K, |a(0)|, \|\omega\|_\infty, \|\Psi\|_{C^{0;\beta}(B)}, \|G\|_{C^{0;\beta}(B)}$.
\end{lemma}

\begin{proof}
We first decompose the term under consideration as follows:
\begin{align*}
    \big|\big\langle a\big( & (\xi+G(x))\Psi(x)\big),
    \zeta\,\Psi(x)\big\rangle - 
    \big\langle a\big((\xi_o+G(x_o))\Psi(x_o)\big),
    \zeta\,\Psi(x_o)\big\rangle\big| \\
    &\le
    \big|\big\langle a\big((\xi+G(x))\Psi(x)\big),
    \zeta (\Psi(x)-\Psi(x_o))\big\rangle| \\
    &\phantom{\le\ }+
    \big|\big\langle a\big((\xi+G(x))\Psi(x)\big) - a\big((\xi+G(x))\Psi(x_o)\big),
    \zeta\, \Psi(x_o)\big\rangle \big| \\
    &\phantom{\le\ }+
    \big|\big\langle a\big((\xi+G(x))\Psi(x_o)\big) - a\big((\xi_o+G(x))\Psi(x_o)\big),
    \zeta\, \Psi(x_o)\big\rangle \big| \\
    &\phantom{\le\ }+
    \big|\big\langle a\big((\xi_o+G(x))\Psi(x_o)\big) - a\big((\xi_o+G(x_o))\Psi(x_o)\big),
    \zeta\, \Psi(x_o)\big\rangle \big| \\
    &=:
    \mbox{I} +\mbox{II} + \mbox{III} + \mbox{IV}
\end{align*}
with the obvious meaning of I -- IV.
For the estimate of I we use the H\"older continuity of $\Psi$ and \eqref{diff-a} from Lemma \ref{lem:growth-a} with $\xi_o=0$ to infer that 
\begin{align*}
    \mbox{I}
    &\le
    \rho^\beta \|\Psi\|_{C^{0;\beta}(B)}
    \big|a\big((\xi+G(x))\Psi(x)\big)\big| |\zeta| \\
    &\le
    \rho^\beta \|\Psi\|_{C^{0;\beta}(B)}
    \Big[\big|a\big((\xi+G(x))\Psi(x)\big)-a(0)\big|+|a(0)|\Big] |\zeta| \\
    &\le
    c\,\rho^\beta \|\Psi\|_{C^{0;\beta}(B)}
    \big(1 + |(\xi+G(x))\Psi(x)|^2\big)^{\frac{p-1}{2}} |\zeta| \\
    &\le
    c\,\rho^\beta \|\Psi\|_{C^{0;\beta}(B)}
    \big(1 + \|\Psi\|_{L^\infty(B)}^2\|G\|_{L^\infty(B)}^2 +|\xi|^2\big)^{\frac{p-1}{2}} |\zeta| \\
    &\le
    c\,\rho^\beta 
    \big(1 + |\xi|^2\big)^{\frac{p-1}{2}} |\zeta| \\
    &\le
    c\,\Big[|\xi-\xi_o|^{p-1} + 
    \rho^\beta(1+|\xi_o|)^{p-1}\Big] |\zeta| ,
\end{align*}
where $c=c(p, \K, |a(0)|, \|\omega\|_\infty, \|\Psi\|_{C^{0;\beta}(B)}, \|G\|_{L^\infty(B)})$.
Similarly, we get for II:
\begin{align*}
    \mbox{II}
    &\le
    \|\Psi\|_{L^\infty(B)}
    \big| a\big((\xi+G(x))\Psi(x)\big) - a\big((\xi+G(x))\Psi(x_o)\big) \big|
    |\zeta| \\
    &\le
    c\,
    \big(1 + |(\xi+G(x))\Psi(x)|^2 + |(\xi+G(x))(\Psi(x)-\Psi(x_o))|^2\big)^{\frac{p-2}{2}} 
    \\
    &\phantom{mmmmmmmmmmmmmmmmm}\cdot
    |(\xi+G(x))(\Psi(x)-\Psi(x_o))|
    |\zeta| \\
    &\le
    c\,\rho^\beta \|\Psi\|_{C^{0;\beta}(B)} 
    \big(1 + |\xi|^2\big)^{\frac{p-1}{2}} 
    |\zeta| \\
    &\le
    c\,\Big[|\xi-\xi_o|^{p-1} + \rho^\beta
    (1+|\xi_o|)^{p-1}\Big] |\zeta| ,
\end{align*}
where $c=c(p, \K, \|\omega\|_\infty, \|\Psi\|_{C^{0;\beta}(B)}, \|G\|_{L^\infty(B)})$.
Once again by \eqref{diff-a} we get
\begin{align*}
    \mbox{III}
    &\le
    \|\Psi\|_{L^\infty(B)}
    \big|a\big((\xi+G(x))\Psi(x_o)\big) - a\big((\xi_o+G(x))\Psi(x_o)\big)\big|
    |\zeta| \\
    &\le
    c\,
    \big(1 + |(\xi_o+G(x))\Psi(x_o)|^2 + |(\xi-\xi_o)\Psi(x_o)|^2\big)^{\frac{p-2}{2}} 
    |(\xi-\xi_o)\Psi(x_o)| |\zeta| \\
    &\le
    c\,
    \big(1+|\xi_o|^2 + |\xi-\xi_o|^2 \big)^{\frac{p-2}{2}} 
    |\xi-\xi_o| |\zeta| ,
\end{align*}
where $c=c(p, \K, \|\omega\|_\infty, \|\Psi\|_{L^\infty(B)}, \|Dg\|^2_{L^\infty(B)})$.
Finally, for the estimate of the term IV we use \eqref{diff-a}  and the H\"older continuity of $G$ which yields 
\begin{align*}
    \mbox{IV}
    &\le
    \|\Psi\|_{L^\infty(B)}
    \big|a\big((\xi_o+G(x))\Psi(x_o)\big) - a\big((\xi_o+G(x_o))\Psi(x_o)\big)\big|
    |\zeta| \\
    &\le
    c\,
    \big(1 + |(\xi_o+G(x))\Psi(x_o)|^2 + |(G(x) - G(x_o))\Psi(x_o)|^2 \big)^{\frac{p-2}{2}} 
    |(G(x) - G(x_o))\Psi(x_o)|
    |\zeta| \\
    &\le
    c\,\rho^\beta\,\|G\|_{C^{0;\beta}(B)}\|\Psi\|_{L^\infty(B)}
    \big(1 + |\xi_o|^2 \|\Psi\|_{L^\infty(B)}^2 + 
    \|G\|_{L^\infty(B)}^2 \|\Psi\|_{L^\infty(B)}^2 \big)^{\frac{p-2}{2}} 
    |\zeta| \\
    &\le
    c\,\rho^\beta
    (1+|\xi_o|)^{p-1} |\zeta| 
\end{align*}
with a constant $c=c(p, \K, \|\omega\|_\infty, \|\Psi\|_{L^\infty(B)}, \|G\|_{C^{0;\beta}(B)})$.
Inserting the preceding estimates for I -- IV above we deduce the desired estimate.
\end{proof}

The following lemma allows to compare the vector field $a$ to the $p$-Laplacian vector field $b$ provided certain quantities are large. The lemma is a modified version of
\cite[Lemmas 2.1 and 2.2]{Chipot-Evans:1986}, or
\cite[Lemma 5.1]{Giaquinta-Modica:1986}.

\begin{lemma}\label{lem:asymp}
Suppose that $a$ is asymptotically regular in the sense of \eqref{def_asymp}. For $\epsilon>0$ and $\omega$ as in Remark \ref{rem:asyp-equiv} we find $K_\epsilon\ge 1$ depending on $\epsilon$ and $\omega$ such that 
\begin{equation}\label{def-Keps}
    \omega(s)
    \le 
    \epsilon
    \qquad\forall\, s\ge K_\epsilon.
\end{equation}
There exists a constant $c=c(p)$ such that the following holds true: Supposed that $\delta\ge 0$ and $A\in\R^{Nn}$ satisfy
\begin{equation}\label{cond-asymp}
    |A| + \delta
    \ge
    \frac{8\|\omega\|_\infty K_\epsilon}{\epsilon}\,,
\end{equation}
then for any $\xi\in\R^{Nn}$ we have
\begin{equation}\label{comp-asymp}
    |a(A)-a(\xi) - [b(A)-b(\xi)]|
    \le
    c\,\epsilon
    (|\xi-A| + \delta)
    \big(1 + |A|^2 + |\xi-A|^2\big)^{\frac{p-2}{2}}.
\end{equation}
\end{lemma}

\begin{proof}
In the following we assume that $A\not=\xi$, since in the case $A=\xi$ estimate \eqref{comp-asymp} is trivially satisfied. We first define
\begin{equation*}
    I_\epsilon
    :=
    \{s\in[0,1]: |A + s(\xi-A)|\le K_\epsilon\}
\end{equation*}
and estimate the left-hand side of \eqref{comp-asymp} with the help of assumption \eqref{def_asymp-} by
\begin{align*}
    |a(A) & -a(\xi) - [b(A)-b(\xi)]| \\
    &=
    \bigg|\int_0^1\big[
    Da(A + s(\xi-A)) - Db(A + s(\xi-A)) \big]
    (\xi-A)\ds \bigg| \\
    &\le
    \int_0^1\omega\big(|A + s(\xi-A)|\big) \big(1 + |A + s(\xi-A)|\big)^{p-2}\ds\, |\xi-A| .
\end{align*}
Decomposing $[0,1]$ into the set $I_\epsilon$ and its complement $[0,1]\setminus I_\epsilon$ we get
\begin{align}\label{asymp-intermed}
    |a(A) & -a(\xi) - [b(A)-b(\xi)]| \\
    &\le
    c(p)\,\big(\|\omega\|_\infty |I_\epsilon| + \epsilon \big) \big(1+|A|^2 + |\xi-A|^2\big)^{\frac{p-2}{2}}\, |\xi-A| . \nn
\end{align}
It now remains to find a suitable bound for the measure of $I_\epsilon$. Initially, we find
\begin{equation*}
    |I_\epsilon|
    =
    \frac{1}{|\xi-A|}\,
    \big|\big\{s\in[0,|\xi-A|] :
    \big|A+s\tfrac{\xi-A}{|\xi-A|}\big|\le K_\epsilon\big\}\big|
    \le
    \frac{2K_\epsilon}{|\xi-A|}
\end{equation*}
which in the case $|\xi-A|\ge 2\|\omega\|_\infty K_\epsilon/\epsilon$ yields
\begin{equation}\label{I-est-1}
    |I_\epsilon|
    \le
    \frac{\epsilon}{\|\omega\|_\infty}\,.
\end{equation}
On the other hand, in the case 
$|\xi-A|<2\|\omega\|_\infty K_\epsilon/\epsilon$ we have with $s\in I_\epsilon$
\begin{equation*}
    |A|
    \le
    |A + s(\xi-A)| + |\xi-A|
    \le
    K_\epsilon + \frac{2\|\omega\|_\infty K_\epsilon}{\epsilon}
    \le
    \frac{4\|\omega\|_\infty K_\epsilon}{\epsilon},
\end{equation*}
since we may assume $\|\omega\|_\infty\ge 1$. By \eqref{cond-asymp} this implies
\begin{equation*}
    \delta
    \ge
    \frac{8\|\omega\|_\infty K_\epsilon}{\epsilon} - |A|
    \ge
    \frac{4\|\omega\|_\infty K_\epsilon}{\epsilon} 
\end{equation*}
and therefore
\begin{equation}\label{I-est-2}
    |I_\epsilon|
    \le
    \frac{2K_\epsilon}{|\xi-A|}
    \le
    \frac{\epsilon\, \delta}{2\|\omega\|_\infty |\xi-A|}\,.
\end{equation}
Combining \eqref{I-est-1} and \eqref{I-est-2} we obtain
\begin{equation*}
    |I_\epsilon|
    \le
    \frac{\epsilon}{\|\omega\|_\infty}\,
    \bigg(1+\frac{\delta}{|\xi-A|}\bigg).
\end{equation*}
Inserting this into \eqref{asymp-intermed} yields the asserted estimate \eqref{comp-asymp}.
\end{proof}

\section{Proof of the global Lipschitz regularity}

\subsection{Transformation to the model situation}\label{sec:model}
Since the proof of our Lipschitz regularity result is of local nature we can locally transform the problem to a model situation on a half-cylinder $Q_R^+$ for some $R>0$ and with boundary values zero on the lateral boundary $\Gamma_R$. The strategy will be outlined in the following.
Let $z_o=(x_o,t_o)\in \partial\Omega\times(0,T)$. Without loss of generality we can assume that $x_o=0$
and that the inward pointing unit normal to $\partial\Omega$ in $x_o$ is $\nu_{\partial\Omega}(x_o) = e_n$. Then, for $R>0$ sufficiently small, there exists a map $\Phi\colon B_R\cap\Omega\to B_R^+$ such that $\Phi(B_R\cap\partial\Omega)\subset D_R:=B_R \cap \{x\in\R^n: x_n=0\}$ with the properties that $\Phi$ and $\Phi^{-1}$ are of class $C^{1;\beta}$ and $\det D\Phi=1=\det D\Phi^{-1}$.
Next, we define the transformed maps
\begin{equation*}
	\hat g(y,t)
	:=
	g\big(\Phi^{-1}(y),t\big) ,\qquad (y,t)\in Q^+_R
\end{equation*}
and
$$
    v(y,t)
    :=
    u\big(\Phi^{-1}(y),t\big) - \hat g(y,t) ,\qquad (y,t)\in Q^+_R.
$$
Then, $v$ is a weak solution to the following Cauchy-Dirichlet problem
\begin{align}\label{transformation}
	\left\{
	\begin{array}{cl}
	\partial_t v -
    \Div \big[a\big((Dv+D\hat g)\Psi\big)\Psi^t\big]
	=
	\hat g_t
	\quad &
	\text{in } Q_R^+,\\[7pt]
	v=0
	\quad &
	\text{on } \Gamma_R,
	\end{array}
	\right.
\end{align}
where
\begin{equation*}
    \Psi(y)
    :=
	D\Phi\big(\Phi^{-1}(y)\big).
\end{equation*}

Now, it is easy to verify the fact that $y\in\Gamma_{R}$ is a regular point of $Dv$ if and only if
$\Phi^{-1}(y)\in\partial\Omega\times(0,T)$ is a regular point of $Du$. Therefore, it suffices to prove Theorem
\ref{thm:main-lip} in the model situation 
\begin{equation}\label{system-lat}
	\left\{
	\begin{array}{cl}
	\partial_t u -
    \Div \big[a\big((Du+Dg)\Psi\big)\Psi^t\big]
	=
	g_t
	\quad &
	\text{in } Q_R^+,\\[7pt]
	u=0
	\quad &
	\text{on } \Gamma_R,
	\end{array}
	\right.
\end{equation}
with a function 
\begin{equation}\label{Psi}
	\Psi\in C^{0;\beta}\big(\overline{B_R^+},\R^n\big) 
	\quad\mbox{such that $\Psi^{-1}\in L^\infty(B_R^+,\R^n)$ and $\Psi(0)=\mathbb I_{n\times n}$}
\end{equation}
and with 
\begin{equation}\label{g}
	Dg\in C^{0;\beta,0}\big(\overline{Q_R^+},\R^{Nn}\big) 
	\quad\mbox{and}\quad 
	\partial_t g\in L^{p',(1-\beta)p'}(Q_R^+;\R^N).
\end{equation}
Hence, Theorem \ref{thm:main-lip} is equivalent to the following 
\begin{proposition}\label{prop:main-lip}
There exists $R_o=R_o(n,N,p, \K, \psi_\beta, \beta,a_o, \|\omega\|_\infty, \mathcal G_\beta)>0$ such that the following holds: Let
$$
	u \in C^0\big(\Lambda_R;L^2(B_R^+,\R^{N})\big)
	\cap 
	L^p\big(\Lambda_R;W^{1,p}(B_R^+,\R^{N})\big)
$$ 
be a weak solution to the partial Cauchy-Dirichlet problem \eqref{system-lat} in $Q_R^+$ with $R\in(0,R_o]$ under the assumptions \eqref{def_asymp}, \eqref{Psi} and \eqref{g}. Then, $u$ is Lipschitz-continuous up to the boundary portion $\Gamma_R$, i.e.
$$
	Du\in L^\infty\big(Q_{R-\epsilon}^+,\R^{Nn}\big)
	\quad\mbox{and}\quad
	u\in C^{0;1,1/2}\big(Q_{R-\epsilon}^+ ,\R^{N}\big)
	\quad\mbox{for any $\epsilon\in(0,1)$.}
$$
Moreover, the
quantitative estimate
\begin{equation*}
    |Du(z_o)|
    \le
    c_1
    \bigg(\mint_{Q_{\rho}^+(z_o)}|Du|^p\dz\bigg)^\frac12 + 
    c_2
\end{equation*}
holds for a.e. $z_o\in Q_R^+$ and $\rho\in(0,R]$ such that $Q_{\rho}(z_o)\subset Q_R$. Thereby the constant $c_1$ depends only on 
$n,N,p,\K$, while $c_2$ depends on  
$n,N,p,\K,\psi_\beta, \beta, a_o, \omega(\cdot), \mathcal G_\beta$, where 
\begin{equation}\label{def-psibeta}
    a_o:=|a(0)|,
    \qquad
    \psi_\beta
    :=
    \|\Psi\|_{C^{0;\beta}(B_R^+)} + \|\Psi^{-1}\|_{L^\infty(B_R^+)} 
    \ge 1
\end{equation}
and 
\begin{equation}\label{def-gbeta}
    \mathcal G_\beta
    :=
    \|Dg\|_{C^{0;\beta,0}(Q_R^+)} + \|g_t\|_{L^{p',(1-\beta)p'}(Q_R^+)} .
\end{equation}
\end{proposition}

In the following we are concerned with the proof of Proposition \ref{prop:main-lip}. The weak form of the partial Cauchy-Dirichlet problem \eqref{system-lat} reads as follows:
\begin{equation}\label{system-lat-weak}
    \int_{Q_R^+} u\cdot\varphi_t -
    \big\langle a\big((Du+Dg)\Psi\big),D\varphi\, \Psi\big\rangle \dz
    =
    \int_{Q_R^+} g_t\cdot \varphi \dz
    \quad \forall\, \varphi\in C_0^\infty(Q_R^+,\R^N).
\end{equation}
Before we start with the proof of Proposition \ref{prop:main-lip} we introduce the following notation that will be used in the rest of the paper.
For $\epsilon\in(0,1]$ we set
\begin{equation}\label{def-Geps}
    \boldsymbol G(\epsilon)
    :=
    \frac{2^{11}\,\psi_\beta}{\epsilon}\Big[\|\omega\|_\infty K_{\epsilon} +
    \|Dg\|_{L^\infty(Q_R^+)}\Big]
    \ge 1,
\end{equation}
where $\psi_\beta$ is defined in \eqref{def-psibeta} and $K_\epsilon$ is chosen according to \eqref{def-Keps}.
Note that $\boldsymbol G$ is decreasing, i.e. $\boldsymbol G(\epsilon) \le \boldsymbol G(\tilde\epsilon)$ whenever $\tilde\epsilon\le\epsilon$.

\subsection{The non-degenerate regime}\label{ndreg}
In the final proof of the gradient $L^\infty$-bound we will distinguish at a certain point whether the solution behaves on a cylinder approximatively like a solution of a linear parabolic system -- this we call the \textit{non-degenerate regime} -- or like a solution of the parabolic $p$-Laplacian system -- this we call the \textit{degenerate regime}. In the present section we start considering the non-degenerate regime. 
The main result is Proposition \ref{prop:iter-a-cal} which states that once a cylinder $Q_{\rho,\lambda}(z_o)$ belongs to the non-degenerate regime (which is characterized by \eqref{cond-dec-a-1-} and \eqref{cond-dec-a-2-}) and $z_o$ is a Lebesgue point of $Du$ it already follows that $|Du(z_o)|$ is bounded by $2\lambda$.

\subsubsection{Caccioppoli and Poincar\'e inequalities}
Here, we provide Caccioppoli and Poincar\'e type inequalities that will be needed later on to prove the decay-estimate for the non-degenerate regime. 
We start with the following Caccioppoli inequality. For the definition and basic properties of the $V$-function we refer to Appendix \ref{sec:V}.

\begin{lemma}\label{lem:cac}
Let $R\in(0,1]$. Under the assumptions of Proposition \ref{prop:main-lip} there exist constants $\epsilon=\epsilon(n,N,p,\K,\psi_\beta)\in(0,1]$ and $c=c(n,N,p,\K,\psi_\beta,a_o,\|\omega\|_\infty,\mathcal G_\beta)\ge 1$ such that for any cylinder $Q_{\rho,\lambda}(z_o)\subset Q_R$
with $(x_o)_n\ge 0$ and $\lambda\ge 1$
and any affine function $\ell\colon\R^n\to\R^N$ independent of $t$ satisfying
\begin{equation}\label{cond-cac}
    |D\ell|
    \ge
    2^{-8}\,\boldsymbol G(\epsilon) ,
    \qquad
    |D\ell|
    \ge
    2^{-7}\,\lambda
\end{equation}
and $\ell\equiv 0$ on $\Gamma$ if $Q_{\rho,\lambda}(z_o)\cap \Gamma\not=\emptyset$ there holds
\begin{align*}
	\mint_{Q_{\rho/2,\lambda}^+(z_o)} 
    	|V_{|D\ell|}(Du-D\ell)|^2 \dz 
	\le
	c \mint_{Q_{\rho,\lambda}^+(z_o)}
	\Big|V_{|D\ell|}\Big(\frac{u-\ell}{\rho}\Big)
    	\Big|^2 \dz +
    	c\, \rho^{\beta} |D\ell|^{p} .
\end{align*}
\end{lemma}

\begin{remark}\upshape
Later on, we will apply the Caccioppoli inequality in two different situations, namely in the interior situation where $Q_{\rho,\lambda}(z_o)\subset Q_R^+$ and in the half space situation where $z_o\in \Gamma_R$. In the latter case we choose $\ell$ of the form $\zeta x_n$ with $\zeta\in\R^N$. Then, $\ell\equiv 0$ on $\Gamma$ and $D\ell=\zeta\otimes e_n$.
\hfill$\Box$
\end{remark}

\begin{proof}
For convenience in notation we abbreviate $B\equiv B_{\rho}^+(x_o)$ and $Q\equiv Q_{\rho,\lambda}^+(z_o)$.
In the following  we shall proceed formally concerning the use of the time derivative $\partial_t u$. The arguments can be made rigorous by use of a smoothing procedure in time, as for instance via Steklov averages.  Next, we choose a cut-off function $\eta\in C_0^1(B_{\rho}(x_o),[0,1])$ in space and $\zeta\in C^{1}(\R,[0,1])$ in time, such that $\eta\equiv 1$ on $B_{\rho/2}(x_o)$ and $|D\eta|\le 4/\rho$ and
$\zeta\equiv 0$ on $(-\infty,t_o-\lambda^{2-p}\rho^2)$, $\zeta\equiv 1$ on 
$(t_o-\lambda^{2-p}(\rho/2)^2, \infty)$ and $0\le\zeta_t\le 2\lambda^{p-2}/\rho^2$.
Moreover, for $\theta>0$ we define $\chi_\theta\in W^{1,\infty}(\R)$ by
\begin{equation*}
	\chi_\theta(t)
	:=
	\begin{cases}
		1 & 
		\mbox{if $t\in (-\infty,t_o+\lambda^{2-p}\rho^2-\theta]$}\\[3pt]
		\tfrac{1}{\theta}(t_o+\lambda^{2-p}\rho^2-t) & 
		\mbox{if $t\in (t_o+\lambda^{2-p}\rho^2-\theta,t_o+\lambda^{2-p}\rho^2)$}\\[3pt]
		0 & 
		\mbox{if $t\in [t_o+\lambda^{2-p}\rho^2,\infty).$}
	\end{cases}
\end{equation*}
and choose in the weak formulation \eqref{system-lat-weak} of the parabolic system the testing-function
$$\varphi_\theta(x,t):=\chi_\theta(t) \zeta(t)^2\eta(x)^p \big(u(x,t)-\ell(x)\big).$$
We thus obtain
\begin{align*}
    & \int_{Q} 
    \langle b(Du\,\Psi) - b(D\ell\,\Psi),D\varphi_\theta\rangle\dz \\
    &\phantom{m}=
    \int_{Q}
    u\cdot\partial_t\varphi_\theta  +
    \big\langle b(Du\,\Psi) - b(D\ell\,\Psi) - a\big((Du+Dg)\Psi\big),D\varphi_\theta\,\Psi\big\rangle -
    g_t\cdot\varphi_\theta\dz .
\end{align*}
Using assumption \eqref{cond-cac}$_2$ we have $0\le\zeta_t\le 2\lambda^{p-2}/\rho^2\le 2^{7p} |D\ell|^{p-2}/\rho^2$ and hence
\begin{align*}
    \int_{Q} u\cdot\partial_t & \varphi_\theta \dz
    =
    \int_{Q} (u-\ell)\cdot\partial_t \varphi_\theta \dz 
    =
    \frac12 \int_{Q} |u-\ell|^2 \eta^p \partial_t(\chi_\theta\zeta^2) \dz \\
    &=
    \frac12 \int_{Q} |u-\ell|^2 \eta^p\chi_\theta \partial_t \zeta^2 \dz -
    \frac{1}{2\theta} \int_{t_o+\lambda^{2-p}\rho^2-\theta}^{t_o+\lambda^{2-p}\rho^2} \int_{B_{\rho}^+(x_o)} |u-\ell|^2 \eta^p \zeta^2 \dx\dt \\
    &\le
    c(p)\,|D\ell|^{p-2}\int_{Q} \Big|\frac{u-\ell}{\rho}\Big|^2 \dz
    \le
    c(p) \mint_{Q}
	\Big|V_{|D\ell|}\Big(\frac{u-\ell}{\rho}\Big)
    	\Big|^2 \dz.
\end{align*}
Recalling the definition of the vector-field $b$ from \eqref{def_asymp} and Lemma \ref{lem:monotone} (i), i.e.
\begin{align*}
    \int_{Q} &
    \big\langle b(Du\,\Psi) - b(D\ell\,\Psi),(Du-D\ell)\Psi \big\rangle \eta^p\zeta^2\dz \\
    &\ge
    \frac{\K}{c(n,N,p)}
    \int_{Q}
    |V_{|D\ell\Psi|}(Du\Psi-D\ell\Psi)|^2 \eta^p\zeta^2\dz \\
    &\ge
    \frac{\K }{c(n,N,p)\|\Psi^{-1}\|_{L^\infty}^{p-2}}
    \int_{Q}
    |V_{|D\ell|}(Du-D\ell)|^2 \eta^p\zeta^2\dz
\end{align*}
we obtain in the limit $\theta\downarrow 0$ that
\begin{align}\label{split-cac}
    \mathbf L
    &:=
    \int_{Q} 
    |V_{|D\ell|}(Du-D\ell)|^2 \eta^p\zeta^2\dz \nn\\
    &\phantom{:}\le
    c(n,N,p,\K, \psi_\beta)\bigg[
    \int_{Q}
	\Big|V_{|D\ell|}\Big(\frac{u-\ell}{\rho}\Big)
    	\Big|^2 \dz +
    \mbox{I} + \mbox{II} + \mbox{III} + \mbox{IV}  + \mbox{V}\bigg],
\end{align}
where
\begin{align*}
    &\mbox{I}
    := 
    -\int_{Q}
    \big\langle b(Du\,\Psi) - b(D\ell\,\Psi),D\eta^p\otimes(u-\ell)\Psi\big\rangle\zeta^2\dz \\
    &\mbox{II}
    :=
    \int_{Q}
    \big\langle b(Du\,\Psi) - b(D\ell\,\Psi) -
    \big[b\big((Du+Dg)\Psi\big) - b\big((D\ell+Dg)\Psi\big)\big], D\varphi\,\Psi\big\rangle\dz \\
    &\mbox{III}
    :=
    \int_{Q}
    \big\langle b\big((Du+Dg)\Psi\big) - b\big((D\ell+Dg)\Psi\big) \\
    &\phantom{mmmmmmmmmmmmmm} - 
    \big[a\big((Du+Dg)\Psi\big) - a\big((D\ell+Dg)\Psi\big)\big],
    D\varphi\,\Psi\big\rangle\dz \\
    &\mbox{IV}
    :=
    \int_{Q}
    \big\langle a\big((D\ell+Dg)\Psi\big),D\varphi\,\Psi\big\rangle \dz \\
    &\mbox{V}
    :=
    -
    \int_{Q}
    g_t\cdot\varphi\dz
\end{align*}
and
$$\varphi(x,t):=\zeta(t)^2\eta(x)^p \big(u(x,t)-\ell(x)\big).$$
We now in turn estimate the terms I -- V.
Using the fact that $|D\eta|\le 4/\rho$ and applying Lemma \ref{lem:monotone} (ii)
we obtain 
\begin{align}\label{def-rest}
    |\,\mbox{I}\,|
    &\le
    c \int_{Q}
    \big(|D\ell|^2 + |Du-D\ell|^2\big)^{\frac{p-2}{2}} |Du-D\ell|\, 
    \Big|\frac{u-\ell}{\rho}\Big|\eta^{p-1}\zeta^2\dz 
    =:
    c\,\mathbf R,
\end{align}
where $c=c(n,N,p,\K,\psi_\beta)$.
Before we come to the estimate of II we first note that \eqref{cond-cac}$_1$ and the definition of $\boldsymbol G(\epsilon)$ from \eqref{def-Geps} imply that
\begin{equation*}
    |Dg(z)|
    \le
    \|Dg\|_{L^\infty(Q_R^+)}
    \le
    2^{-11}\epsilon\, \boldsymbol G(\epsilon)
    \le
    \epsilon|D\ell|\le |D\ell|
     \qquad \mbox{for a.e. $z\in {Q}$.}
\end{equation*}
Together with Lemma \ref{lem:diff-1} we therefore obtain
\begin{align*}
    \big|(Du\,\Psi) - b(D\ell & \,\Psi) -
    \big[b\big((Du+Dg)\Psi\big) - b\big((D\ell+Dg)\Psi\big)\big] \big|\\
    &\le
    c\, |Dg|\big(|D\ell|^2 + |Du-D\ell|^2 + |Dg|^2\big)^{\frac{p-3}{2}} |Du-D\ell| \\
    &\le
    c\,\epsilon|D\ell| \big(|D\ell|^2 + |Du-D\ell|^2\big)^{\frac{p-3}{2}} |Du-D\ell| \\
    &\le
    c(p,\K,\psi_\beta)\,\epsilon
    \big(|D\ell|^2 + |Du-D\ell|^2 \big)^{\frac{p-2}{2}} |Du-D\ell|.
\end{align*}
Inserting this estimate into II, recalling the definition of $\varphi$, using Young's inequality and the fact that $\epsilon\le 1$ we find
\begin{align*}
    \mbox{II}
    &\le
    c\,\epsilon\int_{Q}
    \big(|D\ell|^2 + |Du-D\ell|^2 \big)^{\frac{p-2}{2}} |Du-D\ell||D\varphi| \dz \\
    &\le
    c\,\epsilon\int_{Q}
    \big(|D\ell|^2 + |Du-D\ell|^2 \big)^{\frac{p-2}{2}} |Du-D\ell|
    \bigg[|Du-D\ell|\eta^p + \Big|\frac{u-\ell}{\rho}\Big| \eta^{p-1} \bigg]\zeta^2 \dz \\
    &\le 
    c\,\epsilon\,\mathbf L +
    c\,\mathbf R,
\end{align*}
where $c=c(p,\K,\psi_\beta)$, $\mathbf L$ is defined in \eqref{split-cac} and $\mathbf R$ in \eqref{def-rest}.
For the estimate of III we first apply Lemma \ref{lem:asymp} with $(A,\xi,\delta)$ replaced by $((D\ell+Dg(z))\Psi(z),(Du(z)+Dg(z))\Psi(z),0)$ for any $z\in Q$. Note that hypothesis \eqref{cond-asymp} is satisfied due to the assumption \eqref{cond-cac}$_1$ and the following computation
\begin{align*}
	|(D\ell+Dg(z))\Psi(z)|
	&\ge
	\psi_\beta^{-1} |D\ell+Dg(z)|
	\ge
	\psi_\beta^{-1} (|D\ell| - |Dg(z)|) \\
	&\ge
	\psi_\beta^{-1} \big(2^{-8}\boldsymbol G(\epsilon) - \|Dg\|_{L^\infty(Q_R^+)}\big) \\
	&=
	\psi_\beta^{-1} \bigg[
	\frac{8\,\psi_\beta}{\epsilon}\Big[\|\omega\|_\infty K_{\epsilon} +
    \|Dg\|_{L^\infty(Q_R^+)}\Big]- \|Dg\|_{L^\infty(Q_R^+)}\bigg] \\
	&\ge
	\frac{8\|\omega\|_\infty K_\epsilon}{\epsilon}\, .
\end{align*}
In this way we obtain
\begin{align*}
    \mbox{III}
    &\le
    c(p)\,\epsilon \int_{Q}
    |(Du-D\ell)\Psi|\big(1+|(D\ell+Dg)\Psi|^2 + |(Du-D\ell)\Psi|^2\big)^{\frac{p-2}{2}} 
    |D\varphi\,\Psi| \dz .
\end{align*}
Subsequently we use that $|D\ell|\ge 2^{-8}\boldsymbol G(\epsilon)\ge 1$ which is a consequence of hypothesis \eqref{cond-cac}$_1$ and then proceed as in the estimate of II. This leads us to 
\begin{align*}
    \mbox{III}
    \le
    c\,\epsilon \int_{Q}
    |Du-D\ell|\big(|D\ell|^2 + |Du-D\ell|^2\big)^{\frac{p-2}{2}} 
    |D\varphi| \dz 
    \le
    c\,\epsilon\,\mathbf L +
    c\,\mathbf R,
\end{align*}
where $c=c(p,\psi_\beta,\mathcal G_\beta)$.
At this point we estimate the remainder $\mathbf R$ in I -- III. With the help of Young's inequality we get for any $\delta>0$ that
\begin{align*}
    \mathbf R
    &\le
    \delta\,\mathbf L +
    \frac{1}{\delta}\int_{Q} 
    \big(|D\ell|^2 + |Du-D\ell|^2\big)^{\frac{p-2}{2}}  
    \Big|\frac{u-\ell}{\rho}\Big|^2\eta^{p-2}\zeta^2 \dz \\
    &\le
    \delta\,\mathbf L +
    \delta \int_{Q} |Du-D\ell|^{p} \eta^{p}\zeta^2 \dz +
    \frac{c(p)}{\delta^{2}}\int_{Q} \bigg[
    |D\ell|^{p-2} \Big|\frac{u-\ell}{\rho}\Big|^2 +
    \Big|\frac{u-\ell}{\rho}\Big|^p \bigg]\dz   \\  
    &\le
    2\delta\,\mathbf L +
    \frac{c(p)}{\delta^{2}}\int_{Q}
    \Big|V_{|D\ell|}\Big(\frac{u-\ell}{\rho}\Big)\Big|^2\dz .
\end{align*}
For the estimate of IV we first subtract the term 
\begin{align*}
    \int_{Q}
    \big\langle a\big((D\ell+Dg(x_o,t))\Psi(x_o)\big),D\varphi\,\Psi(x_o)\big\rangle
    \dz 
    \equiv
    0
\end{align*}
and then
apply Lemma \ref{lem:diff-2} with $\xi,\xi_o=D\ell$ and $G=Dg(\cdot,t)$ slice wise on $B\times\{t\}$ for a.e. $t\in\Lambda_{\rho,\lambda}(t_o)$. Subsequently we use the fact that 
$|D\ell|\ge 2^{-8}\boldsymbol G(\epsilon)\ge 1$ which is a consequence of assumption \eqref{cond-cac}$_1$ and Young's inequality to infer that
\begin{align*}
    \mbox{IV}
    &=
    \int_{Q}
    \big\langle a\big((D\ell+Dg)\Psi\big),D\varphi\,\Psi\big\rangle -
    \big\langle a\big((D\ell+Dg(x_o,t))\Psi(x_o)\big),D\varphi\,\Psi(x_o)\big\rangle
    \dz \\
    &\le
    c\,\rho^\beta (1+|D\ell|)^{p-1}
    \int_{Q}|D\varphi|\dz \\
    &\le
    c\,\rho^\beta |D\ell|^{p-1}
    \int_{Q}|Du-D\ell|\eta^p\zeta^2 + 
    \Big|\frac{u-\ell}{\rho}\Big| \eta^{p-1} \zeta^2 \dz \\
    &\le
    \delta \int_{Q}|D\ell|^{p-2}|Du-D\ell|^2\eta^p\zeta^2 \dz +
    \int_{Q}|D\ell|^{p-2}\Big|\frac{u-\ell}{\rho}\Big|^2 \dz +
    c\,\delta^{-1}\rho^{2\beta} |D\ell|^p |{Q}|\\
    &\le
    \delta\,\mathbf L + 
    \int_{Q}
    \Big|V_{|D\ell|}\Big(\frac{u-\ell}{\rho}\Big)\Big|^2\dz +
    c\,\delta^{-1}\rho^{2\beta} |D\ell|^{p}|{Q}|,
\end{align*}
where $c=c(p,\K, \psi_\beta,a_o,\|\omega\|_\infty,\mathcal G_\beta)$.
Finally, we estimate V with Young's inequality:
\begin{align*}
    \mbox{V}
    \le
    \int_{Q} |u-\ell| |g_t| \dz
    \le
    \int_{Q} \Big|\frac{u-\ell}{\rho}\Big|^p \dz +
    \rho^{p'}\int_{Q} |g_t|^{p'} \dz .
\end{align*}
Using the fact that $Q\equiv Q_{\rho,\lambda}^+(z_o)\subset Q_\rho(z_o)\cap Q_R^+$ (recall that $\lambda\ge 1$) we obtain for the second term on the right-hand side
\begin{align}\label{morrey-est}
    \rho^{p'}\int_{Q} |g_t|^{p'} \dz 
    &\le
    \rho^{p'}\int_{Q_\rho(z_o)\cap Q_R^+} |g_t|^{p'} \dz 
    \le
    \rho^{n+2+\beta p'} \|g_t\|^{p'}_{L^{p',(1-\beta)p'}(Q_R^+)} \nn\\
    &\le
    c(n)\,\mathcal G_\beta \rho^{\beta p'} \lambda^{p-2} |Q| 
    \le
    c(n)\,\mathcal G_\beta \rho^{\beta} |D\ell|^p |Q|,
\end{align}
where in the last line we also used $\lambda\ge 1$, $\rho\le 1$ and assumption \eqref{cond-cac}$_2$. This leads us to
\begin{align*}
    \mbox{V}
    \le
    \int_{Q} \Big|V_{|D\ell|}\Big(\frac{u-\ell}{\rho}\Big)\Big|^2 \dz +
    c(n)\,\mathcal G_\beta \rho^{\beta} |D\ell|^p |Q|,
\end{align*}
Joining the preceding estimates for I -- V with \eqref{split-cac} we arrive at
\begin{align*}
    \mathbf L
    &\le
    c_1(\delta + \epsilon)
    \mathbf L +
    \frac{c_2}{\delta^{2}}\int_{Q}
    \Big|V_{|D\ell|}\Big(\frac{u-\ell}{\rho}\Big)\Big|^2 +
    \frac{c_2}{\delta}\,\rho^{\beta} |D\ell|^{p} |Q| ,
\end{align*}
where $c_1=c_1(n,N,p,\K,\psi_\beta)$ and $c_2=c_2(n,N,p,\K,\psi_\beta,a_o,\|\omega\|_\infty,\mathcal G_\beta)$. Choosing $\delta=\epsilon=1/(4c_1)$ we can absorb the first integral of the right-hand side into the left. Note that this amounts in a dependence of $\epsilon$ on $n,N,p,\K,\psi_\beta$.
Finally, taking into account that $\eta^p\zeta^2\equiv 1$ on $Q_{\rho/2,\lambda}^+(z_o)$ and taking mean values we obtain the desired Caccioppoli inequality.
\end{proof}

In the next lemma we provide a Poincar\'e type inequality on interior cylinders for solutions to the parabolic system \eqref{system-lat}. The strategy is to apply the usual Poincar\'e inequality on time-slices $B\times\{t\}$ with respect to the spacial variable $x$. For the time-direction this is not allowed since we do not know that the time-derivative exists in a certain Sobolev-space. Therefore, we shall utilize the parabolic system which provides some regularity in time. More precisely, we can show that the weighted means in space are absolutely continuous with respect to time. This will be enough to prove the Poincar\'e inequality. 

\begin{lemma}\label{lem:poin-inter}
Let $R\in(0,1]$ and $\epsilon\in(0,1]$ and suppose that the assumptions of Proposition \ref{prop:main-lip} are in force.
Then, for any $s\in[1,p]$, any parabolic cylinder $Q_{\rho,\lambda}(z_o)\subset Q_R^+$ with $\rho\in(0,1]$, $\lambda\ge 1$ and any $A\in\R^{Nn}$ satisfying
\begin{equation}\label{cond-poin}
    |A|
    \le
    2\lambda
\end{equation}
there holds
\begin{align*}
    & \mint_{Q_{\rho,\lambda}(z_o)} 
    |u - (u)_{z_o;\rho,\lambda} - A(x-x_o)|^s\dz \nonumber\\
    &\phantom{mm}\le
    c\,\rho^s 
    \mint_{Q_{\rho,\lambda}(z_o)} |Du-A|^s dz +
    c\,\rho^s \bigg[
    \lambda^{2-p} 
    \mint_{Q_{\rho,\lambda}(z_o)} |Du-A|^{p-1} dz +
    \rho^{\beta} \lambda\bigg]^s
\end{align*}
for a constant 
$c=c(n,N,p, \K, \psi_\beta, a_o, \|\omega\|_\infty, \mathcal G_\beta)$.
\end{lemma}

\begin{proof}
In the following, we abbreviate $B:=B_\rho(x_o)$, $\Lambda:=\Lambda_{\rho,\lambda}(t_o)$ and $Q:=Q_{\rho,\lambda}(z_o)$, where $z_o=(x_o,t_o)$, so that $Q=B\times\Lambda$. For the construction of the weighted means we fix a nonnegative weight-function $\eta \in C_0^{\infty}(B)$ satisfying
\begin{equation*}
  \eta\ge 0, \qquad
  \mint_{B} \eta \dx=1\ \ \text{ and }\quad
  \|\eta\|_{\infty} + \rho\|D\eta\|_{\infty}\le c_{\eta} .
\end{equation*}
Note that $c_{\eta}$ depends on $n$ only.
Then, for a.e. $t\in\Lambda$ the weighted means of $u(\cdot,t)$ on $B$ are defined by
\begin{equation*}
  	(u)_{\eta}(t) 
	:=
	\mint_{B} u(\cdot,t)\, \eta \dx .
\end{equation*}	
Now, we decompose
\begin{align}\label{poin-ineq-1}
    \mint_{Q} & |u - (u)_{Q} - A(x-x_o)|^s\dz \\
	&\le
	3^{s-1}\bigg[ \mint_\Lambda \mint_{B} |u(x,t) - (u)_\eta(t) - A(x-x_o)|^s \dx\dt
    \nonumber\\
    &\phantom{\le3^{s-1}\bigg[\ }+
    \mint_\Lambda\mint_\Lambda
	|(u)_{\eta}(t) - (u)_{\eta}(\tau)|^s \dt\dtau +
    \bigg| \mint_\Lambda (u)_{\eta}(\tau)\dtau - (u)_{Q} \bigg|^s\bigg] \nn\\
    &=:
    3^{s-1}\big[\mbox{I} + \mbox{II} + \mbox{III}\big] , \nn
\end{align}
with the obvious meaning of I -- III.
To estimate I, we apply Poincar\'e's inequality slicewise with respect to $x$ for a.e. $t\in\Lambda$ to $u(\cdot,t)-(u)_\eta(t)$, which yields
\begin{equation*}
  \mbox{I}
  \le
  c(n,s)\,\rho^s \mint_{Q} |Du-A|^s\dz.
\end{equation*}
The estimate  for III is similar, since $\mbox{III}\le \mbox{I}$. It therefore remains to estimate II.
Here we use the fact that $u$ is a solution of the parabolic system \eqref{system-lat}. To be more precise, we start with its Steklov-formulation
\begin{equation*}
    \int_{B_R^+} \partial_t u_h(\cdot,t)\cdot\varphi - 
    \big\langle \big[a\big((Du+Dg)\Psi\big)\big]_h(\cdot,t),D\varphi\, \Psi\big\rangle
    \dx
    =
    \int_{B_R^+} [g_t]_h(\cdot,t)\cdot \varphi \dx
\end{equation*}
for any $\varphi\in C_0^\infty(B_R^+,\R^N)$, $0<|h|\le R^2$ and 
for a.e. $t\in\Lambda_R$. 
Thereby, the Steklov-mean $[f]_h$ of a function $f \in L^1(Q_R^+)$ is defined for $0<|h|\le R^2$ by
\begin{equation*}
	[f]_h(x,t) \equiv
	\left\{
	\begin{array}{cl}
		\displaystyle{\frac{1}{|h|} \int_t^{t+h} f(x,s) \,ds ,}
		& t\in [-R^2+|h|,R^2-|h|] , \\[7pt]
		0 ,
		& t\in (-R^2,-R^2+|h|)\cup (R^2-|h|,R^2) 
\,.
	\end{array}
	\right.
\end{equation*}
Now, for a.e. $t\in\Lambda\subset\Lambda_R$ and $i\in\{1,\dots,N\}$ we choose the test-function
$\varphi\colon\R^{n+1}\to\R^N$ with $\varphi(x,t)=\eta(x) e_i$ where $e_1,\dots,e_N$ denotes the standard basis in $\R^N$.
Integrating the result with respect to $t$ over $(t_1,t_2)\subset\Lambda$ yields
\begin{align*}
     & \big|\big([u_i]_h\big)_{\eta} (t_2) - \big([u_i]_h\big)_{\eta} (t_1)\big|
    =
    \bigg|\int_{t_1}^{t_2} \partial_t \big([u_i]_h\big)_{\eta} \dt \bigg| 
    =
    \bigg|\int_{t_1}^{t_2} \mint_{B} \partial_t [u_i]_h\cdot \eta \dx\dt \bigg| \\
    &\phantom{mmmmmmmmmm}=
    \bigg|\int_{t_1}^{t_2} \mint_{B} 
    \big\langle \big[a\big((Du+Dg)\Psi\big)\big]_h,D\eta\otimes e_i \Psi\big\rangle +
    \eta[\partial_t g_i]_h  \dx\dt\bigg| .
\end{align*}
Passing to the limit $h\downarrow 0$ and 
using that $(t_1,t_2)\subset\Lambda$ and $|\Lambda|=2\lambda^{2-p} \rho^2$
we find 
\begin{align*}
    |(u_i)_{\eta} (t_2) - (u_i)_{\eta} (t_1)| 
    &\phantom{i}\le
    c\,\lambda^{2-p} \rho^2 \bigg[
    \mint_{Q} |g_t| \dz +
    \bigg|\mint_{Q}\big\langle a\big((Du{+}Dg)\Psi\big), D\eta\otimes e_i \Psi \big\rangle
	\dz\bigg| \bigg]\\
    &\phantom{i} =
    c\,\lambda^{2-p} \rho^2 \big[\mbox{II}_1 +\mbox{II}_2 \big]
\end{align*}
with the obvious meaning of II$_1$ and II$_2$.
For the estimate of II$_1$ we use $\|\eta\|_\infty\le c_\eta$, H\"older's inequality and the second last estimate in \eqref{morrey-est} to infer that
\begin{align*}
    \mbox{II}_1
    \le
    c \bigg(\mint_{Q} |g_t|^{p'} \dz \bigg)^{\frac{1}{p'}}
    \le
    c\,\rho^{\beta-1}\lambda^{p-2},
\end{align*}
where $c=c(n,\mathcal G_\beta)$.
Next, we use that $a((A+Dg(x_o,t))\Psi(x_o))$ is constant with respect to $x$ and apply Lemma \ref{lem:diff-2} with $(G,\xi,\xi_o,\zeta)$ replaced by $(Dg(\cdot,t),Du,A,D\eta\otimes e_i)$. Subsequently recalling that 
$\lambda\ge 1$, $|A|\le 2\lambda$ and $\|D\eta\|_\infty\le c_\eta/ \rho$ we obtain  
\begin{align*}
	\mbox{II}_2
	&\le 
	\mint_{Q}
	\big|\big\langle a\big((Du{+}Dg)\Psi\big), D\eta\otimes e_i \Psi \big\rangle -
    \big\langle a\big((A{+}Dg(x_o,t))\Psi(x_o)\big), D\eta\otimes e_i \Psi(x_o)\big\rangle 	\big| \dz \\ 
	&\le
	c\mint_{Q}
   \Big[ \big(1+|A|^2 + |Du-A|^2\big)^{\frac{p-2}{2}} |Du-A| +
   \rho^\beta(1+|A|)^{p-1}\Big] |D\eta| \dz \\
	&\le 
	c\,\rho^{-1}\lambda^{p-2} \mint_{Q} |Du-A| \dz +
   c\,\rho^{-1} \mint_{Q}|Du-A|^{p-1} \dz +
   c\,\rho^{\beta-1} \lambda^{p-1} ,
\end{align*}
with a constant $c=c(n,p, \K, \psi_\beta, a_o, \|\omega\|_\infty, \mathcal G_\beta)$.
Inserting the preceding estimates for II$_1$ and II$_2$ above, summing over $i=1,\dots,N$ and using again that $\lambda\ge 1$ we obtain
\begin{align*}
   |(u)_{\eta} (t_2) - (u)_{\eta} (t_1)| \le
   c\,\rho \mint_{Q} |Du-A| \dz + 
   c\,\lambda^{2-p} \rho \mint_{Q}
   |Du-A|^{p-1} \dz +
	c\,\rho^{1+\beta} \lambda
\end{align*}
for a.e. $t_1,t_2\in\Lambda$, and this together with H\"older's inequality leads us to
\begin{align*}
    \mbox{II}
    &\le
    \sup_{t_1,t_2\in\Lambda}
    |(u)_{\eta} (t_2) - (u)_{\eta} (t_1)|^s \\
    &\le
    c\,\rho^s \mint_{Q} |Du-A|^s \dz + 
    c\,\rho^s \bigg(\lambda^{2-p} \mint_{Q} |Du-A|^{p-1} \dz\bigg)^s 	+
    c\,\rho^{s(1+\beta)} \lambda^s
\end{align*}
for a constant 
$c=c(n,N,p, \K, \psi_\beta, a_o, \|\omega\|_\infty, \mathcal G_\beta)$.
Combining the  estimates for I -- III with
(\ref{poin-ineq-1}), we obtain the desired Poincar\'e type inequality.
\end{proof}

When considering the boundary situation 
a Poincar\'e inequality for general maps $u\in L^p(\Lambda_{\rho,\lambda}(t_o);W^{1,p}(B_\rho(x_o)^+,\R^N))$ satisfying
$u\equiv0$ on the lateral boundary $\Gamma_{\rho,\lambda}(z_o)$ holds. This inequality can be obtained applying 
\cite[Corollary 4.5.3]{Ziemer} for a.e. $t\in\Lambda_{\rho,\lambda}(t_o)$
and then integrating with respect to $t$. Precisely, we have the following

\begin{lemma}\label{lem:poin-boundary}
Let $s>1$ and $Q_{\rho,\lambda}(z_o)\subset\R^{n+1}$ be a parabolic cylinder with $0\le (x_o)_n<\rho/2$.
Then, for any map
$u\in L^p(\Lambda_{\rho,\lambda}(t_o);W^{1,p}(B_\rho(x_o)^+,\R^N))$ satisfying $u\equiv 0$ on $\Gamma_{\rho,\lambda}(z_o)$
there holds
\begin{equation*}
	\mint_{Q_{\rho,\lambda}^+(z_o)} |u|^s \dz
	\le
	c(n,N,s)\, \rho^s \mint_{Q_{\rho,\lambda}^+(z_o)} |D u|^s \dz .
\end{equation*}
\end{lemma}

\subsubsection{Approximate $\mathcal A$-caloricity}
Given a parabolic cylinder $Q_{\rho,\lambda}(z_o)\subset Q_R$ with $\rho,\lambda>0$ and $(x_o)_n\ge 0$ and $A\in\R^{Nn}$ we define the
\textit{excess functional} by
\begin{equation}\label{def-excess}
    \Phi_\lambda(z_o,\rho,A)
    =
    \mint_{Q_{\rho,\lambda}^+(z_o)} |V_{|A|}(Du-A)|^2 \dz.
\end{equation}
In the next lemma we prove that the solution $u$ approximately satisfies a linear system, provided the excess is small. Later on, this will be the starting point to prove excess-decay estimates for the non-degenerate regime.

\begin{lemma}\label{lem:approx-a-cal}
Let $R\in(0,1]$ and $\epsilon_1,\epsilon_2\in(0,1]$ and suppose that the assumptions of Proposition \ref{prop:main-lip} are in force.
Then, for any cylinder $Q_{\rho,\lambda}(z_o)\subset Q_R$, with $(x_o)_n\ge 0$, $\lambda\ge \boldsymbol G(\epsilon_1)$
and any $A\in\R^{Nn}$ satisfying
\begin{equation}\label{cond-a-cal-2}
    |A|
    \ge
    2^{-7}\lambda,
\end{equation}
and
\begin{equation}\label{cond-a-cal-1}
    \Phi_\lambda(z_o,\rho,A)
    \le
    \epsilon_2^p\,|A|^p\,,
\end{equation}
there holds
\begin{align*}
    \bigg|\mint_{Q_{\rho,\lambda}^+(z_o)}  & 
    u\cdot\partial_t \varphi - 
    Db(A) (Du - A,D\varphi)\dz
    \bigg| \\
    &\le
    c\, |A|^{\frac{p-2}{2}} 
    \Big[\big(\epsilon_1+\sqrt{\epsilon_2}+R^\beta\big) \sqrt{\Phi_\lambda(z_o,\rho,A)} +
    \rho^\beta |A|^{\frac{p}{2}}\Big]
    \sup_{Q_{\rho,\lambda}^+(z_o)} |D\varphi|,
\end{align*}
for all $\varphi\in C_0^1(Q_{\rho,\lambda}^+(z_o) ,\R^N)$ and with a constant $c=c(n,p,\K,\psi_\beta,a_o, \|\omega\|_\infty,\mathcal G_\beta)$.
\end{lemma}

\begin{proof}
Without loss of generality we may assume that $\|D\varphi\|_\infty\le 1$ and we abbreviate $Q\equiv Q_{\rho,\lambda}^+(z_o)$ and $\Phi_\lambda\equiv \Phi_\lambda(z_o,\rho,A)$.
From the weak formulation \eqref{system-lat-weak} of the parabolic system we obtain
\begin{equation}\label{a-cal-split}
    \mint_{Q} 
    u \cdot\partial_t \varphi - Db(A) (Du - A,D\varphi)\dz 
    =
    \mbox{I}+\mbox{II}+\mbox{III}+\mbox{IV}+\mbox{V},
\end{equation}
where 
\begin{align*}
    \mbox{I}
    &:=
    \mint_{Q} \Big[
    Db(A)\big((Du - A)\Psi(x_o),D\varphi\,\Psi(x_o)\big) -
    Db(A) (Du - A,D\varphi)\Big]\dz \\
    \mbox{II}
    &:=
    \mint_{Q} \int_0^1
    \Big[Db\big((A+s(Du-A)+Dg(x_o,t))\Psi(x_o)\big) - Db(A)\Big] ds\, \\
    &\phantom{mmmmmmmmmmmmmmmmmmmmmmi} \cdot
    \big((Du - A)\Psi(x_o),D\varphi\,\Psi(x_o)\big)\dz \\
    \mbox{III}
    &:=
    \mint_{Q}
    \big\langle a\big((Du+Dg(x_o,t))\Psi(x_o)\big) - a\big((A+Dg(x_o,t))\Psi(x_o)\big) \nn\\
    &\phantom{mmn}- 
    \big[b\big((Du+Dg(x_o,t))\Psi(x_o)\big) - b\big((A+Dg(x_o,t))\Psi(x_o)\big)\big],
    D\varphi\,\Psi(x_o)\big\rangle \dz \nn\\
    \mbox{IV}
    &:=
    \mint_{Q}
    \big\langle a\big((Du+Dg)\Psi\big),D\varphi\,\Psi\big\rangle -
    \big\langle a\big((Du+Dg(x_o,t))\Psi(x_o)\big),D\varphi\,\Psi(x_o)\big\rangle \dz \nn\\
    \mbox{V}
    &:=
    \mint_{Q}
    g_t\cdot \varphi \dz .
\end{align*}
Here, we used in the term III that 
\begin{align*}
    \mint_{Q}
    \big\langle a\big((A+Dg(x_o,t))\Psi(x_o)\big), D\varphi\,\Psi(x_o)\big\rangle \dz 
    =0.
\end{align*}
Moreover, we have the following auxiliary estimate which is a consequence of assumption \eqref{cond-a-cal-2} and the definition of $\boldsymbol G(\epsilon_1)$:
\begin{equation}\label{Dg}
    |Dg(z)|
    \le
    \|Dg\|_{L^\infty(Q_R^+)}
    \le
    \frac{\epsilon_1}{2^{11}} \boldsymbol G(\epsilon_1)
    \le
    \frac{\epsilon_1}{2^{11}} \lambda
    \le
    \epsilon_1|A|
    \le 
    |A|
    \quad \mbox{for a.e. $z\in Q$.}
\end{equation}
Now, we in turn estimate the terms I -- V.
For I we use \eqref{Psi} and the definition \eqref{def_asymp} of $b$ to infer that
\begin{align*}
    \mbox{I}
    &=
    \mint_{Q} 
    Db(A)\big((Du - A)(\Psi(x_o)-\mathbb I_{n\times n}),D\varphi\,\Psi(x_o)\big) \dz \\
    &\phantom{\le\ }+
    \mint_{Q}
    Db(A) \big(Du - A,D\varphi(\Psi(x_o)-\mathbb I_{n\times n})\big)\dz \\
    &\le
    c\,R^\beta|A|^{p-2}
    \mint_{Q} |Du-A|\dz 
    \le
    c(p,\psi_\beta)\,R^\beta|A|^{\frac{p-2}{2}} \sqrt{\Phi_\lambda}.
\end{align*}
Concerning the term II we first note that $\mbox{II}=0$ in the case $p=2$. In the case $p>2$ we abbreviate 
\begin{align*}
	\mathfrak X_s(x,t)
	:=
    A(\Psi(x_o)-\mathbb I_{n\times n})+Dg(x_o,t)\Psi(x_o) +s(Du(x,t)-A)\Psi(x_o)
\end{align*}
and then use definition \eqref{def_asymp} of $b$ to estimate for $s\in[0,1]$:
\begin{align*}
	&\big|Db\big((A+s(Du-A)+Dg(x_o,t))\Psi(x_o)\big)-Db(A)\big| 
	=
    \bigg|\int_0^1
    D^2 b\big(A+\sigma\mathfrak X_s\big)\,d\sigma\ \mathfrak X_s \bigg| 
    \nn\\
    &\phantom{mmmmmmmmmi}\le
    c \int_0^1
    |A+\sigma\mathfrak X_s|^{p-3}\,d\sigma \ |\mathfrak X_s|
    \le
    c(p,\K) 
    \big(|A|^2+|\mathfrak X_s|^2\big)^{\frac{p-3}{2}}\, |\mathfrak X_s| ,
\end{align*}
where for the last step we have also used Lemma \ref{lem:Fusco} if $p\in(2,3)$.
To be precise, this computation has to be justified, since the argument of $D^2b(\cdot)$ could become zero (see \cite[Remark 2.14]{Boegelein-habil} for the details).
Now, we integrate with respect to $s\in[0,1]$ and subsequently use Lemma \ref{lem:Fusco} if $p\in(2,3)$ as well as \eqref{Dg} to infer that
\begin{align*}
	\bigg|\int_0^1 & Db\big((A+s(Du-A)+Dg(x_o,t))\Psi(x_o)\big)-Db(A)\ds \bigg| \\
    &\le
    c \int_0^1 \big(|A|^2+|\mathfrak X_s|^2\big)^{\frac{p-3}{2}}\, |\mathfrak X_s| \ds
    \\
    &\le
    c\, \big(|A|^2+|A(\Psi(x_o)-\mathbb I_{n\times n})+Dg(x_o,t)\Psi(x_o)|^2 +
    |(Du-A)\Psi(x_o)| ^2\big)^{\frac{p-3}{2}} \\
    &\phantom{mmmmmmm}\cdot
    \big(|A(\Psi(x_o)-\mathbb I_{n\times n})+Dg(x_o,t)\Psi(x_o)| +|(Du-A)\Psi(x_o)| \big)
    \\
    &\le
    c \Big[ |A|^{p-3}\big(|A(\Psi(x_o)-\mathbb I_{n\times n})+Dg(x_o,t)\Psi(x_o)| +|(Du-A)\Psi(x_o)| \big) \\
    &\phantom{\le c\Big[} +
    \chi_{p>3}\big(|A(\Psi(x_o)-\mathbb I_{n\times n})+Dg(x_o,t)\Psi(x_o)| +
    |(Du-A)\Psi(x_o)|\big)^{p-2}\Big] \\
    &\le
    c\Big[(R^\beta+\epsilon_1)|A|^{p-2} + 
    |A|^{p-3}|Du-A| +
    \chi_{p>3}\big((R^\beta+\epsilon_1)|A| + |Du-A|\big)^{\frac{p-2}{2}} \Big]\\
    &\le
    c\Big[|A|^{p-3}|Du-A|+
    (\epsilon_1+R^\beta)|A|^{p-2} +
    \chi_{p>3}|Du-A|^{p-2} \Big],
\end{align*}
where $\chi_{p>3}=1$ if $p>3$ and $\chi_{p>3}=0$ if $p\le 3$ and $c=c(p,\K,\psi_\beta)$.
This together with H\"older's inequality and \eqref{cond-a-cal-1} yields
\begin{align*}
    |\,\mbox{II}\,|
    &\le
    c\,
    \mint_{Q} |A|^{p-3}|Du-A|^2 +
    (\epsilon_1+R^\beta) |A|^{p-2}|Du-A| +
    \chi_{p>3} |Du-A|^{p-1}\dz \\
    &\le
    c\,|A|^{-1}
    \Phi_\lambda +
    c\,(\epsilon_1+R^\beta)|A|^{\frac{p-2}{2}}
    \sqrt{\Phi_\lambda} +
    c\,\chi_{p>3}\Phi_\lambda^{1-\frac{1}{p}} \\
    &\le
    c\,\Big[\epsilon_2^{\frac{p}{2}} +\epsilon_1+R^\beta+\chi_{p>3}\epsilon_2^{\frac{p-2}{2}}\Big] 
    |A|^{\frac{p-2}{2}}\sqrt{\Phi_\lambda} \\
    &\le
    c(p,\K)\,\big(\epsilon_1+\epsilon_2^\frac12+R^\beta\big) \, |A|^{\frac{p-2}{2}}
    \sqrt{\Phi_\lambda}.
\end{align*}
For the term III we infer a pointwise bound of the integrand with the help of Lemma \ref{lem:asymp} applied with $(A,\xi,\delta,\epsilon)$ replaced by 
$((A+Dg(x_o,t))\Psi(x_o), (Du(z)+Dg(x_o,t))\Psi(x_o),0,\epsilon_1)$. Note that this is possible due to assumption \eqref{cond-a-cal-2} and the fact that $\lambda\ge\boldsymbol G(\epsilon_1)$. In this way we obtain
\begin{align*}
    |\,\mbox{III}\,|
    &\le
    c\,\epsilon_1
    \mint_{Q}
    |(Du{-}A)\Psi(x_o)| \big(1+|A{+}Dg(x_o,t)\Psi(x_o)|^2 + |(Du{-}A)\Psi(x_o)|^2\big)^{\frac{p-2}{2}} \dz
\end{align*}
for a constant $c=c(p)$. 
In turn using \eqref{Dg}, the fact that $|A|\ge 1$ which is a consequence of \eqref{cond-a-cal-2} and the fact that $\lambda\ge\boldsymbol G(\epsilon_1)$, H\"older's inequality and \eqref{cond-a-cal-1} we find 
\begin{align*}
    |\,\mbox{III}\,|
    &\le
    c\,\epsilon_1
    \mint_{Q}
    \big(|A|^2 + |Du-A|^2\big)^{\frac{p-2}{2}}
    |Du-A| \dz \nn\\
    &\le
    c\,\epsilon_1 \sqrt{\Phi_\lambda(\rho)}
    \bigg(\mint_{Q}
    \big(|A|^2 + |Du-A|^2\big)^{\frac{p-2}{2}} \dz\bigg)^\frac12 \\
    &\le
    c\,\epsilon_1 \sqrt{\Phi_\lambda(\rho)}
    \Big[|A|^{\frac{p-2}{2}} +
    \Phi_\lambda(\rho)^{\frac{p-2}{2p}}\Big] \\
    &\le
    c(p,\psi_\beta)\,\epsilon_1|A|^{\frac{p-2}{2}} \sqrt{\Phi_\lambda(\rho)}.
\end{align*}
From Lemma \ref{lem:diff-2}, the fact that $|A|\ge 1$ and \eqref{cond-a-cal-1} we obtain
\begin{align*}
    |\,\mbox{IV}\,|
    &\le
    c\,\rho^\beta \mint_Q (1+|Du|)^{p-1} \dz 
    \le
    c\,\rho^\beta \mint_Q (|A|+|Du-A|)^{p-1} \dz  \\
    &\le
    c\,\rho^\beta \Big[|A|^{p-1} + \Phi_\lambda^{1-\frac{1}{p}}\Big] 
    \le
    c\,\rho^\beta |A|^{p-1} ,
\end{align*}
where $c=c(p,\K,\psi_\beta,a_o, \|\omega\|_\infty,\mathcal G_\beta)$.
Finally, we estimate the last term in \eqref{a-cal-split} with the help of H\"older's and Poincar\'e's inequality, the second last inequality in \eqref{morrey-est}, the fact that $|D\varphi|\le 1$ and hypothesis \eqref{cond-a-cal-2} as follows
\begin{align*}
	|\,\mbox{V}\,|
    &\le
    \bigg(\mint_{Q} |\varphi|^p\dz \bigg)^\frac1p
    \bigg(\mint_{Q} |g_t|^{p'} \dz \bigg)^{\frac{1}{p'}}  
    \le
    c\,\rho \bigg(\mint_{Q} |D\varphi|^p\dz \bigg)^\frac1p
    \mathcal G_\beta^{\frac{1}{p'}} \rho^{\beta-1} \lambda^{\frac{p-2}{p'}} \nn\\
    &\le
    c\,\rho^\beta \lambda^{p-2}
    \le
    c\,\rho^\beta \lambda^{p-1}
    \le
    c\,\rho^\beta |A|^{p-1} ,
\end{align*}
where $c=c(n,p,\mathcal G_\beta)$.
Joining the estimates for I -- V with \eqref{a-cal-split} we deduce the assertion of the lemma.
\end{proof}

\begin{remark}\label{rem:a-cal-p2}\upshape
In the case $p=2$ assumption \eqref{cond-a-cal-1} in the statement of Lemma \ref{lem:approx-a-cal} -- which characterizes the non-degenerate regime -- is not necessary. This can be seen from the proof of the lemma as follows: assumption \eqref{cond-a-cal-1} is needed only in the estimate of the term I which is zero in the case $p=2$.
This considerably simplifies the proof for the case $p=2$ since then the degenerate regime considered in Section \ref{sec:degenerate} is not necessary.
\hfill$\Box$
\end{remark}

\subsubsection{Boundedness of the gradient in the non-degenerate regime}
In the following lemma we provide a first excess-decay estimate for the non-degenerate regime which is characterized by (\ref{cond-dec-a-2}) below.
Here and in the following we recall the definitions and notations for the relevant affine functions from Section \ref{sec:affine}.

\begin{lemma}\label{lem:decay-a-cal}
There exist constants
$R_o\in(0,1]$, 
$\theta\in(0,8^{-p/\beta}]$, $\epsilon_1\in(0,1]$, $\epsilon_2\in(0,\theta^{n+2}]$ and $\bar c\ge 1$ depending on 
$n,N,p,\K, \psi_\beta, \beta, a_o, \|\omega\|_\infty, \mathcal G_\beta$
such that under the assumptions of Proposition \ref{prop:main-lip} with some $R\in(0,R_o]$ the following holds: Whenever $\lambda\ge \boldsymbol G(\epsilon_1)$ and
$Q_{\rho,\lambda}(z_o)\subset Q_R$ is a cylinder with $(x_o)_n\ge 0$ on which
\begin{equation}\label{cond-dec-a-1}
    2^{-7}\lambda
    \le
    |D\ell_{z_o;\rho,\lambda}|
    \le
    2 \lambda
\end{equation}
and
\begin{equation}\label{cond-dec-a-2}
    \Phi_\lambda\big(z_o,\rho,D\ell_{z_o;\rho,\lambda}\big)
    \le
    \epsilon_2^p\, |D\ell_{z_o;\rho,\lambda}|^p,
\end{equation}
hold, then we have the following excess-decay estimate:
\begin{equation}\label{decay-a-cal}
    \Phi_\lambda\big(z_o,\theta\rho,D\ell_{z_o;\theta\rho,\lambda}\big)
    \le
    \tfrac12 \theta^{\beta}
    \Big[\Phi_\lambda\big(z_o,\rho,D\ell_{z_o;\rho,\lambda}\big) +
    \bar c\,\rho^{\beta} |D\ell_{z_o;\rho,\lambda}|^{p}
    \Big] .
\end{equation}
\end{lemma}

\begin{proof}
For convenience of the reader we use the shorter notation $\ell_{\rho,\lambda}$ and $\Phi_\lambda(\rho)$ for $\ell_{z_o;\rho,\lambda}$ and $\Phi_\lambda(z_o,\rho,D\ell_{z_o;\rho,\lambda})$. In the following, we assume that $\Phi_\lambda(\rho)>0$ since otherwise the conclusion of the lemma holds trivially.
Let $R_o\in (0,1]$ and $\theta\in(0,8^{-p/\beta}]$ to be fixed later, and assume that 
$\lambda\ge \boldsymbol G(\epsilon_1)$ for some $\epsilon_1\in(0,1]$. Moreover, assume that \eqref{cond-dec-a-2} is valid for some $\epsilon_2\in(0,\theta^{n+2}]$. The precise values of
$\epsilon_1,\epsilon_2$ will be determined in the course of the proof.
Note that \eqref{cond-dec-a-1} implies $|D\ell_{\rho,\lambda}|\ge 2^{-7}\lambda\ge 2^{-7}\boldsymbol G(\epsilon_1)\ge 1$.   Moreover, we let $\epsilon>0$ (to be specified later) and $\delta=\delta(n,N,p,$ $\nu\equiv 2^{-7(p-2)}\K,L\equiv 2^{p-2}p\K,\epsilon)
=\delta(n,N,p,\K,\epsilon)\in(0,1]$ be minimum of the constants from Lemma \ref{lem:a-cal} and Lemma \ref{lem:a-cal-boundary}.
We define
\begin{equation*}
    \tilde v(x,t)
    :=
    u(x,t) - \ell_{\rho,\lambda}(x)
    \qquad\text{for } (x,t)\in Q_{\rho,\lambda}^+(z_o)
\end{equation*}
and the rescaled function
\begin{equation}\label{def-w}
    v(x,t)
    :=
    \frac{\tilde v(x,\lambda^{2-p} t)}{c_1 \gamma\, |D\ell_{\rho,\lambda}|}
    \qquad\text{for } (x,t)\in Q_\rho^+(z_o)\equiv Q_{\rho,1}^+(z_o),
\end{equation}
where $c_1\ge 1$ will be specified later and
\begin{equation*}
    \gamma
    :=
    \bigg(\frac{\Phi_\lambda(\rho) + \delta^{-2}\rho^{2\beta} |D\ell_{\rho,\lambda}|^{p}}
    {|D\ell_{\rho,\lambda}|^p}\bigg)^\frac12.
\end{equation*}
Note that \eqref{cond-dec-a-1}, \eqref{cond-dec-a-2}, $\rho\le R_o$ and the additional assumption
\begin{equation}\label{cond-rho}
    \epsilon_2^p +
    \delta^{-2}R_o^{2\beta}
    \le
    1
\end{equation}
imply that $\gamma\le 1$.
In the following we will ensure that the assumptions of either the interior, or the boundary version of the $\mathcal A$-caloric approximation lemma are satisfied.
Due to assumptions \eqref{cond-dec-a-1}, \eqref{cond-dec-a-2} and $\lambda\ge \boldsymbol G(\epsilon_1)$ we are allowed to apply Lemma \ref{lem:approx-a-cal} which yields 
\begin{align*}
    \bigg|\mint_{Q_{\rho,\lambda}^+(z_o)} &
    \tilde v\cdot\partial_t \varphi - 
    Db(D\ell_{\rho,\lambda}) 
    (D\tilde v,D\varphi)\dz\bigg| \\
    &\le
    c\, |D\ell_{\rho,\lambda}|^{\frac{p-2}{2}} 
    \Big[\big(\epsilon_1+\sqrt{\epsilon_2}+R^\beta\big) \sqrt{\Phi_\lambda(\rho)} +
    \rho^\beta |D\ell_{\rho,\lambda}|^{\frac{p}{2}}
    \Big]
    \sup_{Q_{\rho,\lambda}^+(z_o)}|D\varphi| ,
\end{align*}
for any $\varphi\in C_0^1(Q_{\rho,\lambda}^+(z_o) ,\R^N)$. 
Note that
$c=c(n,p,\K,\psi_\beta,a_o, \|\omega\|_\infty,\mathcal G_\beta)$.
Recalling the definition of $v$ and scaling to the cylinder $Q_{\rho}^+(z_o)$ this inequality can be rewritten as 
\begin{align*}
    &\bigg|\mint_{Q_{\rho}^+(z_o)} 
    v\cdot\partial_t \varphi - 
    \frac{Db(D\ell_{\rho,\lambda})}{\lambda^{p-2}} 
    (Dv,D\varphi)\dz\bigg| \nonumber\\
    &\qquad \le
    \frac{c\,|D\ell_{\rho,\lambda}|^{\frac{p}{2}-2}}{c_1\gamma\, \lambda^{p-2}}
    \Big[\big(\epsilon_1+\sqrt{\epsilon_2}+R^\beta\big) \sqrt{\Phi_\lambda(\rho)} +
    \rho^\beta |D\ell_{\rho,\lambda}|^{\frac{p}{2}}
    \Big]
    \sup_{Q_{\rho}^+(z_o)}|D\varphi|\nonumber\\
    &\qquad \le \frac{c}{c_1}\,
    \big[\epsilon_1+\sqrt{\epsilon_2}+R^\beta + \delta \big]
    \sup_{Q_{\rho}^+(z_o)}|D\varphi|,
\end{align*}
for any $\varphi\in C_0^1(Q_{\rho}^+(z_o) ,\R^N)$.
In the last line we have also used assumption
\eqref{cond-dec-a-1}.
Choosing $c_1\ge 2c$ in dependence on $n,p, \K, \psi_\beta, a_o, \|\omega\|_\infty, \mathcal G_\beta$ large enough and assuming
\begin{equation}\label{cond-eps0}
    \epsilon_1+\sqrt{\epsilon_2}+R_o^\beta
    \le
    2^{-(n+3)}\delta
\end{equation}
we obtain
\begin{equation}\label{approx}
    \bigg|\mint_{Q_{\rho}^+(z_o)} 
    v\cdot\partial_t \varphi - 
    \frac{Db(D\ell_{\rho,\lambda})}{\lambda^{p-2}} 
    (Dv,D\varphi)\, dz\bigg| 
    \le
    2^{-(n+3)}\delta \sup_{Q_{\rho}^+(z_o)}\! |D\varphi|.
\end{equation}
Next, we use the definition of $\Phi_\lambda(\rho)$ and $\gamma$ to infer for $s\in\{2,p\}$ that 
\begin{align}\label{small-Dv-s}
    \mint_{Q_{\rho,\lambda}^+(z_o)} |D\tilde v|^s\dz 
    \le
    |D\ell_{\rho,\lambda}|^{s-p}\Phi_\lambda(\rho) 
    \le
    \gamma^2 |D\ell_{\rho,\lambda}|^s 
\end{align}
%
With the help of H\"older's inequality and the fact that $\gamma\le 1$ we conclude from \eqref{small-Dv-s} that
\begin{align}\label{small-Dv-3}
    \mint_{Q_{\rho,\lambda}^+(z_o)}\! |D\tilde v|^{p-1} \dz 
    &\le
    \gamma^{2-\frac2p} |D\ell_{\rho,\lambda}|^{p-1} 
    \le
    \gamma |D\ell_{\rho,\lambda}|^{p-1} .
\end{align}

In the following we have to distinguish the cases $(x_o)_n\ge \rho/2$ and $(x_o)_n< \rho/2$. In the first case when $(x_o)_n\ge \rho/2$ we apply the Poincar\'e type inequality from Lemma \ref{lem:poin-inter} on $Q_{\rho/2,\lambda}(z_o)\subset Q_R^+$, which is allowed due to assumption \eqref{cond-dec-a-1}. 
In this way we infer for $s\in\{2, p\}$ that
\begin{align*}
    \mint_{Q_{\rho/2,\lambda}(z_o)} 
    \Big|\frac{\tilde v}{\rho/2}\Big|^s\dz 
    &\le
    c \mint_{Q_{\rho/2,\lambda}(z_o)}|D\tilde v|^s\dz +
    c \bigg[\lambda^{2-p}\mint_{Q_{\rho/2,\lambda}(z_o)}
    |D\tilde v|^{p-1}\dz  +
    \rho^{\beta} \lambda\bigg]^s
    \nn\\
    &\le
    c \mint_{Q_{\rho,\lambda}^+(z_o)}|D\tilde v|^s\dz +
    c \bigg[\lambda^{2-p}\mint_{Q_{\rho,\lambda}^+(z_o)}
    |D\tilde v|^{p-1}\dz  +
    \rho^{\beta} \lambda\bigg]^s,
\end{align*}
where $c=c(n,N,p,\K,\psi_\beta,a_o,\|\omega\|_\infty, \mathcal G_\beta)$.
In order to further estimate the right-hand side we use \eqref{small-Dv-s} and \eqref{small-Dv-3}.
Subsequently, we use hypothesis \eqref{cond-dec-a-1} as well as $\rho^\beta\le\gamma$ and $\gamma\le 1$. This leads us to
\begin{align}\label{est-poin}
    \mint_{Q_{\rho/2,\lambda}(z_o)} 
    \Big|\frac{\tilde v}{\rho/2}\Big|^s\dz 
    &\le
    c\, \gamma^2|D\ell_{\rho,\lambda}|^{s}  +
    c\Big[\lambda^{2-p} \gamma |D\ell_{\rho,\lambda}|^{p-1}
    +
    \rho^{\beta} \lambda \Big]^s \\
    &\le
    c\, \gamma^2|D\ell_{\rho,\lambda}|^{s} +
    c\Big[\gamma |D\ell_{\rho,\lambda}|
    +
    \rho^{\beta} |D\ell_{\rho,\lambda}| \Big]^s \nn\\
    &\le
    c\, \gamma^2|D\ell_{\rho,\lambda}|^{s} ,\nn
\end{align}
where $c=c(n,N,p, \K, \psi_\beta, a_o, \|\omega\|_\infty, \mathcal G_\beta)$.
From the definition of $v$ and \eqref{small-Dv-s} and \eqref{est-poin} applied with $s=2$ and $s=p$ we infer that
\begin{equation}\label{small-v}
    \mint_{Q_{\rho/2}^+(z_o)} 
    \Big|\frac{v}{\rho/2}\Big|^2 + |Dv|^2 \dz +
    \gamma^{p-2} \mint_{Q_{\rho/2}^+(z_o)}
    \Big|\frac{v}{\rho/2}\Big|^p + |Dv|^p \dz 
    \le
    \frac{c}{c_1^2} +
    \frac{c}{c_1^p} 
    \le
    1, 
\end{equation}
provided we have chosen $c_1\gg 1$ large enough, in dependence of $n,N,p, \K,\psi_\beta, a_o,$ $\|\omega\|_\infty, \mathcal G_\beta$. 
Note that in combination with the conditions from above $c_1$ can be chosen in dependence on $n,N,p, \K, \psi_\beta, a_o, \|\omega\|_\infty, \mathcal G_\beta$.
Finally, we observe that \eqref{approx} and $|Q_{\rho/2}(z_o)|/|Q_{\rho}^+(z_o)|\le 2^{n+3}$ imply that
\begin{equation}\label{approx-int}
    \bigg|\mint_{Q_{\rho/2}(z_o)} 
    v\cdot\partial_t \varphi - 
    \frac{Db(D\ell_{\rho,\lambda})}{\lambda^{p-2}} 
    (Dv,D\varphi)\, dz\bigg| 
    \le
    \delta \sup_{Q_{\rho/2}(z_o)} |D\varphi|
\end{equation}
holds for any $\varphi\in C_0^1(Q_{\rho/2}(z_o) ,\R^N)$. 
Now, we define
\begin{equation}\label{def-A}
    \mathcal A(\xi,\xi_o)
    :=
    \frac{Db(D\ell_{\rho,\lambda}) (\xi,\xi_o)}{\lambda^{p-2}}
    \qquad\text{for }\ \xi, \xi_o\in\R^{Nn}.
\end{equation}
From the definition of $b$ in \eqref{def_asymp} and hypothesis \eqref{cond-dec-a-1} we see that $\mathcal A$ satisfies the following
ellipticity and growth conditions:
\begin{equation}\label{ellip-bound-a}
    \mathcal A(\xi,\xi)
    \ge 
    2^{-7(p-2)}\K\,|\xi|^2,
    \qquad
    |\mathcal A(\xi,\xi_o)|\le 2^{p-2}p\K\, |\xi| |\xi_o|,
\end{equation}
for any $\xi,\xi_o\in\R^{Nn}$. At this stage we want to apply the $\mathcal A$-caloric approximation lemma \ref{lem:a-cal}
to $(v,\mathcal A)$; note that $(\nu,L)$ must be replaced by $(2^{-7(p-2)}\K, 2^{p-2}p\K)$ here.
Then, that the hypothesis of the lemma are satisfied due to \eqref{small-v} and \eqref{approx-int}.
The application of Lemma \ref{lem:a-cal} yields the existence of an $\mathcal A$-caloric function $h\in L^p(\Lambda_{\rho/4}(t_o);W^{1,p}(B_{\rho/4}(x_o),\R^N))$ on $Q_{\rho/4}(z_o)$ satisfying 
\begin{align*}
    \mint_{Q_{\rho/4}(z_o)}\Big|\frac{h}{\rho/4}\Big|^2 + |Dh|^2 \dz +
    \gamma^{p-2} \mint_{Q_{\rho/4}(z_o)}
    \Big|\frac{h}{\rho/4}\Big|^p + |Dh|^p \dz
    \le
    c(n,p)
\end{align*}
and
\begin{align*}
    \mint_{Q_{\rho/4}(z_o)}\Big|\frac{v-h}{\rho/4}\Big|^2 +
    \gamma^{p-2}\Big|\frac{v-h}{\rho/4}\Big|^p \dz
    \le
    \epsilon.
\end{align*}

Next, we consider the case $(x_o)_n< \rho/2$.
Here, we use the the definition of $v$ and inequality \eqref{small-Dv-s} applied with $s=2$ and $s=p$ to infer that
\begin{align*}
    \mint_{Q_{\rho}^+(z_o)} |Dv|^2 \dz +
    \gamma^{p-2} \mint_{Q_{\rho}^+(z_o)} |Dv|^p \dz 
    \le
    \frac{1}{c_1^2} + \frac{1}{c_1^p}
    \le 
    \frac{2}{c_1^2}, \nn
\end{align*}
since $c_1\ge 1$. 
Applying the Poincar\'e inequality from Lemma \ref{lem:poin-boundary} which is possible since $v\equiv0$ on $\Gamma_\rho(z_o)$ by the definition of $\ell_{\rho,\lambda}$ and using the last estimate we obtain
\begin{align}\label{small-v-boundary}
    \mint_{Q_{\rho}^+(z_o)} &
    \Big|\frac{v}{\rho}\Big|^2 + |Dv|^2 \dz +
    \gamma^{p-2} \mint_{Q_{\rho}^+(z_o)}
    \Big|\frac{v}{\rho}\Big|^p + |Dv|^p \dz \\ \nn
    &\le
    c\bigg[\mint_{Q_{\rho}^+(z_o)} |Dv|^2 \dz +
    \gamma^{p-2}\mint_{Q_{\rho}^+(z_o)} |Dv|^p \dz \bigg]
    \le
    \frac{c(n,N,p)}{c_1^2} 
    \le
    1, 
\end{align}
provided we have chosen $c_1\gg1$ large enough, in dependence on $n,N,p$. 
In combination with the conditions from above $c_1$ can still be chosen in dependence on $n,N,p, \K, \psi_\beta, a_o,$ $ \|\omega\|_\infty, \mathcal G_\beta$.
At this point we recall the definition of the bilinear form $\mathcal A$ in \eqref{def-A} and observe that the ellipticity and growth condition from \eqref{ellip-bound-a} also hold in the present case. Moreover, from \eqref{approx} and \eqref{small-v-boundary} we see that the assumptions of the boundary version of the $\mathcal A$-caloric approximation lemma, i.e. Lemma \ref{lem:a-cal-boundary} are satisfied. Therefore, we are in the position to apply Lemma \ref{lem:a-cal-boundary}
to $(v,\mathcal A)$ and with $(\nu,L)$ replaced by $(2^{-7(p-2)}\K, 2^{p-2}p\K)$.
The application then yields the existence of an $\mathcal A$-caloric function $h\in L^p(\Lambda_{\rho/4}(t_o);W^{1,p}(B_{\rho/4}^+(x_o),\R^N))$ on $Q_{\rho/4}^+(z_o)$ satisfying 
\begin{align}\label{h-small}
    \mint_{Q_{\rho/4}^+(z_o)}\Big|\frac{h}{\rho/4}\Big|^2 + |Dh|^2 \dz +
    \gamma^{p-2} \mint_{Q_{\rho/4}^+(z_o)}
    \Big|\frac{h}{\rho/4}\Big|^p + |Dh|^p \dz
    \le
    c(n,p)
\end{align}
and
\begin{align}\label{w-h}
    \mint_{Q_{\rho/4}^+(z_o)}\Big|\frac{v-h}{\rho/4}\Big|^2 +
    \gamma^{p-2}\Big|\frac{v-h}{\rho/4}\Big|^p \dz
    \le
    \epsilon
\end{align}
and also $h\equiv 0$ on $\Gamma_{\rho/4}(z_o)$ if $\Gamma_{\rho/4}(z_o)\not=\emptyset$.
Hence, in any case we have proved the existence of an $\mathcal A$-caloric function $h\in L^p(\Lambda_{\rho/4}(t_o);W^{1,p}(B_{\rho/4}^+(x_o),\R^N))$ on $Q_{\rho/4}^+(z_o)$ satisfying \eqref{h-small} and \eqref{w-h}.
We now let $\theta\in(0,2^{-7}]$ to be fixed later and distinguish the cases $(x_o)_n\ge 4\theta\rho$ and $(x_o)_n< 4\theta\rho$.

We first consider the \textit{the case $(x_o)_n\ge 4\theta\rho$}.
With the definition of the affine function $\ell^{(h)}$ from \eqref{def-lh} we use the a priori estimate for the $\mathcal A$-caloric function $h$ from Proposition \ref{prop:lin} which yields that for $s\in\{2,p\}$ there holds (note that $Q_{2\theta\rho}^+(z_o)=Q_{2\theta\rho}(z_o)$)
\begin{align*}
    \mint_{Q_{2\theta\rho}(z_o)} 
    \bigg|\frac{h-\ell^{(h)}_{2\theta\rho}}{16\theta\rho} \bigg|^s\dz 
    &\le
    c\,(8\theta)^s 
    \mint_{Q_{\rho/4}^+(z_o)}\bigg|\frac{h}{\rho/4} \bigg|^s +|Dh|^s \dz \le
    c(n,N,p,\K)\,\gamma^{2-s}\,\theta^s .
\end{align*}
Here we have also used (\ref{h-small}) in the last line. Combining this with (\ref{w-h}) we deduce
\begin{align*}
    \mint_{Q_{2\theta\rho}(z_o)} 
    \bigg|\frac{v-\ell^{(h)}_{\theta\rho}}{\theta\rho} \bigg|^s\dz 
    &\le
    2^{s-1}
    \mint_{Q_{2\theta\rho}(z_o)} 
    \bigg|\frac{v-h}{2\theta\rho} \bigg|^s +
    \bigg|\frac{h-\ell^{(h)}_{2\theta\rho}}{2\theta\rho} \bigg|^s\dz 
	 \\
    &\phantom{m} \le
    2^{s-1}\bigg[
    (8\theta)^{-n-2-s} \mint_{Q_{\rho/4}^+(z_o)}\Big|\frac{v-h}{\rho/4}\Big|^s \dz +
    c\,\gamma^{2-s}\, \theta^s \bigg] \\
    &\phantom{m} \le
    c(n,N,p,\K)\,\gamma^{2-s}
    \big[\theta^{-n-2-s} \epsilon + \theta^s \big].
\end{align*}
At this stage we choose $\epsilon:=\theta^{n+2+2p}$, where $\theta\in (0,2^{-7}]$  is a fixed parameter which will be specified later. This particular choice of $\epsilon$ determines $\delta =\delta (n,N,p,\K,\theta)$.
Rescaling back from $v$ on $Q_{\rho}^+(z_o)$ to $u$
on $Q_{\rho,\lambda}^+(z_o)$ 
we obtain for $s\in\{2,p\}$ that
\begin{align*}
    \mint_{Q_{2\theta\rho,\lambda}(z_o)} &
    \bigg|\frac{u - \ell_{\rho,\lambda} -
    c_1\gamma |D\ell_{\rho,\lambda}|\,
    \ell^{(h)}_{2\theta\rho}}{2\theta\rho} \bigg|^s\dz \nn\\
    &\le
    c\, c_1^s \,\gamma^2
    |D\ell_{\rho,\lambda}|^{s}\,
    \theta^s 
    =
    c\,\theta^s\, |D\ell_{\rho,\lambda}|^{s-p}
    \Big[\Phi_\lambda(\rho) +
    \delta^{-2}\rho^{2\beta}|D\ell_{\rho,\lambda}|^{p}
    \Big],
\end{align*}
where $c=c(n,N,p,\K,c_1)$.
Recalling that by $\widehat\ell_{2\theta\rho,\lambda}\colon\R^n\to\R^N$ we denote the 
unique affine map minimizing \eqref{def-lmin} (with $\rho$ replaced by $2\theta\rho$), the preceding inequality together with Lemma \ref{lem:quasimin} implies for $s\in\{2,p\}$
\begin{align}\label{excess-0}
    \mint_{Q_{2\theta\rho,\lambda}(z_o)}
    \Big|\frac{u - \widehat\ell_{2\theta\rho,\lambda}}{2\theta\rho}\Big|^s\dz
    \le
    c\,\theta^s\, |D\ell_{\rho,\lambda}|^{s-p} \Big[\Phi_\lambda(\rho) +
    \delta^{-2}\rho^{2\beta}|D\ell_{\rho,\lambda}|^{p}
    \Big],
\end{align}
where again $c=c(n,N,p,\K,c_1)$.
In order to proceed further we will show that 
\begin{equation}\label{replace}
	\tfrac12 |D\ell_{\rho,\lambda}|
	\le
	|D\widehat\ell_{2\theta\rho,\lambda}|
	\le 
	2|D\ell_{\rho,\lambda}|
\end{equation}
holds, which allows us to replace $D\ell_{\rho,\lambda}$ in (\ref{excess-0}) by $D\widehat\ell_{2\theta\rho,\lambda}$. 
From \eqref{kronz-est} and the Poincar\'e-type inequality in Lemma \ref{lem:poin-inter} applied with $s=2$ (note that $|D\ell_{\rho,\lambda}|\le2 \lambda$) we obtain the following bound for the difference of the two quantities:
\begin{align*}
    |D\widehat\ell_{2\theta\rho,\lambda} - D\ell_{\rho,\lambda}|^2
    &\le
    n(n+2)
    \mint_{Q_{2\theta\rho,\lambda}(z_o)}
    \bigg|\frac{u - (u)_{2\theta\rho,\lambda} - D\ell_{\rho,\lambda}\, (x-x_o)}
    {2\theta\rho}\bigg|^2\dz \nn\\
    &\le
    c \mint_{Q_{2\theta\rho,\lambda}(z_o)}|Du-D\ell_{\rho,\lambda}|^2\dz 
    \nn \\
    &\phantom{\le\ }+
    c\bigg[\lambda^{2-p}
    \mint_{Q_{2\theta\rho,\lambda}(z_o)}|Du-D\ell_{\rho,\lambda}|^{p-1} \dz +
    \rho^{\beta} \lambda\bigg]^2 ,
\end{align*}
where $c=c(n,N,p, \K, \psi_\beta, a_o, \|\omega\|_\infty, \mathcal G_\beta)$.
Using H\"older's inequality, the definition of $\Phi_\lambda(\rho)$, \eqref{cond-dec-a-1}
and \eqref{cond-dec-a-2} and $\epsilon_2\le \theta^{n+2}$, we further estimate 
\begin{align*}
    & |D\widehat\ell_{2\theta\rho,\lambda} - D\ell_{\rho,\lambda}|^2
    \nonumber\\
    &\phantom{mm}\le
    c\, \theta^{-(n+2)}|D\ell_{\rho,\lambda}|^{2-p}\Phi_\lambda(\rho) +
    c\Big[ |D\ell_{\rho,\lambda}|^{2-p} 
    \big(\theta^{-(n+2)}\Phi_\lambda(\rho)\big)^{1-\frac{1}{p}} +
    \rho^{\beta} |D\ell_{\rho,\lambda}|\Big]^2\\
    &\phantom{mm}\le
    c \Big[
    \theta^{-(n+2)} \epsilon_2^p +
    \big(\theta^{-(n+2)} \epsilon_2^p\big)^{2-\frac{2}{p}} +
    \rho^{2\beta} \Big] |D\ell_{\rho,\lambda}|^{2} \\
    &\phantom{mm}\le
    c_2
    \Big[\theta^{-(n+2)} \epsilon_2^p + \rho^{2\beta}\Big] 
    |D\ell_{\rho,\lambda}|^{2} 
    \le
    \tfrac14|D\ell_{\rho,\lambda}|^{2},
\end{align*}
provided the {\it smallness assumption}
\begin{align}\label{small-2}
    c_2\, \Big[\theta^{-(n+2)} \epsilon_2^p + R_o^{2\beta}\Big]
    \le
    \tfrac14
\end{align}
is satisfied. Note that $c_2=c_2(n,N,p, \K, \psi_\beta, a_o, \|\omega\|_\infty, \mathcal G_\beta)$. This provs the claim \eqref{replace}.
Hence, by \eqref{replace} we are allowed to replace $D\ell_{\rho,\lambda}$ by $D\widehat\ell_{2\theta\rho,\lambda}$ in \eqref{excess-0} which yields that 
\begin{align*}
    \mint_{Q_{2\theta\rho,\lambda}(z_o)}
    \Big|\frac{u - \widehat\ell_{2\theta\rho,\lambda}}{2\theta\rho}\Big|^2\dz
    \le
    c\,\theta^2\, |D\widehat\ell_{2\theta\rho,\lambda}|^{2-p} \Big[\Phi_\lambda(\rho) +
    \delta^{-2}\rho^{2\beta}|D\ell_{\rho,\lambda}|^{p}
    \Big].
\end{align*}
Combining this  with (\ref{excess-0}) for $s=p$ 
we find that
\begin{align*}
    \mint_{Q_{2\theta\rho,\lambda}(z_o)} 
    \bigg|V_{|D\widehat\ell_{2\theta\rho,\lambda}|}
    \bigg(\frac{u - \widehat\ell_{2\theta\rho,\lambda}}{2\theta\rho}\bigg)\bigg|^2  
    \le
    c\,\theta^2
    \Big[\Phi_\lambda(\rho) +
    \delta^{-2}\rho^{2\beta}|D\ell_{\rho,\lambda}|^{p}
    \Big] ,
\end{align*}
where $c=c(n,N,p, \K, \psi_\beta, a_o, \|\omega\|_\infty, \mathcal G_\beta)$. Next, we apply Caccioppoli's inequality from Lemma \ref{lem:cac} in order to estimate the left-hand side
of the preceding inequality from below. We obtain (note that $|D\widehat\ell_{2\theta\rho,\lambda}|\ge \frac12|D\ell_{\rho,\lambda}|\ge 2^{-8}\lambda\ge 2^{-8}\boldsymbol G(\epsilon_1)$, which allows us to apply the lemma) the following estimate:
\begin{align*}
    \mint_{Q_{\theta\rho,\lambda}(z_o)}
    \big|V_{|D\widehat\ell_{2\theta\rho,\lambda}|}
    \big(Du - D\widehat \ell_{2\theta\rho,\lambda}\big)\big|^2 \dz
    \le
    c
    \Big[\theta^2\,\Phi_\lambda(\rho) +
    \delta^{-2}(\theta\rho)^\beta|D\ell_{\rho,\lambda}|^{p}
    \Big] 
\end{align*}
with a constant $c=c(n,N,p, \K, \psi_\beta, a_o, \|\omega\|_\infty, \mathcal G_\beta)$.
Now, recalling the definition of $\ell_{\theta\rho,\lambda}$ Lemma \ref{lem:V-A} allows us to replace $D\widehat\ell_{2\theta\rho,\lambda}$ by $D\ell_{\theta\rho,\lambda}$ in the preceding estimate, i.e.
\begin{align*}
    \Phi_\lambda(\theta\rho)
    &=
    \mint_{Q_{\theta\rho,\lambda}(z_o)}
    \big|V_{|D\ell_{\theta\rho,\lambda}|}
    \big(Du - D\ell_{\theta\rho,\lambda}\big)\big|^2 \dz \\
    &\le
    2^{2p} \mint_{Q_{\theta\rho,\lambda}(z_o)}
    \big|V_{|D\widehat\ell_{2\theta\rho,\lambda}|}
    \big(Du - D\widehat\ell_{2\theta\rho,\lambda}\big)\big|^2 \dz \\
    &\le
    c_3
    \Big[\theta^2\,\Phi_\lambda(\rho) +
    \delta^{-2}(\theta\rho)^\beta|D\ell_{\rho,\lambda}|^{p}
    \Big] ,
\end{align*}
where $\theta\in (0,2^{-7}]$ is still to be chosen and $c_3=c_3(n,N,p, \K, \psi_\beta, a_o, \|\omega\|_\infty, \mathcal G_\beta)$. 

Next, we turn our attention to \textit{the second case $(x_o)_n< 4\theta\rho$}.
We recall that $\theta\in(0,2^{-7}]$ is an arbitrary fixed parameter that shall be chosen at the end of the proof.
With $z_o':=((x_o)_1,\dots,(x_o)_{n-1},0,t_o)$ denoting the orthogonal projection of $z_o$ on $\Gamma$ we have $Q_{16\theta\rho}^+(z_o')\subset Q_{\rho/8}^+(z_o')\subset Q_{\rho/4}^+(z_o)$.
With the definition of the affine function $\ell^{(h)}$ from \eqref{def-lh} we use the a priori estimate for the $\mathcal A$-caloric function $h$ from Proposition \ref{prop:lin} which yields that for $s\in\{2,p\}$ there holds
\begin{align*}
    \mint_{Q_{16\theta\rho}^+(z_o')} 
    \bigg|\frac{h-\ell^{(h)}_{z_o';16\theta\rho}}{16\theta\rho} \bigg|^s\dz 
    &\le
    c\,(128\theta)^s 
    \mint_{Q_{\rho/8}^+(z_o')}\bigg|\frac{h}{\rho/8} \bigg|^s +|Dh|^s \dz \\
    &\le
    c\,\theta^s 
    \mint_{Q_{\rho/4}^+(z_o)}\bigg|\frac{h}{\rho/4} \bigg|^s +|Dh|^s \dz \\
    & \le
    c(n,N,p,\K)\,\gamma^{2-s}\,\theta^s .
\end{align*}
Here we have also used \eqref{h-small} in the last line. Combining this with (\ref{w-h}) we deduce
\begin{align*}
    \mint_{Q_{16\theta\rho}^+(z_o')} 
    \bigg|\frac{v-\ell^{(h)}_{z_o';16\theta\rho}}{16\theta\rho} \bigg|^s\dz 
    &\le
    2^{s-1}
    \mint_{Q_{16\theta\rho}^+(z_o')} 
    \bigg|\frac{v-h}{16\theta\rho} \bigg|^s +
    \bigg|\frac{h-\ell^{(h)}_{z_o';16\theta\rho}}{16\theta\rho} \bigg|^s\dz 
	 \\
    &\le
    2^{s-1}\bigg[
    (128\theta)^{-n-2-s} \mint_{Q_{\rho/8}^+(z_o')}\Big|\frac{v-h}{\rho/8}\Big|^s \dz +
    c\,\gamma^{2-s}\, \theta^s \bigg] \\
    &\le
    c\bigg[
    \theta^{-n-2-s} \mint_{Q_{\rho/4}^+(z_o)}\Big|\frac{v-h}{\rho/4}\Big|^s \dz +
    \gamma^{2-s}\, \theta^s \bigg] \\
    &\le
    c(n,N,p,\K)\,\gamma^{2-s}
    \big[\theta^{-n-2-s} \epsilon + \theta^s \big].
\end{align*}
We choose $\epsilon:=\theta^{n+2+2p}$ as in the first case and note that this particular choice of $\epsilon$ determines $\delta =\delta (n,N,p,\K,\theta)$.
Rescaling back from $v$ on $Q_{\rho}^+(z_o)$ to $u$
on $Q_{\rho,\lambda}^+(z_o)$ we obtain for $s\in\{2,p\}$ that
\begin{align*}
    \mint_{Q_{16\theta\rho,\lambda}^+(z_o')} &
    \bigg|\frac{u - \ell_{\rho,\lambda} -
    c_1\gamma |D\ell_{\rho,\lambda}|\,
    \ell^{(h)}_{z_o';16\theta\rho}}{16\theta\rho} \bigg|^s\dz \nn\\
    &\le
    c\, c_1^s \,\gamma^2
    |D\ell_{\rho,\lambda}|^{s}\,
    \theta^s 
    =
    c\,\theta^s\, |D\ell_{\rho,\lambda}|^{s-p}
    \Big[\Phi_\lambda(\rho) +
    \delta^{-2}\rho^{2\beta}|D\ell_{\rho,\lambda}|^{p}
    \Big],
\end{align*}
where $c=c(n,N,p,\K,c_1)$.
We note that the definitions of $\ell_{\rho,\lambda}$ and $\ell^{(h)}_{z_o';16\theta\rho}$ and the fact that $(x_o)_n< 4\theta\rho<\rho/2$ imply that 
$\ell_{\rho,\lambda} - c_1\gamma |D\ell_{\rho,\lambda}|\,\ell^{(h)}_{z_o';16\theta\rho}=0$ on $\Gamma$. By Lemma \ref{lem:quasimin}, we can replace this affine function by the unique affine map $\widehat\ell_{z_o';16\theta\rho,\lambda}\colon\R^n\to\R^N$ minimizing \eqref{def-lmin} (with $(z_o,\rho)$ replaced by $(z_o',16\theta\rho)$) amongst all affine maps $\ell(z)=\ell(x)$ satisfying $\ell =0$ on $\Gamma$. Hence, for $s\in\{2,p\}$ we get
\begin{align}\label{excess-0-bd}
    \mint_{Q_{16\theta\rho,\lambda}^+(z_o')}
    \Big|\frac{u - \widehat\ell_{z_o';16\theta\rho,\lambda}}{16\theta\rho}\Big|^s\dz
    \le
    c\,\theta^s\, |D\ell_{\rho,\lambda}|^{s-p} \Big[\Phi_\lambda(\rho) +
    \delta^{-2}\rho^{2\beta}|D\ell_{\rho,\lambda}|^{p}
    \Big],
\end{align}
where $c=c(n,N,p,\K,c_1)$.
In order to proceed further we will show that 
\begin{equation}\label{replace-bd}
	\tfrac12 |D\ell_{\rho,\lambda}|
	\le
	|D\widehat\ell_{z_o';16\theta\rho,\lambda}|
	\le 
	2|D\ell_{\rho,\lambda}|
\end{equation}
holds which will allow us to replace $D\ell_{\rho,\lambda}$ in (\ref{excess-0-bd}) by $D\widehat\ell_{z_o';16\theta\rho,\lambda}$. 
Using \eqref{kronz-est-boundary} and the Poincar\'e inequality from Lemma \ref{lem:poin-boundary} we obtain the following bound for the difference of the two quantities:
\begin{align}\label{kronz-bd}
    |D\widehat\ell_{z_o';16\theta\rho,\lambda} - D\ell_{\rho,\lambda}|^2
    &\le
    (n+2)
    \mint_{Q_{16\theta\rho,\lambda}^+(z_o')}
    \bigg|\frac{u - \ell_{\rho,\lambda}}
    {16\theta\rho}\bigg|^2\dz \\
    &\le
    c_4(n,N) 
    \mint_{Q_{16\theta\rho,\lambda}^+(z_o')}|Du-D\ell_{\rho,\lambda}|^2\dz .
    \nn
\end{align}
Taking into account the fact that $Q_{16\theta\rho,\lambda}^+(z_o')\subset Q_\rho^+(z_o)$, the definition of $\Phi_\lambda(\rho)$ and hypothesis
\eqref{cond-dec-a-2}, we further estimate 
\begin{align*}
    |D\widehat\ell_{z_o';16\theta\rho,\lambda} - D\ell_{\rho,\lambda}|^2
    &\le
    c_4\, (16\theta)^{-(n+2)}
    \mint_{Q_{\rho,\lambda}^+(z_o)}|Du-D\ell_{\rho,\lambda}|^2\dz \\
    &\le
    c_4\, \theta^{-(n+2)}|D\ell_{\rho,\lambda}|^{2-p}\Phi_\lambda(\rho) \\
    &\le
    c_4\, \theta^{-(n+2)} \epsilon_2^p \,
    |D\ell_{\rho,\lambda}|^{2}  
    \le
    \tfrac14|D\ell_{\rho,\lambda}|^{2},
\end{align*}
provided the {\it smallness assumption}
\begin{align}\label{small-2-bd}
    c_4\, \theta^{-(n+2)} \epsilon_2^p
    \le
    \tfrac14
\end{align}
is satisfied with the constant $c_4=c_4(n,N)$. This ensures that the claim \eqref{replace-bd} is true.
Hence, by \eqref{replace-bd} we are allowed to replace in \eqref{excess-0-bd} $D\ell_{\rho,\lambda}$ by $D\widehat\ell_{z_o';16\theta\rho,\lambda}$ which yields 
\begin{align*}
    \mint_{Q_{16\theta\rho,\lambda}^+(z_o')}
    \Big|\frac{u - \widehat\ell_{z_o';16\theta\rho,\lambda}}{16\theta\rho}\Big|^2\dz
    \le
    c\,\theta^2\, |D\widehat\ell_{z_o';16\theta\rho,\lambda}|^{2-p} \Big[\Phi_\lambda(\rho) +
    \delta^{-2}\rho^{2\beta}|D\ell_{\rho,\lambda}|^{p}
    \Big].
\end{align*}
Combining this  with (\ref{excess-0-bd}) for $s=p$
we find that
\begin{align*}
    \mint_{Q_{16\theta\rho,\lambda}^+(z_o')} 
    \bigg|V_{|D\widehat\ell_{z_o';16\theta\rho,\lambda}|}
    \bigg(\frac{u - \widehat\ell_{z_o';16\theta\rho,\lambda}}{16\theta\rho}\bigg)\bigg|^2  
    \le
    c\,\theta^2\,
    \Big[\Phi_\lambda(\rho) +
    \delta^{-2}\rho^{2\beta}|D\ell_{\rho,\lambda}|^{p}
    \Big] ,
\end{align*}
where $c=c(n,N,p, \K, \psi_\beta, a_o, \|\omega\|_\infty, \mathcal G_\beta)$. Next, we apply Caccioppoli's inequality from Lemma \ref{lem:cac} in order to estimate the left-hand side
of the preceding inequality from below. We obtain (note that $|D\widehat\ell_{z_o';8\theta\rho,\lambda}|\ge \frac12|D\ell_{\rho,\lambda}|\ge 2^{-8}\lambda\ge 2^{-8}\boldsymbol G(\epsilon_1)$, which allows us to apply the lemma) the following estimate:
\begin{align*}
    \mint_{Q_{8\theta\rho,\lambda}^+(z_o')}
    \big|V_{|D\widehat\ell_{z_o';16\theta\rho,\lambda}|}
    \big(Du - D\widehat \ell_{z_o';16\theta\rho,\lambda}\big)\big|^2 \dz
    \le
    c\Big[\theta^2\,\Phi_\lambda(\rho) +
    \delta^{-2}(\theta\rho)^\beta|D\ell_{\rho,\lambda}|^{p}
    \Big] 
\end{align*}
with a constant $c=c(n,N,p, \K, \psi_\beta, a_o, \|\omega\|_\infty, \mathcal G_\beta)$.
Since $Q_{\theta\rho,\lambda}^+(z_o)\subset Q_{8\theta\rho,\lambda}^+(z_o')$ the last estimate implies
\begin{align*}
    \mint_{Q_{\theta\rho,\lambda}^+(z_o)}
    \big|V_{|D\widehat\ell_{z_o';16\theta\rho,\lambda}|}
    \big(Du - D\widehat \ell_{z_o';16\theta\rho,\lambda}\big)\big|^2 \dz
    \le
    c\Big[\theta^2\,\Phi_\lambda(\rho) +
    \delta^{-2}(\theta\rho)^\beta|D\ell_{\rho,\lambda}|^{p}
    \Big]  .
\end{align*}
Now, recalling the definition of $\ell_{\theta\rho,\lambda}$, Lemma \ref{lem:V-A} allows us to replace $D\widehat\ell_{z_o';16\theta\rho,\lambda}$ by $D\ell_{\theta\rho,\lambda}$ in the preceding estimate, i.e.
\begin{align*}
    \Phi_\lambda(\theta\rho)
    &=
    \mint_{Q_{\theta\rho,\lambda}^+(z_o)}
    \big|V_{|D\ell_{\theta\rho,\lambda}|}
    \big(Du - D\ell_{\theta\rho,\lambda}\big)\big|^2 \dz \\
    &\le
    2^{2p} \mint_{Q_{\theta\rho,\lambda}^+(z_o)}
    \big|V_{|D\widehat\ell_{z_o';16\theta\rho,\lambda}|}
    \big(Du - D\widehat\ell_{z_o';16\theta\rho,\lambda}\big)\big|^2 \dz \\
    &\le
    c_5\Big[\theta^2\,\Phi_\lambda(\rho) +
    \delta^{-2}(\theta\rho)^\beta|D\ell_{\rho,\lambda}|^{p}
    \Big]  ,
\end{align*}
where $\theta\in (0,2^{-7}]$ is still to be fixed and $c_5=c_5(n,N,p, \K, \psi_\beta, a_o, \|\omega\|_\infty$, $\mathcal G_\beta)$. 

At this stage we perform the choices of the constants $\theta,\epsilon_1, \epsilon_2$ and $R_o$.
We first choose $\theta\in(0,2^{-7}]$ such that $c_3 \theta^2\le \frac12 \theta^{\beta}$, $c_5 \theta^2\le \frac12 \theta^{\beta}$ and $\theta\le 8^{-p/\beta}$ holds. Note that $\theta\le 8^{-p/\beta}$ appears in the statement of the lemma and will be needed in the application of the lemma later on.
This fixes  $\theta$ in dependence
on $n,N,p, \K, \beta,\psi_\beta, a_o, \|\omega\|_\infty, \mathcal G_\beta$. 
As mentioned before this fixes firstly $\epsilon=\theta^{n+2+2p}$ in dependence of the same parameters, and secondly
$\delta$, also in dependence on the the same parameters.
Finally, we have to ensure that the smallness conditions \eqref{cond-eps0}, \eqref{small-2} and \eqref{small-2-bd} are satisfied. 
Therefore, we choose $\epsilon_1\in(0,1]$ according to
\begin{equation*}
	\epsilon_1
	\le
	\tfrac13 \delta
\end{equation*}
and $\epsilon_2\in(0,\theta^{n+2}]$ small enough to have 
\begin{equation*}
	\epsilon_2
	\le
	\tfrac13 \delta^2
	\quad\mbox{and}\quad
	\max\{c_2,c_4\}\, \theta^{-(n+2)} \epsilon_2^p
    \le
    \tfrac18
\end{equation*}
and finally $R_o\in(0,1]$ satisfying
\begin{equation*}
	R_o^\beta
	\le 
	\tfrac13\delta
	\quad\mbox{and}\quad
	c_2\, R_o^{2\beta}
    \le
    \tfrac18.
\end{equation*}
Then, $\epsilon_1, \epsilon_2$ and $R$ depend on $n,N,p, \K, \psi_\beta, \beta,a_o, \|\omega\|_\infty, \mathcal G_\beta$.
This finishes the proof of the lemma.
\end{proof}

Our next aim is to iterate Lemma \ref{lem:decay-a-cal}. This is achieved in the following:
%
\begin{proposition}\label{prop:iter-a-cal}
There exist constants
$R_1\in(0,1]$, $\theta\in(0,8^{-p/\beta}]$, $\epsilon_1\in(0,1]$, $\epsilon_2\in(0,\theta^{n+2}]$ and $\bar c\ge 1$ depending on 
$n,N,p, \K,\beta,\psi_\beta, a_o, $ $\|\omega\|_\infty,\mathcal G_\beta$
such that under the assumptions of Proposition \ref{prop:main-lip} with some $R\in(0,R_1]$
the following is true: Whenever 
$Q_{\rho,\lambda}(z_o)\subset Q_R$ is a cylinder with $(x_o)_n\ge 0$ and $\lambda\ge \boldsymbol G(\epsilon_1)$ satisfying
\begin{equation}\label{cond-dec-a-1-}
    2^{-6}\lambda
    \le
    |D\ell_{z_o;\rho,\lambda}|
    \le
    \lambda
\end{equation}
and 
\begin{equation}\label{cond-dec-a-2-}
    \Phi_\lambda\big(z_o,\rho,D\ell_{z_o;\rho,\lambda}\big)
    \le
    \epsilon_2^p\, |D\ell_{z_o;\rho,\lambda}|^p,
\end{equation}
then the limit
\begin{equation}\label{limit-nondeg}
    \Gamma_{z_o}
    \equiv
    \lim_{r\downarrow 0}D\ell_{z_o;r}
\end{equation}
exists and
there holds
\begin{equation}\label{mod-gamma}
    2^{-7}\lambda
    \le
    |\Gamma_{z_o}|
    \le
    2 \lambda.
\end{equation}
\end{proposition}

\begin{proof}
Once again we abbreviate $\ell_{\rho,\lambda}\equiv \ell_{z_o;\rho,\lambda}$ and $\Phi_\lambda(\rho)\equiv\Phi_\lambda(z_o,\rho,D\ell_{z_o;\rho,\lambda})$.
We let
$R_o\in(0,1]$,
$\theta\in(0,8^{-p/\beta}]$, $\epsilon_1\in(0,1]$, $\epsilon_2\in(0,\theta^{n+2}]$ and $\bar c\ge 1$ be the corresponding constants from Lemma \ref{lem:decay-a-cal} depending on 
$n,N,p, \K, \beta,\psi_\beta, a_o, \|\omega\|_\infty, \mathcal G_\beta$.
Next, we choose $R_1\in(0,1]$ satisfying
\begin{equation*}
    R_1
    \le
    \max\Big\{R_o,\Big(\frac{\epsilon_2^p}{2^{p}\bar c}\Big)^{\frac{1}{\beta}}\Big\}.
\end{equation*}
By induction we shall prove that for any $i\in\N$ there holds:
\begin{align*}
	\tag*{({\bf I})$_i$}
	\Phi_\lambda(\theta^i\rho)
    \le
    \theta^{\beta i}\,\epsilon_2^p\, |D\ell_{\rho,\lambda}|^{p} ,
\end{align*}
\begin{align*}
	\tag*{({\bf II})$_i$}
	\bigg[1-\frac14
    \sum_{j=0}^{i-1}2^{-j}\bigg]|D\ell_{\rho,\lambda}|
    \le
    |D\ell_{\theta^i\rho,\lambda}|
    \le
	\bigg[1+\frac14 \sum_{j=0}^{i-1}2^{-j}\bigg]|D\ell_{\rho,\lambda}|\,.
\end{align*}
We first prove (I)$_1$ and (II)$_1$. Due to assumptions \eqref{cond-dec-a-1-} and \eqref{cond-dec-a-2-} we are allowed to apply Lemma \ref{lem:decay-a-cal}. Therefore, \eqref{decay-a-cal}, \eqref{cond-dec-a-2-} and the choice of $R_1$ yield that 
\begin{align*}
	\Phi_\lambda(\theta\rho)
    \le
    \tfrac12 \theta^{\beta}
    \Big[\Phi_\lambda(\rho) +
    \bar c\,\rho^{\beta} |D\ell_{\rho,\lambda}|^{p}
    \Big] 
    \le
    \tfrac12 \theta^{\beta}
    \big[\epsilon_2^p + \bar c\,R_1^{\beta} \big] |D\ell_{\rho,\lambda}|^{p}
    \le
    \theta^{\beta}\epsilon_2^p\, |D\ell_{\rho,\lambda}|^{p} ,
\end{align*}
i.e. assertion (I)$_1$ holds.
Assertion (II)$_1$ now is a consequence of \eqref{cond-dec-a-2-} and
the fact that  $\epsilon_2\le\theta^{n+2}$ and $\theta\le 8^{-p/\beta}$ since
\begin{align*}
    |D\ell_{\theta\rho,\lambda} - D\ell_{\rho,\lambda}|
    &\le
    \bigg(\mint_{Q_{\theta\rho,\lambda}^+(z_o)}
    |Du - D\ell_{\rho,\lambda}|^p\dz\bigg)^{\frac1p}
    \le
    \theta^{-\frac{n+2}{p}}
    \Phi_\lambda(\rho)^{\frac1p} \\
    &\le
    \theta^{-\frac{n+2}{p}} \epsilon_2
    |D\ell_{\rho,\lambda}| 
    \le
    \theta^{(n+2)(1-\frac{1}{p})} 
    |D\ell_{\rho,\lambda}| \\
    &\le
    8^{-\frac1\beta (n+2)(p-1)} 
    |D\ell_{\rho,\lambda}|
    \le
    \tfrac14\, |D\ell_{\rho,\lambda}|.
\end{align*}
Now, we prove (I)$_i$ and (II)$_i$ for $i>1$ assuming that (I)$_{i-1}$ and (II)$_{i-1}$ hold.
Using \eqref{cond-dec-a-1-} in (II)$_{i-1}$ we see that assumption \eqref{cond-dec-a-1}
of Lemma \ref{lem:decay-a-cal} is satisfied on $Q_{\theta^{i-1}\rho,\lambda}^+(z_o)$. 
Moreover, by the bound from below in (II)$_{i-1}$ we have $|D\ell_{\rho,\lambda}|\le 2|D\ell_{\theta^{i-1}\rho,\lambda}|$.
Joining this with (I)$_{i-1}$, (\ref{cond-dec-a-2-}) and  $\theta\le 8^{-p/\beta}$ we get
\begin{align*}
    \Phi_\lambda(\theta^{i-1}\rho)
    \le
    \theta^{\beta(i-1)}\, \epsilon_2^p \, |D\ell_{\rho,\lambda}|^{p} 
    \le
    2^{p}\theta^{\beta(i-1)}\, \epsilon_2^p\, |D\ell_{\theta^{i-1}\rho,\lambda}|^{p} 
    \le
    \epsilon_2^p\, |D\ell_{\theta^{i-1}\rho,\lambda}|^p,
\end{align*}
ensuring that also \eqref{cond-dec-a-2} holds.
Therefore, we can apply Lemma \ref{lem:decay-a-cal} with $\theta^{i-1}\rho$ instead of $\rho$. Together with (I)$_{i-1}$ and (II)$_{i-1}$ and the choice of $R_1$ this yields:
\begin{align*}
    \Phi_\lambda(\theta^{i}\rho)
    &\le
    \tfrac12 \theta^{\beta}
    \Big[\Phi_\lambda(\theta^{i-1}\rho) +
    \bar c\,(\theta^{i-1}\rho)^{\beta} |D\ell_{\theta^{i-1}\rho,\lambda}|^{p}
    \Big] \\
    &\le
    \tfrac12 \theta^{\beta}
    \Big[\theta^{\beta (i-1)}\,\epsilon_2^p\, |D\ell_{\rho,\lambda}|^{p} +
    2^{p}\bar c\,(\theta^{i-1}\rho)^{\beta} |D\ell_{\rho,\lambda}|^{p}
    \Big] \\
    &\le
    \tfrac12 \theta^{\beta i}
    \Big[\epsilon_2^p +
    2^{p}\bar c\,R_1^{\beta} 
    \Big] |D\ell_{\rho,\lambda}|^{p}  
    \le
    \theta^{\beta i}\, \epsilon_2^p \, |D\ell_{\rho,\lambda}|^{p}  ,
\end{align*}
proving (I)$_i$.
Moreover, from (I)$_{i-1}$ and $\epsilon_2\le \theta^{n+2}$ we obtain
\begin{align}\label{mw-diff}
    |D\ell_{\theta^{i}\rho,\lambda} - D\ell_{\theta^{i-1}\rho,\lambda}|^p
    &\le
    \theta^{-(n+2)}
    \Phi_\lambda(\theta^{i-1}\rho) \\
    &\le
    \theta^{\beta(i-1)}\, \theta^{-(n+2)} \,
    \epsilon_2^p |D\ell_{\rho,\lambda}|^p 
    \le
    \theta^{\beta(i-1)}
    |D\ell_{\rho,\lambda}|^p . \nn
\end{align}
Together with the fact that $\theta\le 8^{-p/\beta}$ we therefore have
\begin{align*}
    |D\ell_{\theta^{i}\rho,\lambda} - D\ell_{\theta^{i-1}\rho,\lambda}|
    \le
    8^{-(i-1)}
    |D\ell_{\rho,\lambda}|
    \le
    \tfrac1{4} \cdot 2^{-(i-1)}
    |D\ell_{\rho,\lambda}|,
\end{align*}
which together with (II)$_{i-1}$ proves the claim (II)$_i$.

We now come to the proof of \eqref{limit-nondeg} and \eqref{mod-gamma}.
Given $j<k$ the estimate in \eqref{mw-diff} applied for $i=j+1,\dots,k$ yields
\begin{align*}
    |D\ell_{\theta^j\rho,\lambda} - D\ell_{\theta^{k}\rho,\lambda}|
    &\le
    \sum_{i=j+1}^k |D\ell_{\theta^i\rho,\lambda} - D\ell_{\theta^{i-1}\rho,\lambda}| \\
    &\le
    |D\ell_{\rho,\lambda}|
    \sum_{i=j+1}^k \theta^{\frac{\beta(i-1)}{p}}
    \le
    \frac{\theta^{\frac{\beta j}{p}}}{1- \theta^{\frac{\beta}{p}}}~
    |D\ell_{\rho,\lambda}|  .
\end{align*}
Therefore, $\{D\ell_{\theta^{i}\rho,\lambda}\}_{i\in\N}$ is a Cauchy sequence and the limit
\begin{equation*}
    \widetilde \Gamma_{z_o}
    =
    \lim_{i\to\infty}D\ell_{\theta^{i}\rho,\lambda}
\end{equation*}
exists.
Passing to the limit $k\to\infty$ in the preceding inequality and taking into account that 
$1/(1- \theta^{\frac{\beta}{p}})\le 2$ since $\theta\le 8^{-\frac{p}{\beta}}$ yields
\begin{equation}\label{gamma-tilde}
    |D\ell_{\theta^j\rho,\lambda} - \widetilde\Gamma_{z_o}|
    \le
    2 \theta^{\frac{\beta j}{p}}\,
    |D\ell_{\rho,\lambda}|,
    \qquad\, \forall\, j\in\N .
\end{equation}
Finally, we have to replace $\widetilde\Gamma_{z_o}$ by $\Gamma_{z_o}$, i.e. by the limit of $D\ell_r$ defined with respect to non-intrinsic parabolic cylinders when $r\downarrow 0$.
For $r\in(0,\lambda^{\frac{2-p}{2}}\rho]$ we choose $j\in\N_0$ such that $\lambda^{(2-p)/2}\theta^{j+1}\rho< r\le\lambda^{(2-p)/2}\theta^j\rho$. From (I)$_j$ and $\epsilon_2\le \theta^{n+2}$ we deduce that
\begin{align*}
    |D\ell_r - D\ell_{\theta^j\rho,\lambda}|^p
    &\le
    \mint_{Q_r^+(z_o)}|Du-D\ell_{\theta^j\rho,\lambda}|^p \dz \\
    &\le
    \frac{|Q_{\theta^j\rho,\lambda}^+(z_o)|}{|Q_r^+(z_o)|}
    \mint_{Q_{\theta^j\rho,\lambda}^+(z_o)}
    |Du-D\ell_{\theta^j\rho,\lambda}|^p \dz \\
    &\le
    \frac{(\theta^j\rho)^{n+2}\lambda^{2-p}}{r^{n+2}} \
    \theta^{\beta j}\,\epsilon_2^p |D\ell_{\rho,\lambda}|^p\\
    &\le
    \theta^{-(n+2)} \lambda^{\frac{n(p-2)}{2}} \theta^{\beta j}\,\epsilon_2^p|D\ell_{\rho,\lambda}|^p \\
    &\le
    \lambda^{\frac{n(p-2)}{2}} \theta^{\beta j}\,|D\ell_{\rho,\lambda}|^p.
\end{align*}
Together with \eqref{gamma-tilde} we therefore have
\begin{align*}
    |D\ell_r - \widetilde\Gamma_{z_o}|
    &\le
    |D\ell_r - D\ell_{\theta^j\rho,\lambda}| +
    |D\ell_{\theta^j\rho,\lambda} - \widetilde\Gamma_{z_o}| \\
    &\le
    \lambda^{\frac{n(p-2)}{2p}} \theta^{\frac{\beta j}{p}}\,|D\ell_{\rho,\lambda}| +
    2 \theta^{\frac{\beta j}{p}}\,|D\ell_{\rho,\lambda}| \\
    &\le
    2\theta^{-\frac{\beta}{p}}\lambda^{\frac{(n+\beta)(p-2)}{2p}} \Big(\frac{r}{\rho}\Big)^{\frac{\beta}{p}}\,|D\ell_{\rho,\lambda}| .
\end{align*}
Passing to the limit $r\downarrow 0$ in the right-hand side we infer that
\begin{equation*}
    \Gamma_{z_o}
    \equiv
    \lim_{r\downarrow 0}D\ell_{r}
    =
    \widetilde \Gamma_{z_o}.
\end{equation*}
Moreover, passing to the limit $i\to\infty$ in (II)$_i$ and using assumption \eqref{cond-dec-a-1-} we find that (\ref{mod-gamma}) is satisfied for $\widetilde \Gamma_{z_o}$
and hence for $\Gamma_{z_o}$.
This finishes the proof of the Proposition.
\end{proof}


\subsection{The degenerate regime}\label{sec:degenerate}
In this section we are concerned with the so called {\it degenerate regime} where the solution of the original parabolic system is comparable to the solution of the parabolic $p$-Laplacian system. The main result of this chapter -- a bound for the mean value of $|Du|^p$ on a smaller nested cylinder -- is stated in Lemma \ref{lem:decay-p-cal}. The proof is achieved via a comparison problem and delicate a priori estimates which are a consequence of the $C^{1;\alpha}$-theorey for the parabolic $p$-Laplacian system due to DiBenedetto \& Friedman.

\subsubsection{A comparison problem}

Throughout this section we let $R\in(0,1]$ and suppose that the hypothesis of Proposition \ref{prop:main-lip} are in force.
In the following we let $A\in \R^{Nn}$ and $\epsilon\in(0,1]$ and consider a cylinder $Q_{\rho,\lambda}(z_o)\subset Q_R$ with $(x_o)_n\ge 0$ and $\lambda\ge 1$ and suppose that
\begin{equation}\label{cond-comp}
    \boldsymbol G(\epsilon)^p
    \le
    \mint_{Q_{\rho,\lambda}^+(z_o)}|Du|^p\dz
    \le
    \lambda^p
\end{equation}
holds, where $\boldsymbol G(\epsilon)$ is defined in \eqref{def-Geps}. At this point we note that \eqref{cond-comp} implies
\begin{equation}\label{cond-comp:d}
    \|Dg\|_{L^\infty(Q_R^+)}
    \le
    \epsilon\,\boldsymbol G(\epsilon)
    \le
    \epsilon\, \lambda.
\end{equation}
By 
$$v\in L^p\big(\Lambda_{\rho,\lambda}(t_o);W^{1,p}(B_\rho^+(x_o),\R^N)\big)$$ 
we denote the unique solution of the following parabolic Cauchy-Dirichlet-problem:
\begin{equation}\label{system-comp}
\left\{
\begin{array}{cc}
    \partial_t v - \Div b(Dv) = 0
    &\qquad\mbox{in $Q_{\rho,\lambda}^+(z_o)$} \\[7pt]
    v=u
    &\qquad\mbox{on $\partial_{\mathcal P}Q_{\rho,\lambda}^+(z_o)$.}
\end{array}
\right.
\end{equation}
In the following we will derive a comparison estimate for $v$. Thereby, the computations concerning the use of the time derivatives $\partial_t v$ and $\partial_t u$ are somewhat formal. Nevertheless, they can be made rigorous by the use of a mollification procedure in time as for instance Steklov averages. Since this argumentation is standard, we omit the details and proceed formally.
Moreover, we abbreviate $Q\equiv Q_{\rho,\lambda}^+(z_o)$ and $B\equiv B_\rho^+(x_o)$.

To prove the comparison estimate we first test both, the weak formulation of \eqref{system-comp} and \eqref{system-lat-weak} by $\varphi=v-u$ and then subtract the resulting identities. This leads us to
\begin{align*}
    \int_{Q}
    \big\langle a\big((Du+Dg)\Psi\big) , ( & Du-Dv) \Psi\big\rangle - 
    \langle b(Dv),Du-Dv\rangle\dz \\
    &=
    -\int_{Q}
    \partial_t (u-v)\cdot(u-v) \dz -
    \int_{Q} g_t\cdot(u-v) \dz
\end{align*}
Taking into account that
\begin{align*}
    \int_{Q}
    \partial_t (u-v)\cdot(u-v) \dz 
    &=
    \half \int_{B} |(u-v)(\cdot,t_o+\lambda^{2-p}\rho^2)|^2 \dx
    \ge
    0 
\end{align*}
and re-arranging terms we get
\begin{align}\label{comparison-start}
    \mint_{Q^+} 
    \langle b(Du)-b(Dv),Du-Dv\rangle\dz 
    \le
    \mbox{I} + \mbox{II} + \mbox{III} - \mbox{IV},
\end{align}
where
\begin{align*}
	\mbox{I}
    &:=
    \mint_{Q^+}
    \big\langle b(Du) - b(0) -
    \big[b\big(Du+Dg(0,t)\big) - b\big(Dg(0,t)\big)\big],Du-Dv\big\rangle\dz \nn\\
	\mbox{II}
    &:=
    \mint_{Q^+}
    \big\langle b\big(Du+Dg(0,t)\big) - b\big(Dg(0,t)\big) \nn\\
    &\phantom{mmmmm}- 
    \big[a\big(Du+Dg(0,t)\big) - a\big(Dg(0,t)\big)\big],
    Du-Dv\big\rangle\dz  \nn\\
	\mbox{III}
    &:=
    \mint_{Q^+}
    \big\langle a\big(Du+Dg(0,t)\big), Du-Dv\big\rangle -
    \big\langle a\big((Du+Dg)\Psi\big) ,
    (Du-Dv)\Psi\big\rangle\dz  \nn\\
	\mbox{IV}
    &:=
    \mint_{Q^+} g_t\cdot(u-v) \dz .
\end{align*}
Note that in II we used
$
    \tmint_{Q^+}
    \langle a\big(Dg(0,t)\big), Du-Dv\rangle\dz 
    =
    0
$.
We now in turn estimate the terms I -- IV. For the estimate of I we use Lemma \ref{lem:diff-1} with $(\xi,\xi_o)$ replaced by $(Du,0)$ and \eqref{cond-comp:d} to infer
\begin{align*}
    &\big|b(Du) - b(0) -
    \big[b\big(Du+Dg(0,t)\big) - b\big(Dg(0,t)\big)\big] \big|\\
    &\phantom{mmmmmmm}\le
    c\, |Dg(0,t)|\big(|Du|^2 + |Dg(0,t)|^2\big)^{\frac{p-3}{2}} |Du| \\
    &\phantom{mmmmmmm}\le
    c\,\epsilon \lambda\big(|Du|^2 + |Dg(0,t)|^2\big)^{\frac{p-3}{2}} |Du| \\
    &\phantom{mmmmmmm}\le
    c\,\epsilon\lambda
    \big[|Du|^{p-2}
    +
    \chi_{p>3}|Dg(0,t)|^{p-3} |Du| \big] \\
    &\phantom{mmmmmmm}\le
    c\,\epsilon \big[ \lambda|Du|^{p-2}
    +
    \lambda^{p-2} |Du| \big],
\end{align*}
where $c=c(p,\K)$ and $\chi_{p>3}=1$ if $p>3$ and $\chi_{p>3}=0$ if $p\le 3$.
Inserting this into I and using Young's and H\"older's inequality and \eqref{cond-comp} we obtain for $\delta\in(0,1]$ that
\begin{align*}
    \mbox{I}
    &\le
    c\,\epsilon
    \mint_{Q} \big[ \lambda|Du|^{p-2}
    +
    \lambda^{p-2} |Du| \big] |Du-Dv| \dz \\
    &\le
    \epsilon\delta\mint_{Q} |Du-Dv|^p \dz +
    c\,\epsilon
    \mint_{Q} \Big[\lambda^{\frac{p}{p-1}}|Du|^{\frac{p(p-2)}{p-1}}
    +
    \lambda^{\frac{p(p-2)}{p-1}}|Du|^{\frac{p}{p-1}}\Big]
    \dz \\
    &\le
    \delta\mint_{Q} |Du-Dv|^p \dz +
    c\,\epsilon
    \bigg[\lambda^{\frac{p}{p-1}}\bigg(\mint_{Q} |Du|^p\dz\bigg)^{\frac{p-2}{p-1}}
    +
    \lambda^{\frac{p(p-2)}{p-1}}\bigg(\mint_{Q}|Du|^p\dz\bigg)^{\frac{1}{p-1}}\bigg]
    \\
    &\le
    \delta\mint_{Q} |Du-Dv|^p \dz +
    c(p,\K,1/\delta)\, \epsilon\lambda^p.
\end{align*}
For the estimate of II we apply Lemma \ref{lem:asymp} with the choice $(Dg(0,t),Du(z)+Dg(0,t),\lambda,\epsilon)$ instead of $(A,\xi,\delta,\epsilon)$. Note that this is possible due to \eqref{cond-comp}, since
\begin{align*}
	|Dg(0,t)| + \lambda
	\ge
	\lambda
	\ge
	\boldsymbol G(\epsilon)
	=
    \frac{2^{11}\,\psi_\beta}{\epsilon}\Big[\|\omega\|_\infty K_{\epsilon} +
    \|Dg\|_{L^\infty(Q_R^+)}\Big]
    \ge
    \frac{8\|\omega\|_\infty K_{\epsilon}}{\epsilon}.
\end{align*}
The application of the lemma yields
\begin{align*}
    \mbox{II}
    &\le
    c\,\epsilon \mint_{Q}
    \big(|Du| +\lambda\big)
    \big(1+|Dg(0,t)|^2 + |Du|^2\big)^{\frac{p-2}{2}}
    |Du-Dv| \dz \\
    &\le
    c(p)\,\epsilon \mint_{Q}
    (|Du| +\lambda)^{p-1}
    |Du-Dv| \dz ,
\end{align*}
where we used \eqref{cond-comp:d} and the fact that $\lambda\ge 1$ in the last line.
Together with and Young's inequality and \eqref{cond-comp} we further estimate
\begin{align*}
    \mbox{II}
    &\le
    \epsilon\delta \mint_{Q} |Du-Dv|^p \dz +
    c\,\epsilon
    \mint_{Q}(|Du| +\lambda)^p\dz \nn\\
    &\le
    \delta \mint_{Q} |Du-Dv|^p \dz +
    c(p,1/\delta)\, \epsilon\lambda^p .
\end{align*}
For III we apply Lemma \ref{lem:diff-2} with $(G,\xi,\xi_o,\zeta)$ replaced by $(Dg(\cdot,t),Du,Du,Du-Dv)$ (recall that $\Psi(0)=\mathbb I_{n\times n}$ from \eqref{Psi}) and subsequently use \eqref{cond-comp:d}, the fact that $\lambda\ge 1$, Young's inequality and \eqref{cond-comp}. In this way we obtain
\begin{align*}
    \mbox{III}
    &\le
    c\,R^\beta
    \mint_{Q}
    (\lambda+|Du|)^{p-1} |Du-Dv|\dz \nn\\
    &\le
    R^\beta\delta \mint_{Q} |Du-Dv|^p \dz +
    c\,R^{\beta}
    \mint_{Q}(\lambda+|Du|)^p\dz \\
    &\le
    \delta \mint_{Q} |Du-Dv|^p \dz +
    c(p,\K,\psi_\beta,a_o,\|\omega\|_\infty,\mathcal G_\beta,1/\delta)\,R^{\beta}
    \lambda^{p}  .
\end{align*}
Finally, we use H\"older's and Poincar\'e's inequality, the second last inequality in \eqref{morrey-est}, Young's inequality and that $\lambda\ge1$ to estimate IV as follows:
\begin{align*}
    \mbox{IV}
    &\le
    \bigg(\mint_{Q} |v-u|^p \dz\bigg)^{\frac1p}
    \bigg(\mint_{Q} |g_t|^{p'} \dz \bigg)^{\frac{1}{p'}} \\
    &\le
    c\,\rho\bigg(\mint_{Q} |Dv-Du|^p \dz\bigg)^{\frac1p}
    \mathcal G_\beta^{\frac{1}{p'}} \rho^{\beta-1} \lambda^{\frac{p-2}{p'}}  \\
    &\le
    \delta\mint_{Q} |Dv-Du|^p \dz +
    c(n,p,\mathcal G_\beta,1/\delta)\,R^{\beta} \lambda^p .
\end{align*}
Joining the preceding estimates for I -- IV with \eqref{comparison-start}, applying Lemma \ref{lem:monotone} (i) and choosing $\delta=\frac{\K}{8c(p)}$ to absorb the terms involving $Dv$ from the right-hand side into the left
we arrive at the following \textit{comparison estimate}:
\begin{align}\label{comparison}
    \mint_{Q_{\rho,\lambda}^+(z_o)}
    |Du-Dv|^p \dz
    &\le
    c_{\rm comp}\, \big[\epsilon + R^{\beta}\big]
    \lambda^{p} ,
\end{align}
where $c_{\rm comp}=c_{\rm comp}(n,p,\K,\psi_\beta,a_o,\|\omega\|_\infty,\mathcal G_\beta)$.

\subsubsection{A decay estimate}
Here, we use the comparison estimate from above to deduce a decay estimate for the degenerate regime.

\begin{lemma}\label{lem:decay-p-cal}
Let $\chi\in(0,1]$. There exist constants 
$R_2=R_2(n,N,p,\K,\psi_\beta,\beta,a_o,$ $\|\omega\|_\infty,\mathcal G_\beta,\chi)\in(0,1]$,
$C_d=C_d(n,N,p,\K)\ge 1$, 
$\vartheta=\vartheta(n,N,p,\K,\chi)\in (0,1/4]$ and
$\epsilon=\epsilon(n,$ $N,p,\K,\psi_\beta,a_o,\|\omega\|_\infty,\mathcal G_\beta,\chi)\in(0,1]$
such that the following is true: Suppose that the assumptions of Proposition \ref{prop:main-lip} are satisfied with some $R\in(0,R_2]$ and that 
$Q_{\rho,\lambda}(z_o)\subset Q_R$ is a cylinder with $(x_o)_n\ge 0$ and $\lambda\ge 1$ satisfying
\begin{equation}\label{cond-dec-p-1}
    \boldsymbol G(\epsilon)^p
    \le
    \mint_{Q_{\rho,\lambda}^+(z_o)} |Du|^p\dz
    \le
    \lambda^p,
\end{equation}
then there exists $\lambda_1\in[\boldsymbol G(\epsilon),C_d\lambda]$ such that $Q_{\vartheta\rho,\lambda_1}(z_o)
\subset Q_{\rho,\lambda}(z_o)$ and
\begin{equation}\label{bound-deg}
    \mint_{Q_{\vartheta\rho,\lambda_1}^+(z_o)} |Du|^p\dz
    \le
    \lambda_1^p
\end{equation}
holds. Moreover, if one of the conditions
\begin{equation}\label{cond-dec-p-3}
    |D\ell_{z_o;\vartheta\rho,\lambda_1}|
    \le
    \tfrac{\lambda_1}{64}
    \quad\text{ or }\quad
    \chi^p |D\ell_{z_o;\vartheta\rho,\lambda_1}|^p
    \le
    \Phi_{\lambda_1}\big(z_o,\vartheta\rho, D\ell_{z_o;\vartheta\rho,\lambda_1}\big)
\end{equation}
is satisfied, then there holds
\begin{equation}\label{decay-deg}
    \lambda_1
    \le
    \lambda .
\end{equation}
\end{lemma}

\begin{proof}
We let $\vartheta\in(0,1/4]$ to be fixed later and choose $\epsilon,R_2\in(0,1]$ such that
\begin{equation}\label{eps-rho}
    \epsilon + R_2^{\beta}
    \le
    \frac{\vartheta^{n+4+p}}{c_{\rm comp}},
\end{equation}
where $c_{\rm comp}$ denotes the constant from the comparison estimate \eqref{comparison} depending on $n,p,$ $\K,\psi_\beta,a_o,\|\omega\|_\infty,\mathcal G_\beta$.
Next, we denote by
$$
	v\in L^p\big(\Lambda_{\rho,\lambda}(t_o);W^{1,p}(B_\rho^+(x_o),\R^N)\big)
$$ 
the unique solution of the comparison problem \eqref{system-comp} on $Q_{\rho,\lambda}^+(z_o)$. Due to \eqref{cond-dec-p-1} we know that the comparison estimate \eqref{comparison} holds true, that is
\begin{align}\label{comparison-}
    \mint_{Q_{\rho,\lambda}^+(z_o)}
    |Du-Dv|^p \dz
    \le
    c_{\rm comp}\, \big[\epsilon + R^{\beta}\big]
    \lambda^p
    \le
    \vartheta^{n+4+p}\,
    \lambda^p .
\end{align}
Combining this with assumption \eqref{bound-deg} we obtain
\begin{align*}
    \mint_{Q_{\rho,\lambda}^+(z_o)}
    |Dv|^p \dz
    \le
    2^{p-1}\bigg[\int_{Q_{\rho,\lambda}^+(z_o)}
    |Du-Dv|^p \dz +
    \int_{Q_{\rho,\lambda}^+(z_o)} |Du|^p \dz\bigg]
    \le
    2^p \lambda^p ,
\end{align*}
and therefore we are in the position to apply Lemma \ref{lem:DiBe} to the function $v$.
In order to take advantage of the lower bound \eqref{DiBe-mean} we need to distinguish the cases $(x_o)_n\ge\vartheta\rho$ and $(x_o)_n<\vartheta\rho$, where $\vartheta\in(0,1/4]$ is a parameter that will be fixed at the end of the proof. 
More precisely, with the abbreviation 
\begin{equation}\label{def-frac-z}
	\mathfrak z_o:=
	\begin{cases}
		z_o
		&\mbox{if $(x_o)_n\ge\vartheta\rho$}\\[3pt]
		z_o':=((x_o)_1,\dots,(x_o)_{n-1},0,t_o)
		&\mbox{if $(x_o)_n<\vartheta\rho$}
	\end{cases}
\end{equation}
we have from the last estimate
\begin{align*}
    \mint_{Q_{\rho/2,\lambda}^+(\mathfrak z_o)}
    |Dv|^p \dz
    \le
    2^{n+p+3} \lambda^p .
\end{align*}
Therefore, we may apply Lemma \ref{lem:DiBe} to the function $v$ with the constant $c_\ast\equiv 2^{n+p+3}$, $\mathfrak z_o$ instead of $z_o$ and the choice $r\equiv \widetilde\vartheta\rho$, where 
\begin{equation}\label{def-tilde-vartheta}
	\widetilde\vartheta:=
	\begin{cases}
		\vartheta
		&\mbox{if $(x_o)_n\ge\vartheta\rho$}\\[3pt]
		2\vartheta
		&\mbox{if $(x_o)_n<\vartheta\rho$.}
	\end{cases}
\end{equation}
The application of the lemma yields $\alpha_o=\alpha_o(n,p,\K)\in (0,1)$, $\mu_o=\mu_o(n,N,p,\K)\ge 1$ and a parameter $\mu$
satisfying
\begin{equation}\label{mu}
    \mu_o \bigg(\frac{2\max\{\widetilde\vartheta\rho,\rho_s\}}{\rho}\bigg)^{\alpha_o}
    \le
    \mu
    \le
    2\mu_o \bigg(\frac{2\max\{\widetilde\vartheta\rho,\rho_s\}}{\rho}\bigg)^{\alpha_o}, 
\end{equation}
where $\rho_s\in[0,\rho/2]$, such that $Q_{\widetilde\vartheta\rho,\lambda\mu}(\mathfrak z_o) \subset Q_{\rho/2,\lambda}(\mathfrak z_o)\subset Q_{\rho,\lambda}(z_o)$ and
\begin{equation}\label{w-sup}
    \sup_{Q_{\widetilde\vartheta\rho,\lambda\mu}^+(\mathfrak z_o)} |Dv|
    \le
    \lambda\mu
\end{equation}
and
\begin{equation}\label{w-apriori}
    \mint_{Q_{\widetilde\vartheta\rho,\lambda\mu}^+(\mathfrak z_o)}
    \big|Dv-D\ell^{(v)}_{\mathfrak z_o;\widetilde\vartheta\rho,\lambda\mu}\big|^s\dz
    \le
    c~ \lambda^s\mu^s\min\Big\{1,\frac{\vartheta\rho}{\rho_s}\Big\}^{2\alpha_o}\, \quad\mbox{for $s\in[2,p]$}
\end{equation}
holds, where $c$ depends on $n,N,p,\K$.
Note that in the case $\rho_s=0$ one has to interpret $\widetilde\vartheta\rho/\rho_s=\infty$ in \eqref{w-apriori}.
Finally, if $\widetilde\vartheta\rho<\rho_s$ we additionally have 
\begin{equation}\label{w-mean}
    \big|D\ell^{(v)}_{\mathfrak z_o;\widetilde\vartheta\rho,\lambda\mu}\big|
    \ge
    \frac{\lambda\mu}{16}\, .
\end{equation}

We now set 
\begin{equation*}
	\lambda_1:=\max\{2\lambda\mu, \boldsymbol G(\epsilon)\}
\end{equation*} 
and note that $\lambda_1\le 4\mu_o\lambda=: C_d\lambda$ and moreover 
$$
	Q_{\vartheta\rho,\lambda_1}(z_o) 
	\subset
	Q_{\widetilde\vartheta\rho,\lambda_1}(\mathfrak z_o) 
	\subset
	Q_{\widetilde\vartheta\rho,\lambda\mu}(\mathfrak z_o) 
	\subset 
	Q_{\rho,\lambda}(z_o).
$$
Furthermore, from \eqref{comparison-} and \eqref{w-sup} and the fact that $\vartheta\le1/4$ we get
\begin{align*}
    \mint_{Q_{\vartheta\rho,\lambda_1}^+(z_o)}
    |Du|^p \dz
    & \le
    2^{p-1}\bigg[\vartheta^{-(n+2)}\Big(\frac{\lambda_1}{\lambda}\Big)^{p-2} 
    \mint_{Q_{\rho,\lambda}^+(z_o)}\!
    |Du-Dv|^p \dz +
    \sup_{Q_{\widetilde\vartheta\rho,\lambda\mu}^+(\mathfrak z_o)}\!\! |Dv|^p \bigg] \\
    &\le
    2^{p-1} \bigg[ \vartheta^{2+p}\,\Big(\frac{\lambda_1}{\lambda}\Big)^{p-2} \lambda^p +
    (\lambda\mu)^p \bigg] 
    \le
    2^{p-1} \bigg[ \frac{\vartheta^{2}}{4^p}\cdot\frac{\lambda_1^p}{2^{2}\mu^{2}} +
    \frac{\lambda_1^p}{2^{p}} \bigg]
    \le
    \lambda_1^p ,
\end{align*}
where in the last step we used the fact that $\mu\ge \vartheta^{\alpha_o}\ge\vartheta$ which is a consequence of \eqref{mu}. This proves the first assertion of the lemma, i.e. \eqref{bound-deg}.

It remains to prove the second part of the lemma, concerning the bound \eqref{decay-deg} for $\lambda_1$. From now on we can additionally assume that condition \eqref{cond-dec-p-3} is in force.
Moreover, it is enough to consider the case where $\lambda_1=2\lambda\mu$, since otherwise we trivially have 
$\lambda_1=\boldsymbol G(\epsilon)\le\lambda$ by hypothesis \eqref{cond-dec-p-1}.
In order to prove \eqref{decay-deg} for $\lambda_1=2\lambda\mu$ we shall first derive a bound of the form $\rho_s\le c\,\vartheta\rho$ with some constant $c\ge 1$. For this aim we may assume that $\rho_s>\widetilde\vartheta\rho$ since otherwise the bound is satisfied with $c=2$. But this ensures the validity of \eqref{w-mean}, so that we have
\begin{align}\label{mean-lb}
    &|D\ell_{z_o;\vartheta\rho,\lambda_1}| \\ \nn
    &\phantom{mm}\ge
    |D\ell^{(v)}_{\mathfrak z_o;\widetilde\vartheta\rho,\lambda\mu}| -
    |D\ell^{(v)}_{\mathfrak z_o;\widetilde\vartheta\rho,\lambda\mu} - 
    D\ell_{\mathfrak z_o;\widetilde\vartheta\rho,\lambda\mu}| -
    |D\ell_{\mathfrak z_o;\widetilde\vartheta\rho,\lambda\mu} - 
    D\ell_{z_o;\vartheta\rho,\lambda_1}| \\ \nn
    &\phantom{mm}\ge
    \frac{\lambda\mu}{16} -
    \mint_{Q_{\widetilde\vartheta\rho,\lambda\mu}^+(\mathfrak z_o)}
    |Du-Dv|\dz  -
    \mint_{Q_{\vartheta\rho,\lambda_1}^+(z_o)} 
    |Du - D\ell_{\mathfrak z_o;\widetilde\vartheta\rho,\lambda\mu}| \dz.
\end{align}
From \eqref{comparison-} and the fact that $\vartheta\le \vartheta^{\alpha_o}\le\mu$ we infer
\begin{align}\label{comparison-Qnew}
    \mint_{Q_{\widetilde\vartheta\rho,\lambda\mu}^+(\mathfrak z_o)} |Du-Dv|^p\dz
    &\le
    \vartheta^{-(n+2)} \mu^{p-2} \mint_{Q_{\rho,\lambda}^+(z_o)} |Du-Dv|^p\dz \\
    &\le
    \vartheta^{p+2} \mu^{p-2} \lambda^p
    \le
    \vartheta^{p} \mu^p \lambda^p . \nn
\end{align}
The last inequality together with H\"older's inequality and \eqref{w-apriori} yields for $s\in\{2,p\}$ that
\begin{align}\label{v-decay}
    &\mint_{Q_{\widetilde\vartheta\rho,\lambda\mu}^+(\mathfrak z_o)} 
    |Du-D\ell_{\mathfrak z_o;\widetilde\vartheta\rho,\lambda\mu}|^s\dz \\
    &\quad\le
    c\bigg[
    \mint_{Q_{\widetilde\vartheta\rho,\lambda\mu}^+(\mathfrak z_o)} 
    |Du-Dv|^s\dz +
    \mint_{Q_{\widetilde\vartheta\rho,\lambda\mu}^+(\mathfrak z_o)}
    |Dv-D\ell^{(v)}_{\mathfrak z_o;\widetilde\vartheta\rho,\lambda\mu}|^s\dz 
    \bigg] \nn\\
    &\quad\le
    c\bigg[
    \vartheta^s \mu^s \lambda^s + 
    \lambda^s\mu^s \Big(\frac{\vartheta\rho}{\rho_s}\Big)^{2\alpha_o}\bigg]
    \le
    c\, \lambda^s\mu^s \Big(\frac{\vartheta\rho}{\rho_s}\Big)^{2\alpha_o},
    \nn
\end{align}
where $c=c(n,N,p,\K)$. This together with H\"older's inequality immediately implies
\begin{align}\label{v-decay-}
    \mint_{Q_{\vartheta\rho,\lambda_1}^+(z_o)} 
    |Du - D\ell_{\mathfrak z_o;\widetilde\vartheta\rho,\lambda\mu}| \dz 
    &\le
    2^{n+2} 2^{p-2}\mint_{Q_{\widetilde\vartheta\rho,\lambda\mu}^+(\mathfrak z_o)}
    |Du-D\ell_{\mathfrak z_o;\widetilde\vartheta\rho,\lambda\mu}|\dz \\
    &\le
    c(n,N,p,\K)\, \lambda\mu \Big(\frac{\vartheta\rho}{\rho_s}\Big)^{\alpha_o}.
    \nn
\end{align}
Joining \eqref{comparison-Qnew} and \eqref{v-decay-} with \eqref{mean-lb} we find that
\begin{align}\label{mean-lbound}
    |D\ell_{z_o;\vartheta\rho,\lambda_1}|
    \ge
    \frac{\lambda\mu}{16} -
    \vartheta\lambda\mu  -
    c\, \lambda\mu \Big(\frac{\vartheta\rho}{\rho_s}\Big)^{\alpha_o}
    \ge
    \frac{\lambda\mu}{16} -
    c\, \lambda\mu \Big(\frac{\vartheta\rho}{\rho_s}\Big)^{\alpha_o}
\end{align}
for $c=c(n,N,p,\K)$.
At this point we take advantage of the additional assumption \eqref{cond-dec-p-3}. We first consider the case where \eqref{cond-dec-p-3}$_1$ holds. Then, from \eqref{mean-lbound} and \eqref{cond-dec-p-3}$_1$ we get
\begin{align*}
    \frac{\lambda\mu}{16} -
    c\, \lambda\mu \Big(\frac{\vartheta\rho}{\rho_s}\Big)^{\alpha_o}
    \le
    |D\ell_{z_o;\vartheta\rho,\lambda_1}|
    \le
    \frac{\lambda_1}{64}
    =
    \frac{\lambda\mu}{32},
\end{align*}
which proves that
\begin{align*}
    \rho_s
    \le
    (32 c)^{\frac{1}{\alpha_o}}\,\vartheta\rho
    =
    c(n,N,p,\K)\,\vartheta\rho .
\end{align*}
On the other hand, if \eqref{cond-dec-p-3}$_2$ holds true we first infer a bound for
$\Phi_{\lambda_1}(\vartheta\rho)$. This is a consequence of the quasi-minimality of $D\ell_{z_o;\vartheta\rho,\lambda_1}$ with respect to the mapping
$\xi\mapsto \tmint_{Q_{\vartheta\rho,\lambda_1}^+(z_o)}|Du-\xi|^s\dz$ for $s\in\{2,p\}$ and 
\eqref{bound-deg} and \eqref{v-decay} applied with $s=2$ and $s=p$:
\begin{align*}
    &\Phi_{\lambda_1}\big(z_o,\vartheta\rho, D\ell_{z_o;\vartheta\rho,\lambda_1}\big) \\
    &\phantom{m}\le
    c\mint_{Q_{\vartheta\rho,\lambda_1}^+(z_o)}
    |Du-D\ell_{\mathfrak z_o;\widetilde\vartheta\rho,\lambda\mu}|^p\dz +
    c\,\lambda_1^{p-2}
    \mint_{Q_{\vartheta\rho,\lambda_1}^+(z_o)}
    |Du-D\ell_{\mathfrak z_o;\widetilde\vartheta\rho,\lambda\mu}|^2\dz \\
    &\phantom{m}\le
    c\mint_{Q_{\widetilde\vartheta\rho,\lambda\mu}^+(\mathfrak z_o)}
    |Du-D\ell_{\mathfrak z_o;\widetilde\vartheta\rho,\lambda\mu}|^p\dz +
    c\,\lambda_1^{p-2}
    \mint_{Q_{\widetilde\vartheta\rho,\lambda\mu}^+(\mathfrak z_o)}
    |Du-D\ell_{\mathfrak z_o;\widetilde\vartheta\rho,\lambda\mu}|^2\dz \\
    &\phantom{m}\le
    c\,\lambda^p\mu^p \Big(\frac{\vartheta\rho}{\rho_s}\Big)^{2\alpha_o} +
    c\,\lambda_1^{p-2}\lambda^2\mu^2 \Big(\frac{\vartheta\rho}{\rho_s}\Big)^{2\alpha_o} 
    \le
    c\, \lambda^p\mu^p \Big(\frac{\vartheta\rho}{\rho_s}\Big)^{2\alpha_o}     ,
\end{align*}
where $c=c(n,N,p,\K)$. Combining this estimate with \eqref{mean-lbound} and \eqref{cond-dec-p-3}$_2$ we get
\begin{align*}
    \frac{\lambda\mu}{16}
    &\le
    |D\ell_{z_o;\vartheta\rho,\lambda_1}| +
    c\, \lambda\mu \Big(\frac{\vartheta\rho}{\rho_s}\Big)^{\alpha_o} \\
    &\le
    \chi^{-1} \Phi_{\lambda_1}^{\frac1p}\big(z_o,\vartheta\rho, D\ell_{z_o;\vartheta\rho,\lambda_1}\big) +
    c\, \lambda\mu \Big(\frac{\vartheta\rho}{\rho_s}\Big)^{\alpha_o}
    \le
    c\,(\chi^{-1}+1) \lambda\mu \Big(\frac{\vartheta\rho}{\rho_s}\Big)^{\frac{2\alpha_o}{p}},
\end{align*}
which implies
\begin{align*}
    \rho_s
    \le
    \big[16c\,(\chi^{-1}+1)\big]^{\frac{p}{2\alpha_o}} \vartheta\rho
    =
    c(n,N,p,\K,\chi)\,\vartheta\rho .
\end{align*}
Therefore, in any case we have
$\rho_s \le c\,\vartheta\rho$,
where $c$ depends at most on $n,N,p,\K,\chi$. By \eqref{mu} we therefore conclude
\begin{align*}
    \lambda_1
    =
    2\lambda\mu
    \le
    4\mu_o \Big(\frac{2\rho_s}{\rho}\Big)^{\alpha_o}\lambda
    \le
    4\mu_o (2c\,\vartheta)^{\alpha_o}\lambda
    =
    c(n,N,p,\K,\chi)\, \vartheta^{\alpha_o}\lambda.
\end{align*}
Now we choose $\vartheta\in(0,1/4]$ in dependence of $n,N,p,\K,\chi$ in such a way that
\begin{align*}
    c\, \vartheta^{\alpha_o}
    \le
    1.
\end{align*}
This proves \eqref{decay-deg}. Having fixed $\vartheta$ we can perform the choices of $\epsilon$ and $R_2$ according to \eqref{eps-rho}. Note that this choice amounts in the dependencies of the parameters indicated in the statement of the lemma. This finishes the proof of the lemma.
\end{proof}

\subsection{Concluding the proof}

In order to complete the proof Propositon \ref{prop:main-lip} we now proceed in two steps. We first show that the spatial gradient belongs to $L^\infty$ and subsequently use this fact together with the parabolic Poincar\'e inequality to conclude that $u$ is Lipschitz continuous with respect to the parabolic metric.

\subsubsection{Boundedness of the spatial gradient}
We first fix the constants from Proposition \ref{prop:iter-a-cal} and Lemma \ref{lem:decay-p-cal}.
By 
$\theta\in(0,8^{-p/\beta}]$, $R_1\in(0,1]$, $\epsilon_1\in(0,1]$, $\epsilon_2\in(0,\theta^{n+2}]$ and $\bar c\ge 1$ we denote the constants from Proposition \ref{prop:iter-a-cal} depending on 
$n,N,p, \K, \psi_\beta, \beta,a_o, \|\omega\|_\infty, \mathcal G_\beta$.
Next, we fix the constants of Lemma \ref{lem:decay-p-cal} for the choice $\chi=\epsilon_2$. This yields 
$C_d=C_d(n,N,p,\K)\ge 1$ and 
$R_2\in(0,1]$,
$\vartheta\in (0,1/4]$, 
$\epsilon\in(0,1]$ 
depending at most on 
$n,N,p, \K, \beta,\psi_\beta, a_o, \|\omega\|_\infty, \mathcal G_\beta$. Finally, we let
\begin{equation*}
	\epsilon_o
	:=
	\min\{\epsilon,\epsilon_1\}
	\quad\mbox{and}\quad
	R_o
	:=
	\min\{R_1,R_2\}	.
\end{equation*}
Suppose now that the assumptions of Proposition \ref{prop:main-lip} are satisfied with some $R\in(0,R_o]$ and that $z_o\in Q_R^+$ is a Lebesgue point of $Du$, i.e. there holds
\begin{equation*}
    Du(z_o)
    =
    \lim_{r\downarrow 0}
    (Du)_{z_o;r}^+
    \qquad\mbox{and}\qquad
    \lim_{r\downarrow 0}
    \mint_{Q_r^+(z_o)} |Du-(Du)_{z_o;r}^+|\dz
    =
    0.
\end{equation*}
We choose $\rho\in(0,R]$ such that $Q_{\rho}(z_o)\subset Q_R$ and define
\begin{equation*}
    \lambda_o
    :=
    \bigg(\mint_{Q_{\rho}^+(z_o)}|Du|^p\dz\bigg)^\frac12 + \boldsymbol G(\epsilon_o),
\end{equation*}
where $\boldsymbol G(\epsilon_o)$ is defined in \eqref{def-Geps}.
In the following we will prove that there exists a non-increasing sequence of radii $0\le\rho_i\le\vartheta^i\rho$ and $\lambda_i>0$ such that for any $i\in\N_0$ there holds $\boldsymbol G(\epsilon_o)\le\lambda_i\le \lambda_o$ and
\begin{equation}\label{bound-i}
\left\{
\begin{array}{cl}
    |Du(z_o)|
    \le
    2C_d\,\lambda_o
    &\qquad\mbox{if $\rho_i=0$,}\\[7pt]
    \displaystyle\boldsymbol G(\epsilon_o)^p
    \le
    \mint_{Q_{\rho_i,\lambda_i}^+(z_o)}|Du|^p\dz
    \le
    \lambda_i^p
    &\qquad\mbox{if $\rho_i>0$.} 
\end{array}\right.
\end{equation}
The assertion will be proved by iteration. We start with the case $i=0$. Here, we define
$\rho_o:=\rho$ if $\tmint_{Q_{r,\lambda_o}^+(z_o)}|Du|^p\dz>\boldsymbol G(\epsilon_o)^p$ for any $r\in(0,\rho]$ and 
\begin{equation*}
    \rho_o
    :=
    \inf\bigg\{ r\in(0,\rho] :
    \mint_{Q_{r,\lambda_o}^+(z_o)}|Du|^p\dz
    \le
    \boldsymbol G(\epsilon_o)^p \bigg\}
\end{equation*}
otherwise.
We now distinguish three cases. In the case $\rho_o=0$ we use the fact that $z_o$ is a Lebesgue point of $Du$ to infer that
\begin{align*}
    |Du(z_o)|
    &=
    |\lim_{r\downarrow 0}(Du)_{z_o;r}^+|
    \le
    \limsup_{r\downarrow 0}|(Du)_{z_o;r,\lambda_o}^+| +
    \limsup_{r\downarrow 0}|(Du)_{z_o;r,\lambda_o}^+ - (Du)_{z_o;r}^+| \\
    &\le
    \limsup_{r\downarrow 0} \bigg(\mint_{Q_{r,\lambda_o}^+(z_o)}|Du|^p\dz \bigg)^\frac1p +
    \limsup_{r\downarrow 0}
    \mint_{Q_{r,\lambda_o}^+(z_o)}
    |Du - (Du)_{z_o;r}^+|\dz \\
    &\le
    \boldsymbol G(\epsilon_o) +
    \limsup_{r\downarrow 0} \lambda_o^{p-2}
    \mint_{Q_{r}^+(z_o)} |Du - (Du)_{z_o;r}^+|\dz
    =
    \boldsymbol G(\epsilon_o)
    \le\lambda_o.
\end{align*}
Setting $\rho_i=0$ and $\lambda_i=\lambda_o$ for any $i\in\N$ this proves \eqref{bound-i}$_i$ for any $i\in\N_0$ and we can stop the iteration scheme.

In the case $\rho_o=\rho$ we have
\begin{align*}
    \boldsymbol G(\epsilon_o)^p
    \le
    \mint_{Q_{\rho_o,\lambda_o}^+(z_o)}|Du|^p\dz
    =
    \mint_{Q_{\rho,\lambda_o}^+(z_o)}|Du|^p\dz
    \le
    \lambda_o^{p-2}
    \mint_{Q_{\rho}^+(z_o)}|Du|^p\dz
    \le
    \lambda_o^p,
\end{align*}
which proves \eqref{bound-i}$_0$.
Finally, in the case $0<\rho_o<\rho$ there holds
\begin{align*}
    \mint_{Q_{\rho_o,\lambda_o}^+(z_o)}|Du|^p\dz
    =
    \boldsymbol G(\epsilon_o)^p
    \le
    \lambda_o^p,
\end{align*}
which again shows \eqref{bound-i}$_0$. This finishes the proof for $i=0$.

We now suppose that $i\ge 0$ and that $\rho_j, \lambda_j$ have already been constructed for $j=0,1,\dots,i$ in such a way that \eqref{bound-i} holds. In the case that $\rho_i=0$ there is nothing to do, since then also $\rho_k=0$ and \eqref{bound-i}$_k$ holds for any $k\ge i$. Therefore, it is enough to consider the case where $\rho_i>0$ and the second case of \eqref{bound-i}$_i$ holds.
By \eqref{bound-i}$_i$ we can apply Lemma \ref{lem:decay-p-cal} with $(\rho_i,\lambda_i)$ instead of $(\rho,\lambda)$ to infer the existence of $\lambda_{i+1}\in[\boldsymbol G(\epsilon_o), C_d\lambda_i]$ such that $Q_{\vartheta\rho_i,\lambda_{i+1}}(z_o)\subset Q_{\rho_i,\lambda_i}(z_o)$ and
\begin{align}\label{bound-i+1}
    \mint_{Q_{\vartheta\rho_i,\lambda_{i+1}}^+(z_o)} |Du|^p\dz
    \le
    \lambda_{i+1}^{p}
\end{align}
holds. We now distinguish whether we are in the non-degenerate or in the degenerate regime at level $\vartheta\rho_i$. If we are in the non-degenerate regime, i.e. if
\begin{equation}\label{non-deg-i+1}
    \tfrac{\lambda_{i+1}}{64}
    \le
    |D\ell_{z_o;\vartheta\rho_i, \lambda_{i+1}}|
    \quad\!\!\text{and}\quad\!\!\!
    \Phi_{\lambda_{i+1}}\big(z_o,\vartheta\rho_i,D\ell_{z_o;\vartheta\rho_i,\lambda_{i+1}}\big)
    \le
    \epsilon_2^{p} |D\ell_{z_o;\vartheta\rho_i, \lambda_{i+1}}|^p
\end{equation}
holds we can apply \eqref{mod-gamma} from Proposition \ref{prop:iter-a-cal}. Together with the definition of $D\ell_{z_o;r}$ and the fact that $z_o$ is a Lebesgue point this yields
\begin{equation*}
    |Du(z_o)|
    \le
    2\lambda_{i+1}
    \le
    2 C_d\lambda_i
    \le
    2 C_d\lambda_o.
\end{equation*}
Setting $\rho_k=0$ and $\lambda_k=\lambda_i$ for any $k> i$ (note that here we possibly redefine $\lambda_{i+1}$) this proves \eqref{bound-i}$_k$ for any $k>i$. Therefore, in this case we have proved the assertion and can stop the iteration scheme.

On the other hand, if \eqref{non-deg-i+1} fails to hold we know that \eqref{cond-dec-p-3} is satisfied and therefore we can apply the second part of Lemma \ref{lem:decay-p-cal} to conclude that
\begin{equation*}
    \lambda_{i+1}
    \le
    \lambda_i
    \le
    \lambda_o.
\end{equation*}
We now define $\rho_{i+1}:=\vartheta\rho_i$ if 
$\tmint_{Q_{r,\lambda_{i+1}}^+(z_o)}|Du|^p\dz>\boldsymbol G(\epsilon_o)^p$ for any $r\in(0,\vartheta\rho_i]$ and 
\begin{equation*}
    \rho_{i+1}
    :=
    \inf\bigg\{ r\in(0,\vartheta\rho_i] :
    \mint_{Q_{r,\lambda_{i+1}}^+(z_o)}|Du|^p\dz
    \le
    \boldsymbol G(\epsilon_o)^p
    \bigg\}
\end{equation*}
otherwise. In the case $\rho_{i+1}=0$ we use that $z_o$ is a Lebesgue point of $Du$ to deduce that
\begin{align*}
    &|Du(z_o)|
    =
    |\lim_{r\downarrow 0}(Du)_{z_o;r}|\\
    &\quad\le
    \limsup_{r\downarrow 0}|(Du)_{z_o;r,\lambda_{i+1}}^+| +
    \limsup_{r\downarrow 0}|(Du)_{z_o;r,\lambda_{i+1}}^+ - (Du)_{z_o;r}^+| \\
    &\quad\le
    \limsup_{r\downarrow 0} \bigg(\mint_{Q_{r,\lambda_{i+1}}^+(z_o)}|Du|^p\dz \bigg)^\frac1p +
    \limsup_{r\downarrow 0}
    \mint_{Q_{r,\lambda_{i+1}}^+(z_o)}
    |Du - (Du)_{z_o;r}^+|\dz \\
    &\quad\le
    \boldsymbol G(\epsilon_o) +
    \limsup_{r\downarrow 0} \lambda_{i+1}^{p-2}
    \mint_{Q_{r}^+(z_o)} |Du - (Du)_{z_o;r}^+|\dz 
    =
    \boldsymbol G(\epsilon_o)
    \le
    \lambda_o.
\end{align*}
Setting $\rho_k=0$ and $\lambda_k=\lambda_i$ for any $k> i+1$ this proves \eqref{bound-i}$_k$ for any $k>i$. Therefore, in this case we have proved the assertion and can stop the iteration scheme.

In the case $\rho_{i+1}=\vartheta\rho_i$ we have
\begin{align*}
    \boldsymbol G(\epsilon_o)^p
    \le
    \mint_{Q_{\rho_{i+1},\lambda_{i+1}}^+(z_o)}|Du|^p\dz
    =
    \mint_{Q_{\vartheta\rho_i,\lambda_{i+1}}^+(z_o)}|Du|^p\dz
    \le
    \lambda_{i+1}^p
\end{align*}
by \eqref{bound-i+1}, while in the case $0<\rho_{i+1}<\vartheta\rho_i$ there holds
\begin{align*}
    \mint_{Q_{\rho_{i+1},\lambda_{i+1}}^+(z_o)}|Du|^p\dz
    =
    \boldsymbol G(\epsilon_o)^p
    \le
    \lambda_{i+1}^p.
\end{align*}
Hence, in both cases we have shown that \eqref{bound-i}$_{i+1}$ holds. This finishes the proof of \eqref{bound-i}.

We now come to the final proof of the bound for $|Du(z_o)|$. From \eqref{bound-i} we conclude that either 
\begin{equation*}
    |Du(z_o)|
    \le
    2C_d\, \lambda_o,
\end{equation*}
or the second case of \eqref{bound-i} holds for all $i\in\N_0$. In the latter case we use the fact that $z_o$ is a Lebesgue point of $Du$ to infer
\begin{align*}
    |Du(z_o)|
    &=
    |\lim_{i\to\infty}(Du)_{z_o;\rho_i}^+| \\
    &\le
    \limsup_{i\to\infty}|(Du)_{z_o;\rho_i,\lambda_{i}}^+| +
    \limsup_{i\to\infty}|(Du)_{z_o;\rho_i,\lambda_{i}}^+ - (Du)_{z_o;\rho_i}^+| \\
    &\le
    \limsup_{i\to\infty} \bigg(\mint_{Q_{\rho_i,\lambda_{i}}^+(z_o)}|Du|^p\dz \bigg)^\frac1p +
    \limsup_{i\to\infty}
    \mint_{Q_{\rho_i,\lambda_{i}}^+(z_o)}
    |Du - (Du)_{z_o;\rho_i}^+|\dz \\
    &\le
    \limsup_{i\to\infty} \lambda_i +
    \limsup_{i\to\infty} \lambda_{i}^{p-2}
    \mint_{Q_{\rho_i}^+(z_o)} |Du - (Du)_{z_o;\rho_i}^+|\dz 
    \le
    \lambda_o.
\end{align*}
Therefore, recalling the definition of $\lambda_o$ we deduce that in any case there holds 
\begin{equation*}
    |Du(z_o)|
    \le
    2C_d\, \lambda_o
    =
    c(n,N,p,\K)
    \bigg[\bigg(\mint_{Q_{\rho}^+(z_o)}|Du|^p\dz\bigg)^\frac12 + 
    \boldsymbol G(\epsilon_o)\bigg].
\end{equation*}
Finally, inserting the definition of $\boldsymbol G(\epsilon_o)$ from \eqref{def-Geps} we get 
\begin{equation}\label{grad-bound}
    |Du(z_o)|
    \le
    c_1
    \bigg(\mint_{Q_{\rho}^+(z_o)}|Du|^p\dz\bigg)^\frac12 + 
    c_2,
\end{equation}
where $c_1=c_1(n,N,p,\K)$ and 
$c_2=c_2(n,N,p, \K, \psi_\beta, \beta,a_o, \omega(\cdot), \mathcal G_\beta)$.
This proves the statement concerning $Du$ in Proposition \ref{prop:main-lip}.

\subsubsection{Lipschitz continuity of the solution}\label{sec:lipschitz}
It now remains to prove the Lipschitz continuity of $u$.
As already mentioned above, this is an immediate consequence of the gradient bound and the parabolic Poincar\'e inequality from Lemma
\ref{lem:poin-inter}. We let $\epsilon\in(0,R)$.
For $z_o\in Q_{R-\epsilon/2}^+$ we infer from \eqref{grad-bound} that 
\begin{equation*}
    |Du(z_o)|
    \le
    c_1
    \bigg(\mint_{Q_{\epsilon/2}^+(z_o)}|Du|^p\dz\bigg)^\frac12 + 
    c_2
    \le
    c\,\Big(\frac{R}{\epsilon}\Big)^{\frac{n+2}{2}}
    \bigg(\mint_{Q_R^+}\big(|Du|^p + 1\big)\dz\bigg)^\frac12 .
\end{equation*}
We now consider $\mathfrak z_o=(\mathfrak x_o, \mathfrak t_o)\in Q_{1-\epsilon}^+$ and $r\in (0,\epsilon/8]$.
In the case that $(\mathfrak x_o)_n\ge r$ and hence $Q_{r}(\mathfrak z_o)\subset Q_R^+$
we obtain from Poincar\'e's inequality, i.e. Lemma \ref{lem:poin-inter} applied with $\lambda=1$ and $A=0$ and the preceding estimate (note that $Q_{r}^+(\mathfrak z_o)\subset Q_{R-\epsilon/2}^+$) that
\begin{align*}
    r^{-2}\mint_{Q_{r}(\mathfrak z_o)} 
    |u - (u)_{\mathfrak z_o;r}|^2\dz 
    &\le
    c \mint_{Q_{r}(\mathfrak z_o)} |Du|^2 dz +
    c  \bigg[
    \mint_{Q_{r}(\mathfrak z_o)} |Du|^{p-1} dz +
    r^{\beta} \bigg]^2 \\
    &\le
    c\, \bigg(\mint_{Q_R^+}\big(|Du|^p + 1\big)\dz\bigg)^{p-1}
\end{align*}
for a constant 
$c=c(n,N,p, \K, \psi_\beta, \beta,a_o, \omega(\cdot), \mathcal G_\beta,\epsilon)$.
In the case that $(\mathfrak x_o)_n< r$ we apply the boundary Poincar\'e inequality from Lemma \ref{lem:poin-boundary} to infer that

\begin{align*}
    r^{-2}\mint_{Q_{r}^+(\mathfrak z_o)} 
    |u - (u)_{\mathfrak z_o;r}^+|^2\dz 
    &\le
    r^{-2}\mint_{Q_{r}^+(\mathfrak z_o)} |u|^2\dz 
    \le
    2^{n+2}r^{-2}\mint_{Q_{2r}^+(\mathfrak z_o')} |u|^2\dz \\
    &\le
    c \mint_{Q_{2r}^+(\mathfrak z_o')} |Du|^2 dz 
    \le
    c\, \mint_{Q_R^+}\big(|Du|^p + 1\big)\dz ,
\end{align*}
where by $\mathfrak z_o':=((\mathfrak x_o)_1\dots,(\mathfrak x_o)_{n-1},0,\mathfrak t_o)$ we denote the orthogonal projection of $\mathfrak z_o$ on $\Gamma$.
By the characterization of H\"older continuity of Campanato-Da Prato \cite{DaPrato:1965} it follows that $u$
is $C^{0;1, 1/2}$-continuous
in $Q_{R-\epsilon}^+$. This finishes the proof of Proposition \ref{prop:main-lip}.  
\hfill$\Box$

\bibliographystyle{plain}

\end{document}